\documentclass[12pt]{article}
\usepackage{defs1}
\usepackage{amsmath}                    % Klammern, texte in Formeln
\usepackage{amssymb}                                    % \rightrightarrows
\usepackage{amsfonts}

\usepackage[normalem]{ulem}

\title{Twisted differential cohomology}
\author{Ulrich Bunke\thanks{NWF I - Mathematik,
Universit{\"a}t Regensburg 
93040 Regensburg,
{Germany}, ulrich.bunke@mathematik.uni-regensburg.de}  and  Thomas Nikolaus\thanks{NWF I - Mathematik,
Universit{\"a}t Regensburg 
93040 Regensburg,
{Germany}, thomas.nikolaus@mathematik.uni-regensburg.de}
}

\theoremstyle{plain}

\newtheorem{theorem}{Theorem}[section] 
\newtheorem{prop}[theorem]{Proposition}
\newtheorem{lem}[theorem]{Lemma}
\newtheorem{ddd}[theorem]{Definition}
\newtheorem{kor}[theorem]{Corollary}

\newtheorem{thm}[theorem]{Theorem}
\newtheorem{corollary}[theorem]{Corollary}
\newtheorem{proposition}[theorem]{Proposition}
\newtheorem{lemma}[theorem]{Lemma}
\newtheorem{example}[theorem]{Example}

\newtheorem{definition}[theorem]{Definition}

\theoremstyle{definition}
\newtheorem{rem2}[theorem]{Remark}
\newtheorem{exa2}[theorem]{Example}

\newcommand{\uli}[1]{\textcolor{red}{#1}}

\newcommand {\Tw}{\mathcal{T}\hspace{-1pt}\mathrm{w}}

\renewcommand{\curv}{\mathcal{C}\mathrm{urv}}
\newcommand{\bZ}{\mathbf{Z}}
 
\newcommand{\KU}{\mathrm{K}}
\newcommand{\K}{\mathrm{K}}
\newcommand{\KO}{\mathbf{KO}}
\newcommand{\BGL}{\mathrm{BGL}}
\newcommand{\gl}{\mathrm{gl}}
\newcommand{\bgl}{\mathrm{bgl}}
 \newcommand{\pic}{\mathrm{pic}}

\newcommand{\CAlg}{\mathrm{CAlg}}
\newcommand{\cN}{{\mathcal{N}}}
\newcommand{\cR}{{\mathcal{R}}}

\newcommand{\Rings}{{\mathrm{Rings}}}

\newcommand{\cZ}{{\mathcal{Z}}}

\newcommand{\PSh}{{\mathrm{PSh}}}

 \newcommand{\Cat}{{\mathcal{C}\mathrm{at}}}

\newcommand{\bKU}{\mathbf{KU}}
\newcommand{\PicLf}{\Pic^\textnormal{loc,fl}}

\newcommand{\Picf}{\Pic^\textnormal{wh,fl}}
\newcommand{\Picff}{\Pic^\textnormal{loc,fl}}

\newcommand{\Picfff}{ \Pic^{\textnormal{fl}}}
\renewcommand{\Mf}{\mathcal{M}\mathrm{f}}

\renewcommand{\Im}{\operatorname{Im}}

\newcommand{\GL}{\mathrm{GL}}

\newcommand{\Sing}{\Pi_\infty}%\textnormal{Sing}}
\newcommand{\iso}{\stackrel{\sim}{\to}}
\newcommand{\GCS}{\mathrm{Gr}\mathcal{R}\mathrm{ing}\Sp}

\renewcommand{\Sp}{\mathcal{S}\mathrm{p}}

\renewcommand{\Fun}{\mathrm{Fun}}

\renewcommand{\loc}{\mathrm{loc}}
\newcommand {\calc}{\mathcal{C}}
\newcommand {\cald}{\mathcal{D}}

\newcommand {\calf}{\mathcal{F}}

\newcommand {\ho}{\mathcal{H}\mathrm{o}}

\renewcommand{\Mod}{\mathrm{Mod}}
\newcommand{\Ho}{\mathcal{H}\text{o}}
\newcommand{\AbGr}{\mathrm{Ab}}
\renewcommand{\Ab}{\AbGr}

\renewcommand{\Ch}{\mathrm{Ch}}
\newcommand{\PicI}{\Pic^{h}}
\newcommand{\PicL}{\Pic^{\loc}}
\newcommand{\PicwLf}{\Pic^{\mathrm{wloc,fl}}}

\renewcommand{\Sh}{\mathrm{Sh}}
\newcommand{\grRing}{\mathrm{Gr}\mathcal{R}\mathrm{ing}}
\newcommand{\SymMon}{\mathcal{S}\mathrm{ym}\mathcal{M}\mathrm{on}_\infty}
 \newcommand{\symstr}{\mathcal{S}\mathrm{ym}\mathcal{M}\mathrm{on}}
 \newcommand{\CGrp}{\mathrm{CGrp}}
  \newcommand{\CMon}{\mathrm{CMon}}

\newcommand{\one}{\langle 1\rangle}
\newcommand{\zero}{\langle 0 \rangle}

\newcommand{\tho}[1]{\textcolor{blue}{#1}}
\newcommand{\tfoot}[1]{\tho{\footnote{\tho{#1}}}}

\begin{document}
\maketitle

\tableofcontents

\section{Introduction}

\subsection{On the contents of the paper}

The main goal of the present paper is the construction of twisted generalized differential cohomology theories and the comprehensive statement of its basic functorial properties. Technically it combines the homotopy theoretic approach to (untwisted) generalized differential cohomology developed in \cite{MR2192936}, \cite{bg}, \cite{skript} with the $\infty$-categorical treatement of twisted cohomology in \cite{2011arXiv1112.2203A}.  {One of the main motivations for twisted differential cohomology is the theory of integration (resp. pushforward) in non orientable situations. Our approach leads to such a theory, which will be discussed elsewhere.}

\bigskip

Following an idea of the second author a neat way to state  all the  properties of a twisted generalized differential cohomology is to say that it forms a sheaf of graded {ring}  spectra on the category of
smooth manifolds where the underlying sheaf of gradings is the Picard stack of differential twists.
This formulation incorporates the multiplicative structure of twisted cohomology (which adds twists), automorphisms of twists, Mayer-Vietoris sequences as well as the functoriality in the underlying manifold.

Given the approach of \cite{bg} the construction of twisted generalized differential cohomology 
is actually easy once the differential twists are understood. So the main emphasis of the present paper is on the investigation of the Picard groupoid of differential twists.  Our main results concern the existence of differential refinements of topological twists and the classification of such refinements. 
As an interesting aspect of the theory we show that for a commutative ring spectrum $R$
there is a map (anticipated e.g. in \cite{MR2742428}) from a version of differential $\bgl_{1}(R)$-cohomology to differential $R$-twists. 
 
 \bigskip

We now explain the structure of the paper in detail. It consists of three parts and a technical appendix. In the first part we develop twisted generalized differential cohomology in general. In the second part we  consider in detail the differential form part of differential twists and discuss the existence and uniqueness of differential refinements of topological twists. The third part contains more concrete material and examples.

In the first part we develop the general theory of twisted generalized differential cohomology starting with the notion of a graded ring spectrum %\tfoot{Sollen wir eigentlich irgendwo sagen, dass man im Prinzip auch f‚Äö√†√∂¬¨‚à´r weniger als $E_\infty$-Twists machen kann, wir uns aber auf $E_\infty$-einschr‚Äö√†√∂¬¨√ünken?}
 in  Section \ref{may0501}.    In Section \ref{ulimar2201} we review the homotopy theoretic version  of twisted cohomology using the language of graded ring spectra. In Section  \ref{sec_def_twisted} we introduce the Picard-$\infty$ stack of differential twists and construct the sheaf of graded ring spectra representing twisted generalized cohomology. The first part is concluded with Section \ref{sep0602} where we list explicitly some of the functorial structures of twisted differential cohomology.

In the second part we investigate differential twists in detail. While the definition of differential twists given in Section \ref{sec_def_twisted} is straightforward it is not at all obvious whether a given topological twist has a differential refinement at all, and how many of  such refinements exist. In the preparatory Section \ref{sep0602nnnn} we introduce a particular class of Rham models which can conveniently be used in the construction of differential twists.
  In Section   \ref{may0510} we  show under mild assumptions that the extension of a topological twist to a differential twist is unique up to non-canonical equivalence. The existence of differential twists is studied in Section \ref{sec:existence}.

The third part is devoted to special cases and examples.  While the general theory works for arbitrary commutative ring spectra all   {examples of interest}  satisfy a set of  particular assumptions  termed differentially simple in Section \ref{realization}. This assumption simplifies the choice of suitable de Rham models considerably. In Section \ref{may0511} we show that cycles for differential $\bgl_{1}(R)$-cohomology determine differential $R$-twists. This provides one explicit way to construct differential twists. In this framework  we discuss, in particular, the construction of differential twists from
higher geometric gerbes. In Section \ref{may0512} we show how one can determine the de Rham
 part of the differential twist refining a topological twist $E$ from the knowledge of differentials of the Atiyah-Hirzebruch spectral sequence associated to $E$.
 
\bigskip

In the second subsection of this introduction we explain {some} new contributions of this paper 
in the example of complex $K$-theory.
\subsection{The case of complex $K$-theory}\label{may0650}

The study of field theories involving higher-degree differential forms  on manifolds with $B$-field backgrounds lead to quantization conditions which could conveniently be explained by the requirement, that the field strength represents the Chern character of a twisted $K$-theory class.
Subsequently various models of twisted $K$-theory have been constructed. It was clear from the beginning, that twisted $K$-theory fixes the topological background, and the actual fields
%\tfoot{Sind das wirklich die Felder und nicht eher die Branes? Oder denkst du an Ramond-Ramond Felder?}
 are classified by twisted differential $K$-theory classes. It turned out to be important to understand the twists themselves in the framework of differential cohomology. For  {a} detailed description of the phyiscal motivation we refer to  \cite{MR2742428}.

 Twisted complex $K$-theory and its differential version have been investigated to some {extent} 
 \cite{MR0282363}, \cite{MR2119241}, \cite{MR2307274}, \cite{MR2172633}, \cite{MR3037783}, \cite{MR2860342}, \cite{MR2831111}, \cite{2013arXiv1304.0939B}, \cite{MR1911247}, \cite{2013arXiv1303.5159G}, \cite{2012arXiv1211.6761H}, \cite{MR2291795}, \cite{2012arXiv1208.4921H} (we refrain from listing paper dealing with non-twisted differential $K$-theory).
 
 \bigskip

 In most of the literature and applications one
 restricts to the subset of topological twists which are classified  by
third cohomology. In the present paper we consider the degree as part of the twist. Therefore given a manifold $M$, an integer $n\in \Z$, and a cohomology class $z\in H^{3}(M;\Z)$ one  has  a twisted $K$-theory group $\K^{n+z}(M)$.   A classical calculation of twisted $K$-theory is  (see Example \ref{may0601})
$$\K^{1+z}(S^{3})\cong \Z/k\Z , \quad \K^{0+z}(S^{3})=0\ ,$$ if $z\in H^{3}(S^{3};\Z)$ is the $k$th multiple of the {canonical} generator and $k\not=0$.  Twisted $K$-theory has a product (which is trivial in the $S^{3}$-example)
$$\cup :\K^{n+z}(M)\otimes \K^{n^{\prime}+z^{\prime}}(M)\to \K^{n+n^{\prime}+z+z^{\prime}}(M)\ .$$ 

So far this picture of twisted $K$-theory has the draw-back that the group $\K^{n+z}(M)$ is only defined up to non-canonical isomorphism.  In order to fix this group  itself
one must {choose} a cycle for the cohomology class $z$. In general, the automorphisms of cycles act nontrivially on the twisted $K$-theory groups. A detailed picture of this {phenomenom} has been explained in an axiomatic way in \cite[3.1]{MR2130624}.

The precise meaning of a cycle for a twist depends on 
 the particular construction of twisted $K$-theory. In the literature cycles  are taken as $U(1)$-gerbes,  principal $PU$-bundle,  bundles of Fredholm operators or maps to a model of the Eilenberg-MacLane space $K(\Z,3)$. 
In the language of the present paper  such a map would be considered as an object of the Picard-$\infty$ groupoid $\underline{K(\Z,3)}(M)$. 

\bigskip
In the present paper we work with 
the Picard-$\infty$ groupoid of {\em all} topological twists of complex $K$-theory on $M$ which  is denoted by $\PicL_{\underline{\KU}}(M)$. The specialization to third cohomology is obtained via a map of Picard-$\infty$ stacks
$\underline{K(\Z,3)}\to \PicL_{\underline{\KU}}$ which is {essentially} unique  by \cite{2011arXiv1106.5099A}. 

Our way to phrase the dependence of twisted $K$-theory on the cycle, the action of the automorphism groups of cycles, and the multiplicative structure is to say that the association
$${\K^{(\cdots)}(M)} : \quad \Ho  \PicL_{\underline{\KU}}(M)\to \Ab$$
is a {lax} symmetric monoidal functor. {Here $\Ho  \PicL_{\underline{\KU}}(M)$ is the homotopy category of twists on $M$ and $\Ab$ denotes the category of abelian groups. }After specialization to $\underline{K(\Z,3)}$
this realizes the axioms listed in \cite[3.1]{MR2130624} and includes in addition the multiplicative structure.
In the language of the present paper, the functor ${\K^{(\cdots)}(M)} $ is a 
 graded ring, graded by the homotopy category  $\Ho  \PicL_{\underline{\KU}}(M)$  of topological twists of $\KU$ on $M$. 

\bigskip

The real model for the   untwisted  complex $K$-theory of a manifold $M$ is the $2$-periodic de Rham complex 
 {$\Omega^*(M)[b,b^{-1}]$}, where $\deg(b)=-2$. It receives a multiplicative Chern character 
 transformation
$$\ch:\K^{*}(M) \to {H^{*}\big(\Omega^*(M)[b,b^{-1}]\big)\ .}$$
which becomes an isomorphism after tensoring the domain with $\R$ provided $M$ is compact.
If $V\to M$ is a complex vector bundle and $\nabla$ is a connection on $V$, then
$\ch([V])$ is represented by the closed form \begin{equation}\label{may0610}\Tr\exp\left(\frac{b}{2\pi i } R^{\nabla}\right)\in Z^{0}{\big(\Omega^*(M)[b,b^{-1}]\big)}\ .\end{equation}

{We now come back to the twisted case and consider an equivalence} class of twists coming from a class $z\in H^{3}(M;\Z)$. Then one considers a closed three form $\omega\in Z^{3}\big({\Omega^*(M)}\big)$ representing the image of the class $z$ in de Rham cohomology and forms the twisted de Rham complex
$\big({\Omega^*(M)[b,b^{-1}]},d+b\omega\big)$ by perturbing the  de Rham differential with $b\omega$. The twisted de Rham complex is the correct real model of $z$-twisted $K$-theory. This is justified by the existence of a twisted version of the Chern character
\begin{equation}\label{may0520}\ch:\K^{*+z}(M) \to H^{*}\big({\Omega^*(M)[b,b^{-1}]},d+b\omega\big)\end{equation}
which again becomes an   isomorphism after tensoring the domain with $\R$ for compact  $M$.
This twisted Chern character has been constructed in various models \cite{MR1977885}, \cite{MR2200848}, \cite{MR2271013}, \cite{MR2672802} and is in general (with the exception of \cite{MR2672802}) based on an infinite-dimensional generalization of  Chern-Weyl theory and the formula \eqref{may0610}. One consequence of the present paper is a homotopy theoretic construction of the twisted Chern character.
\bigskip

Since the twisted $K$-group is determined by the class $z$ only up to non-canonical isomorphism it is impossible to say that the twisted Chern character is natural in the twist or compatible with products etc.
Improvements in this direction  have been adressed in the literature. The present paper provides a solution of the problem  as a consequence of  the construction of a differential twisted complex $K$-theory.

As a start we fix the commutative differential graded algebra $\R[b,b^{-1}]$ (with trivial differential {and $\deg(b)=-2$})
as a real model of complex $K$-theory. In fact, there exists a map of commutative ring spectra $$c:\KU\to H\R[b,b^{-1}]$$ ({where $H$ is the Eilenberg-MacLane equivalence between chain complexes and $H\Z$-module spectra}) which becomes an equivalence after forming the product of $\KU$ with the Moore spectrum $M\R  \simeq H\R $. The fact that complex $K$-theory admits such a simple real model is a consequence of the property of being differentially simple (Definition \ref{good}). We call the triple $\widehat{\KU}:=(\KU,\R[b,b^{-1}],c)$ a differential refinement of $\KU$.

 In \cite{bg}, \cite{skript} the choice of such a differential refinement was the starting point for the construction of a multiplicative (untwisted) differential complex $K$-theory. In the present paper we show that this choice also suffices to construct twisted differential complex $K$-theory.

\bigskip 

In Definition \ref{may0620} we define, functorially in {the manifold  $M$}, a Picard-$\infty$ groupoid  $\Tw_{ \widehat{\KU}}(M)$   of differential twists of $\KU$ on $M$. A differential twist of $\KU$ on $M$ is essentially a triple $(E,\cM,d)$ {consisting }of 
\begin{itemize}
\item a topological $\KU$-twist $E\in \PicL_{\underline{\KU}}(M)$,
\item a sheaf of {real chain} complexes $\cM$ {over $M$ which is additionally a $DG$-module over the sheaf $\Omega \R[b,b^{-1}]$ (given by $U \mapsto \Omega^*(U)[b,b^{-1}]$}) ,
\item and an equivalence $d$ which identifies $\cM$ with {a real version $E_\R$ of $E$}. 
\end{itemize}

By Example \ref{may0220} the groupoid $\Tw_{\widehat{\KU}}(M)$ in particular contains twists of the form
\begin{equation}\label{may0521}
\hat E:=\big(E,{(\Omega \R[b,b^{-1}]},d+b\omega),d\big)\ ,
\end{equation} where $E$ is a topological $\KU$-twist classified by a class $z\in H^{3}(M;\Z)$ and $\omega\in Z^{3}(\Omega^*(M))$ represents the de Rham class of $z$.

\bigskip

Twisted differential complex $K$-theory on $M$ {is by Corollary \ref{cor_june26}} a {lax} symmetric monoidal functor (i.e. a graded ring)
$${\widehat{\K}^{(\dots)}}(M):\Ho\Tw_{\widehat{\KU}}(M)\to \Ab\ .$$

An important piece of the structure of twisted differential cohomology is the differential form contribution. 
More precisely the construction of twisted  differential $K$-theory  provides a morphism of abelian groups
$$\curv: \widehat{\K}^{(E,\cM,d)}(M) \to Z^{0}\cM(M) \ ,$$
where $Z^{0}\cM(M)$ are the zero-cycles in   the    $DG$-$\Omega^*(M)[b,b^{-1}]$-module 
%\tfoot{Das macht so keinen Sinn, entweder es ist ein DG-Modul, also insbesondere ein Kettenkomplex, oder nur ein Kettenkomplex}
 $\cM(M)$. The morphism $\curv$ maps a differential cohomology class to its underlying `curvature form'. 
% It is natural in the twist $(E,\cM,d)$, more precisely it is a morphism of graded rings. 
This transformation can be considered as a multiplicative differential form level refinement of the Chern character  \eqref{may0520} which is natural in the twist.

%\ntho{ More precisely,
%the evaluation on the second component of 
%differential twists produces a symmetric monoidal functor \tho{from  $\Ho\Tw_{ \widehat{\KU}}(M)$} to sheaves  of $DG$-$\Omega \R[b,b^{-1}]$-modules over $M$. If one further composes this functor with the evaluation at $M$ and the functor which takes zero cycles, then one obtains a symmetric monoidal functor 
%%$$\cZ^{0}:\Ho\Tw_{\widehat{\KU}}(M)   \to \Ab \qquad \tho{ (E,\cM,d) \mapsto Z^0\cM(M) }\ .$$
%The construction of twisted differential $K$-theory provides a natural transformation
%%$$\curv:\widehat{\KU}^{(\dots)}(M) \to \cZ^{0}(\dots) \ ,$$
%which maps a differential cohomology class to its underlying curvature form. }

\bigskip

Let us make this transparent for the twist \eqref{may0521}.  In this case we get a map
$$\curv: \widehat\K^{\hat E}(M)\to {Z^{0}\big(\Omega^*(M)[b,b^{-1}],d+b\omega\big)} \subset {\Omega^*(M)[b,b^{-1}]}\ .$$
For  $x\in \widehat\K^{\hat E}(M)$ the form $\curv(x)\in {\Omega^*(M)[b,b^{-1}]}$ satisfies the twisted cycle condition 
$(d +b\omega)\curv(x)=0$, and its cohomology class
 represents $\ch(I(x))$ in twisted de Rham cohomology, where $I(x)\in \K^{E}(M)$ is the underlying twisted $K$-theory class of $x$ and $\ch$ is the twisted Chern character map \eqref{may0520}.  

\bigskip

We continue this example and investigate the action of automorphisms of twists. As an application of the differential cohomology interpretation of differential twists explained in Section \ref{may0511}, in particular 
Remark \ref{may0201}, 8.,
we have a map ${\widehat{H}^{2}(M,\Z)}\to {\Aut_{\Ho \Tw_{\widehat{\KU}}(M)}(\hat E)}$.
If {a class $\phi\in \widehat{H}^{2}(M,\Z)$ in second ordinary differential cohomology} gives rise to such an automorphism of $\hat E$, then    it in particular acts on the complex $\Omega^*(M, \R[b,b^{-1}])$. {In this case the action} is given by multiplication by $\exp(-b \curv(\phi))$ {with $\curv(\phi) \in \Omega^2(M)$. Moreover the diagram}
$$\xymatrix{
\widehat{\K}^{\hat E}(M)\ar[d]^{\curv}\ar[rrr]^{\phi} &&&\widehat{\K}^{\hat E}(M)\ar[d]^{\curv}\\
{Z^{0}\big(\Omega^*(M)[b,b^{-1}],d+b\omega\big)} \ar[rrr]^-{\exp(-b\curv(\phi))} &&& {Z^{0}\big(\Omega^*(M)[b,b^{-1}],d+b\omega + b \curv(\phi) \big)}
}$$

commutes. This can be interpreted as  part of {the} statement of  naturality of the twisted Chern character forms.

Given a second differential twist $\hat E^{\prime}$ of the form \eqref{may0521} we can form the product
 of differential twists $$\hat E+\hat E^{\prime}=(E+E^{\prime},\big(\Omega \R[b,b^{-1}],d+b(\omega+\omega^{\prime})\big) ,d+d^{\prime})\ .$$  The fact that $\curv$ is a symmetric monoidal transformation  in particular says that the diagram
$${\scriptsize \xymatrix{
\widehat{\K}^{\hat E}(M)\otimes \widehat {\K}^{\hat E^{\prime}}(M)\ar[d]^{\curv\otimes \curv}\ar[rr]^{\cup}&&\widehat{\K}^{\hat E+\hat E^{\prime}}(M)\ar[d]^{\curv}\\
%\tho{Z^{0}\big(\Omega^*\big(M, \R[b,b^{-1}]\big),d+b\omega\big)} \otimes \tho{Z^{0}\big(\Omega^*\big(M, \R[b,b^{-1}]\big),d+b\omega\prime\big)}
Z^{0}\big(\Omega^*(M)[b,b^{-1}],d+b\omega\big) \otimes Z^{0}\big(\Omega^*(M)[b,b^{-1}],d+b\omega'\big)
\ar[rr]^-{\wedge } &&  
Z^{0}\big(\Omega^*(M)[b,b^{-1}],d+b(\omega+\omega')\big)
%\tho{Z^{0}\big(\Omega^*\big(M, \R[b,b^{-1}]\big),d+b(\omega + \omega \prime)\big)}
}}$$
commutes.
This is a differential form level statement of the fact that the twisted Chern character is compatible with products.

\bigskip

So far we  {have} discussed consequences of the construction of twisted differential $K$-theory for {the} Chern character. A major mathematical motivation to consider differential cohomology (and its twisted versions) is the fact that it receives secondary {invariants}. The application (see the examples discussed in \cite[Sec. 4]{skript}) of differential cohomology to the construction of secondary invariants is usually related with the exact sequences stated in general in Proposition \ref{may0630}. In the literature, a major part of the effort is invested in the construction of these sequences. In fact, their verification is important in order to state that the respective construction really produces a version of (twisted) differential cohomology. In our approach, similar as in  \cite{MR2192936}, these sequences are immediate consequences of the construction.

\bigskip

Observe that the theory developed in the present paper does not completely provide the mathematical background for the constructions envisaged in \cite{MR2742428}. For this it would be necessary {to} include group actions, orbifolds and {equivariant twists}. 

While the advantage of our homotopy theoretic construction of twisted differential $K$-theory is that it is easy to verify functorial properties and construct maps out of twisted differential $K$-theory, a  shortcome is the difficulty to construct maps into differential twisted $K$-theory from
differential geometry or analytic objects (see 
\cite{bg} for a more detailed discussion).  \\

\noindent
\textbf{Acknowledgements.}
We would like to thank David Gepner and Michael V\"olkl for helpful discussions.  {The authors are supported by the Collaborative Research Centre 1085 ``Higher Invariants''. }

\part{Twisted cohomology via graded ring spectra}
%%%%%%%%%%%%%%%%%%%%%%%%%%%%%%%%%%%%%%%%%%%%%%%%%%%%%%%%%%%%%%%%%%

%\section{Notation, Conventions}

%\tho{We should say that we follow Lurie and Joyal in the use of $\infty$-categories. Also we should say that commutative ring spectrum means the version of $E_\infty$} \\
%\uli{Wir brauchen einheitliche Notation f\"ur
%\begin{enumerate}
%\item Shifts: Ich bef\"uhrworte $E[n]$ (f\"ur Komplexe wie f\"ur Spektren).
%%\item $\underline{E}$ oder $\Sm(E)$.
%%Letzteres hat den Vorteil, da{\ss} klar ist, da{\ss} $\Sm(E)(M)\cong \Map(M,E)$.
%\item $\hat \cR$ und $\cR$ f\"ur das graduierte Ringspektrum. 
%%\item Das Symbol $\iota$ generell f\"ur Lokalisierung verwenden.
%Also z.B. $\Omega A_{\infty}$ durch $\iota \Omega A$ ersetzen $gl_{1}$ roman oder bf machen.
%%\item $gl_{1}$ und $bgl_{1}$ roman oder bf machen.
%\item Auswertung von Objetion $E\in \Sh(M)$ auf $M$ ist $E(M\to M)$ statt $E(M)$,. Mu{\ss} durchgecheckt werden.
%\item generelle Konvention: Kettenkomplex und Kokettenkomplexe dasselbe $A_{i}=A^{-i}$ und $(A[n])^{i}=A^{n+i}$. Dann $H(A[n])=(HA)[n]$.
%\item $\hat R$ und $\widehat R$ vereinheitlichen! Zu $\widehat{R}$!
%\item Notation f\"ur Spektren vereinheitlichen ($MSpin$ (preferiere $\mathbf{MSpin}$ Boldface?), $\KU$
%\item Notations $R$ und $\curv$ vereinheitlichen zu $\curv$
%\item Check usage of $\Diff$ for the differential function spectrum and $\widehat R^{\dots}$ for the cohomology
%\end{enumerate}
%}
 
%%%%%%%%%%%%%
\section{Graded ring spectra}\label{may0501}
%%%%%%%%%%%%%%%%%%%%%%%%%%%%%%%%%%%%%%%%%%%%%%%%%%%%%%%%%%%%%%%%%%%%%%%%%%%%%%%

\newcommand{\sym}{\SymMon}
\newcommand{\UFib}{(\Cat_\infty)_\star}

In this section we will, {generalizing   commutative {(more precisely $E_\infty$)} ring spectra,  define the notion of a commutative graded ring spectrum. As motivation we first explain the simpler case of the generalization from commutative rings to commutative graded rings.}

\bigskip

If $A$ is an abelian monoid, then a commutative $A$-graded ring 
is given by  the datum of an abelian group $R_a$ for every 
$a \in A$ together with morphisms $R_a \otimes R_b \to R_{ab}$ which must satisfy  
associativity and commutative conditions. 
We let $\underline{A}$ be the discrete    category
 on the underlying set of $A${, i.e. the objects {of $\underline{A}$} are the elements of $A$ and there are only identity morphisms}. The monoid structure of $A$ turns $\underline{A}$
into a symmetric monoidal category.  {With this in mind, a}  commutative $A$-graded ring {can   be considered as the   datum  of a} lax symmetric monoidal
functor from $\underline{A}\to \AbGr$, where $\AbGr$ denotes the symmetric monoidal category 
of abelian groups.

It   makes sense to replace $\underline{A}$ by an arbitrary symmetric
monoidal category $\calc $ and define the category of 
commutative $\calc$-graded rings to be the category of lax symmetric monoidal
functors from $\calc$ to abelian groups:
\begin{equation*}
\grRing_\calc := \Fun^{\otimes}_{\mathrm{lax}}(\calc,\AbGr)\ .
\end{equation*}

Lax symmetric monoidal here means that the monoidal coherence morphisms are possibly non-invertible. The direction is such 
that {a lax monoidal functor} sends  algebra objects in $\calc$ to algebra objects in $\AbGr$.
The assignment \begin{equation}\label{apr2401}  {\symstr\to \Cat^{op}}\ , \quad  \calc \mapsto \grRing_\calc \end{equation} defines a contravariant functor
from the 2-category of symmetric monoidal categories {and lax symmetric monoidal functors} $\symstr$ to the 2-category of
categories. 
\begin{ddd}\label{june27}
The $2$-category $\grRing$ is defined as the Grothendieck construction of the functor \eqref{apr2401}. \end{ddd}
   Thus
objects in $\grRing$ are   pairs
$(\calc,R_\bullet)$, where $\calc$ is a symmetric monoidal category and $R_\bullet$ is a
$\calc$-graded ring. {Moreover} there is a functor $U: \grRing \to \symstr$   given by $(\calc,R_\bullet)\mapsto \calc$.

\bigskip

We now want to generalize the notion of graded rings to graded ring spectra. 
Since spectra form an $\infty$-category it is most natural to {allow} indexing categories which are themselves
symmetric monoidal 
$\infty$-{categories} instead of ordinary symmetric monoidal categories.
In the present paper we consider the $\infty$-category of symmetric monoidal $\infty$-catgeories {and lax monoidal functors} $\sym$\footnote{This is defined a the {full} subcategory of the $\infty$-category of $\infty$-operads {spanned by symmetric monoidal $\infty$-categories}, see \cite{HA}.}
as the home of  the following web of full reflective $\infty$-subcategories
\begin{equation}\label{apr2930}\xymatrix{{\Ab} \ar[r]\ar[d]&\CGrp(\cS)\ar[d] & \symstr\ar[d]\\ 
\Ab\mathrm{Mon}\ar[r]&\CMon(\cS)\ar[r]&\sym
}\end{equation}
of commutative groups and monoids in the category of sets and the $\infty$-category $\cS$ of spaces. {We will not write these inclusions explicitly and hope that this abuse will not lead to confusion.} 
Commutative groups in spaces are also called Picard-$\infty$ groupoids. Thus a
 Picard-$\infty$-groupoid  is a symmetric monoidal
$\infty$-category such that the underlying $\infty$-category is an
$\infty$-groupoid  with the property that every object has a tensor inverse.  Note that the $
\infty$-loop space functor $\Omega^{\infty}$ provides an equivalence of the $\infty$-category of connective spectra $\Sp_{\ge 0}$ with the $\infty$-category of Picard-$\infty$ groupoids.

 \bigskip
Recall that the $\infty$-category $\Sp$ of spectra is a symmetric monoidal $\infty$-category
with respect to the smash product (see {\cite[Section 6.3.2]{HA}}).
 
\begin{definition} 
 Let $\calc$ be a symmetric monoidal $\infty$-category. 
A commutative $\calc$-graded ring spectrum is a lax symmetric monoidal functor $\calc
\to \Sp$. 
 \end{definition}

In the following we illustrate this definition by relating some special cases with more classical notions.

\begin{exa2}
 \begin{enumerate}
 \item Let $\calc$ be the trivial symmetric monoidal
$\infty$-category, i.e. a point considered as a {group-object in spaces}.  Then a $\calc$-graded ring spectrum is just  a commutative ring spectrum.
\item Let $R^\bullet$ be an ordinary commutative $\calc$-graded ring as in Definition \ref{june27}, where $\calc$ is an ordinary symmetric monoidal category. Then the collection of Eilenberg-Mac Lane spectra $HR^\bullet$ has the structure of a $\calc$-graded-ring {spectrum}\footnote{ {Here we should strictly speaking write $N\calc$ for the symmetric monoidal $\infty$-category associated to $\calc$, but we suppress this by our general language abuse}.}. 
%\item \ntho{Let $\cS $ be the symmetric monoidal $\infty$-category of spaces. The functor
%$\Sigma^{\infty}_{+}:\cS\to \Sp$ which associates to a space its suspension spectrum is symmetric monoidal. Hence it can be considered as an $\cS$-graded ring spectrum.}
%\tfoot{Das Beispiel brauchen wir doch gar nicht, und das bringt uns nur in Size-Probleme die ich an der Stelle nicht diskutieren moechte}
%\item  
%\ntho{We consider the symmetric monoidal $\infty$-category of chain complexes $\Ch_{\infty} $ obtained from the category $\Ch$ of chain complexes by inverting the quasi-isomorphisms. 
%We have a lax symmetric monoidal functor $H:\Ch_{\infty}\to \Sp$, the Eilenberg-MacLane
%correspondence. It can be considered as a commutative $\Ch_{\infty}$-graded ring spectrum.}
\item If $\calc$  is an  ordinary abelian monoid considered as a discrete symmetric monoidal  {$\infty$-}category, then a $\calc$-graded ring spectrum  is given by the  following set of data: \begin{enumerate}\item  for every element $a
\in \calc$ a spectrum ${R}_a$, \item for every pair of elements $a,b \in \calc$ a morphism
${R}_a \wedge {R}_b \to {R}_{ab}$, \item for triples $a,b,c \in \calc$ coherence cells
filling the obvious diagram. \item There is moreover an infinite list of higher coherence cells which is made precise by the above definition.\end{enumerate}
\item Let us consider a  $\calc$-graded ring spectrum $R $. If we apply the functor $\pi_0 : \Sp \to \AbGr$, then
get a lax monoidal functor $\calc \to \AbGr$. {It has a canonical factorization over}
 a lax monoidal functor $\ho \calc \to \AbGr$, where $\ho \calc$ is
the homotopy category of $\calc$. Thus the homotopy of a $\calc$-graded ring spectrum  $R$   is a   $\ho \calc$-graded ring $\pi_{0}(R)$. 
%\item Of course the above definition makes sense in any symmetric monoidal
%$\infty$-category $D$ (even in an $\infty$-operad) replacing the category of \marginpar{\uli{Do we need this?}}
%spectra. One might even let the source be an arbitrary $\infty$-operad $C$. 
%Then the definition just gives back $C$-algebras in $D$. 
\end{enumerate}
\end{exa2}
 
 In the following we define an $\infty$-category of commutative graded spectra.
 We let   $\Cat_\infty $ denote the $\infty$-category of   $\infty$-categories.  We  consider the  functor   
\begin{equation}\label{ulimar1810}
 \SymMon^{op} \to \Cat_\infty ,\quad  \calc \mapsto \Fun^{\otimes}_\mathrm{lax}(\calc, \Sp)
\end{equation}
 For the Grothendieck construction in $\infty$-categories
see \cite[Chapter 3]{HTT}.
\begin{definition}\label{def1aug13}
The $\infty$-category  of graded ring spectra $\GCS$ is defined to be the Grothendieck construction of the
functor  \eqref{ulimar1810}.  
\end{definition}
Thus {roughly speaking} an object in $\GCS$ is a  pair $(\calc, R)$
  of a symmetric monoidal $\infty$-category $\calc$ and a 
 $\calc$-graded ring spectrum $R$. A morphism  $(u,f):(\calc, R)\to  (\calc', R')$  consists of 
  a {lax} symmetric monoidal functor $f: \calc \to \calc'$  and a morphisms of graded ring spectra 
$u:R  \to f^{*}R' $. Moreover, we have a canonical forgetful functor $\GCS \to \SymMon$. 

\bigskip

 The following lemma is an immediate consequence of the definitions.
\begin{lemma}
There is a functor 
\begin{equation*}
\GCS \to \grRing \qquad (\calc,R) \mapsto (\ho \calc, \pi_0(R)).
\end{equation*}
\end{lemma}
We have collected some technical facts about graded ring spectra in the {Appendix \ref{aapend}}. 

%%%%%%%%%%%%%%%%%%%%%%%%%%%%%%%%%%%%%%%%%%%%%%%%%%%%%%%%%%%%%%%%%%%%%%%%%%%%%%%%%%%%%%%%%%%%%%%%%%%%%%%55
\section{Twisted cohomology}\label{ulimar2201}
%%%%%%%%%%%%%%%%%%%%%%%%%%%%%%%%%%%%%%%%%%%%%%%%%%%%%%%%%%%%%%%%%%%%%%%%%%%%%%%%%%%%%%%%%%%%%

A classical instance of twisted cohomology is the cohomology of a space with coefficients in a locally constant sheaf,  for example the  orientation sheaf. Twisted $K$-theory has been introduced in \cite{MR0282363}. 
Much later, motivated by applications in physics it became an intensively studied object, see Section \ref{may0650} {of the introduction}.
 General approaches to twisted cohomology are given in \cite{MR2271789} and  \cite{MR2681757}, \cite{2011arXiv1112.2203A}.  This last reference can be considered as a starting point for the approach taken in the present paper.

In this section we review twisted cohomology and phrase it in the framework of graded rings and graded ring spectra.  The construction serves as a model for the slightly more involved case of twisted differential cohomology. 

\bigskip

 We first need some sheaf-theoretic notation. In the present paper we consider sheaves
with values in a presentable $\infty$-category $\cC$. An ordinary category is considered as an $\infty$-category via its  nerve.
 By $\Mf$ we denote the site of smooth manifolds with the pretopology of open coverings. By $\Sh_{\cC}$  %\tfoot{Sollen wir vielleicht doch lieber die Mannigfaltigkeiten mit notieren, finde des etwas verwirrend sonst.}
  we denote the $\infty$-category {of} sheaves on the site $\Mf$ with target $\cC$.   We use the symbol $\Sh_{\cC}(M)$  in order to denote  the category of  sheaves on the restricted site $\Mf/M$.  
 For a sheaf $A\in \Sh_{\cC}$ we let  $A_{|M}\in \Sh_{\cC}(M)$ denote the restriction
of $A$ to $\Mf/M$.  
For presheaves we will use similar conventions and the symbol $\PSh$.
{A reference for sheaves with values in $\infty$-categories is $\cite{HTT}$. See also the note $\cite{diffsheaf}$ for some sheaf theory in this setting.}

The diagram of reflective subcategories \eqref{apr2930} induces a corresponding diagram of reflective subcategories of sheaves on $\Mf$ or any other site {with values in the respective $\infty$-categories}. We will use these inclusions without further notice. {Sheaves} of groupoids will also be {referred} to as stacks. In particular, a sheaf of Picard-$\infty$ groupoids is called a Picard-$\infty$ stack.

An object $X\in \cC$ gives rise to a constant presheaf $C_{X}\in \PSh_{\cC}$ defined by
$ C_{X}(N):=X$ for all objects $N\in \Mf$. By $ \hat C_{X}\in \Sh_{\cC}$ we denote its sheafification.  

\begin{ddd}
 Let $M$ be a smooth manifold. 
A sheaf $\cF\in \Sh_{\cC}(M)$ is called constant if there exists an object $X\in\cC$
and an equivalence $ (\hat C_{X})_{|M}\cong \cF$.

A sheaf $\cF\in \Sh_{\cC}(M)$ is called locally constant if every point $m\in M$ has a neighbourhood $U\subseteq M$ such that $\cF_{|U}\in \Sh_{\cC}(U)$ is constant.
We write
$$\Sh^{\loc}_{\cC}(M)\subseteq \Sh_{\cC}(M)$$ for the full subcategory of locally constant sheaves.
\end{ddd}

\begin{ddd}\label{apr2630}
For an object $X\in \cC$ of  a presentable $\infty$-category $\cC$  we define the sheaf   $\underline{X}\in \Sh_{\cC}$
such that its evaluation on the manifold $M$ is given by $ \underline{X}(M):=X^{M^{top}}$.
\end{ddd}
In this definition,  $M^{top}$ is the   underlying space of $M$, and we use that a presentable $\infty$-category   is cotensored over spaces.  We refer to the appendix of \cite{bg} for  the verification of the the following facts. 
\begin{prop}\phantomsection \label{propju22} 
\begin{enumerate}
\item The presheaf $M {\mapsto}  X^{M^{top}}$ is indeed a sheaf.
%\item \ntho{The sheaf $\underline{X}$ is homotopy invariant.}
\item The sheaf $\underline{X}$ is constant. In fact, we have an 
  equivalence $\underline{X}\cong \hat C_{X}$.
\end{enumerate}
\end{prop}

If $\cC$ is a (closed) symmetric monoidal $\infty$-category, then the categories $\PSh_\cC(M)$, $\Sh_{\cC}(M)$ and 
$\Sh^{\loc}_{\cC}(M)$ also carry (closed) symmetric monoidal structures {given by the objectwise tensor product}. Thus it makes sense to speak of {(commutative)} algebra and module objects in these categories. It turns out that algebras in (locally constant) sheaves are the same as (locally constant) sheaves of algebras. 
\begin{prop} 
Let $\calc$ be a (closed) presentable symmetric monoidal ${\infty}$-category.
Then we have equivalences
$$ \CAlg(\Sh_\cC(M)) \simeq \Sh_{\CAlg(\cC)} \qquad \text{and} \qquad  \CAlg(\Sh^\loc_\cC(M)) \simeq \Sh^\loc_{\CAlg(\cC)}(M).  $$
If $X \in \cC$ is a commutative algebra object, the sheaf $\underline{X}$ is a locally constant sheaf of {commutative} algebras and we have similar equivalences
$$ \Mod_{\underline{X}}(M) \simeq \Sh_{\Mod_X}(M) \qquad \text{and} \qquad \Mod^\loc_{\underline{X}}(M) \simeq \Sh^\loc_{\Mod_X}(M).  $$
\end{prop}
\begin{proof}
{
We give the proof for the first equivalences. The argument for the   second assertion is   similar. 
For categories of presheaves we have an evident equivalence
$\CAlg(\PSh_\cC(M)) \simeq \PSh_{\CAlg(\cC)}$
which directly follows from the definition resp. universal property of the point wise tensor product on functor categories  {\cite[Section 2.2.5.]{HA}}. Thus it is enough to check that under this equivalence sheaves and locally constant sheaves are preserved. For the case of sheaves this follows since the sheaf condition is a limit condition and the forgetful functor which forgets algebra structures commutes with limits in any closed symmetric monoidal $\infty$-category  {\cite[Section 3.2.2.]{HA}}. For the second equivalence we use that the first obviously commutes with restriction. {It therefore suffices to show the equivalence}  for constant sheaves.  {Here we use again} that the forgetful functor that forgets the algebra structure commutes with limits, {so} in particular with the tensoring used in \ref{propju22} {in order to define} constant sheaves.
}
\end{proof}

For a commutative ring spectrum $R$ the $\infty$-category of module spectra over $R$ will be denoted $\Mod_R$. By  $\Pic_{R}$  we denote the maximal Picard-$\infty$-groupoid inside of $\Mod_R$, i.e. the objects are $R$-module spectra which admit a tensor inverse, and the morphisms are the invertible  {$R$-module maps}. Similarly by $\Pic^\loc_{\underline{R}}(M)$ we denote the maximal Picard-$\infty$-groupoid inside the $\infty$-category $\Sh^\loc_{\Mod_R}(M)$ of locally constant sheaves of $R$-modules on $M$.   
%\uli{For $\underline{R}\in \CAlg(\Sh_{\cC})$ we define the Picard stack $\Pic_{\underline{R}}$ in a similar manner.}

\begin{ddd} The objects of $\PicL_{\underline{R}}(M)$ are called $R$-twists on $M$.
\end{ddd}

Let us list a few properties of $R$-twists that we will need later. The proofs are given in Appendix \ref{bapend}.
\begin{prop}\label{prop_jul2}
\begin{enumerate}
\item 
The assignment $M \mapsto \Pic^\loc_R(M)$   {itself}  forms a Picard-$\infty$-stack on $\Mf$, i.e.
\begin{equation}\label{ulimar2002}\PicL_{\underline{R}}\in \Sh_{\CGrp(\cS)}\ .\end{equation}
\item For every manifold $M$ there
 is a canonical equivalence of Picard-$\infty$-groupoids
\begin{equation*}
\PicL_{\underline{R}}(M) \cong (\Pic_R)^{M^{top}} = \underline{\Pic_R}(M)\ .
\end{equation*}
\item Every $R$-twist $E \in \PicL_{\underline{R}}(M)$ on $M$ is homotopy invariant. That means that  for every  map $N \to N'$  in $\Mf/M$ which is a homotopy equivalence   (not necessarily over $M$)  the induced map $E(N' \to M) \to E(N \to M)$ is an equivalence of $R$-modules. 
\end{enumerate}
\end{prop}

\bigskip

For any manifold $M$ the evaluation  provides   a lax symmetric monoidal transformation  $$\Gamma: \PicL_{\underline{R}}(M) \to  \Sp  \qquad E\mapsto \Gamma(M,E):=E(M) . $$ 
{Hence the} pair $(\PicL_{\underline{R}}(M),\Gamma)$ is a graded ring spectrum  which we denote $\cR(M)$.  
\begin{prop} Let $R \in \CAlg(\Sp)$ be a commutative ring spectrum. 
Then the assignment $M \mapsto \cR(M)$ canonically refines to a sheaf of graded ring spectra $\cR\in \Sh_{{\GCS}}$.

\end{prop}
\begin{proof}
By {Proposition   \ref{changesheaf}} 
 % to $ \Mod_{\underline{R}}\cong \Sh_{\Mod_{R}}$ in order to get
 we have an object 
 \begin{equation}\label{ulimar1401}F\in \Sh_{\GCS}\ , \quad \Mf\ni M\mapsto F(M):=( \Sh_{\Sp}(M), \Gamma(M,\dots))\ .\end{equation}
 We now want to change the grading component of $F$ to $\PicL_{\underline{R}}$. Therefore we apply the change of grading Proposition \ref{lem_drei} to $F$, the sheaf $G:=\PicL_{\underline{R}}\in \Sh_{\sym}${,} 
 and the transformation $\eta:\PicL_{\underline{R}}\to \Sh_{\Sp}$ which forgets invertibility of objects and morphisms and the $R$-module structure  in order to get a new sheaf of graded ring spectra
 $\cR\in \Sh_{\Sp}.$ 
\end{proof}

\bigskip

The construction of the graded ring spectrum $\cR$ from $R$ is {even} functorial in $R$.  
More precisely we have:
\begin{prop}\label{prop_jun2ee}
The construction $R\mapsto \cR$ naturally refines to a functor
$$\CAlg(\Sp)\to \Sh_{\GCS}\ .$$
\end{prop}
\begin{proof}
We describe the adjoint functor $ \CAlg(\Sp) \times \Mf^{op} \to \GCS .$
{We apply the  change of grading statement (Proposition \ref{lem_drei}) together with the sheaf-refinement given in Proposition \ref{changesheaf} to the $\infty$-category   $\cald:=    \CAlg(\Sp) \times \Mf^{op} ,$ the functors }
$$ F:  \CAlg(\Sp) \times \Mf^{op} \xrightarrow{\text{pr}_2} \Mf^{op} \xrightarrow{\eqref{ulimar1401}} \GCS$$
and   
 $$G:  \CAlg(\Sp) \times \Mf^{op} \to \sym \ , \quad (R,M) \mapsto \PicI_{\underline{R}}(M)\ ,$$ 
 and the  lax symmetric monoidal transformation $$\eta: G\to  \Sh_{\Sp}\ , \quad   \PicI_{\underline{R}} (M)\to \Sh_{\Sp}(M)\ .$$ 
\end{proof}

{We can now define  twisted cohomology. In order to simplify the notation we write $E(M)$ for the evaluation of a sheaf {$E$} on $\Mf/M$ on $M\stackrel{\id}{\to}M$. }

\begin{ddd}
Given a manifold $M$ with $R$-twist $E \in \ho \PicL_{\underline{R}}(M)$  %\footnote{\tho{Homotopie Kategorie gestricken}}  
we define the $E$-twisted
$R$-cohomology of $M$ by
 {$$R^{E}(M):=\pi_0(E(M))\ .$$}\end{ddd}

\begin{rem2}
The reader might wonder why there is no degree in the twisted cohomology groups $R^E(M)$. He might also know that cohomology groups are usually twisted by classes in $[M, \BGL_1(R)]$ where $\GL_1(R)$ is the space of units of $R$. The reason is that $\PicL_{\underline{R}}(M)$ {encodes both, the degree and  the twist.}
 
More precisely homotopy classes of maps $[M, \BGL_1(R)]$ classify those  {$R$-twists} $E \in {\ho}{\PicL_{\underline{R}}(M)}$ which are trivial when restricted to points. {For such a twists $E$ and integer  $n\in \Z$  the twisted cohomology of degree $n$ can be expressed  in the present language in the form
$R^{E[n]}(M)$, where  $E[n]{\in \ho \PicL_{\underline{R}}(M)}$  is the shift of $n$.} \\

In general, $\PicL_{\underline{R}}(M)$ contains more  then just locally trivial twists ({do not confuse with locally constant}) and shifts.  {For example the $E(1)$-local sphere contains exotic elements in its Picard group \cite{MR1263713}. Another example is given by the Eilenberg-MacLane spectrum $HR$ for $R$ a Dedekind ring with nontrivial class group, such as  $\Z[\sqrt{-5}]$.}
%\footnote{\uli{Koennen wir das belegen?} \tho{Ja klar koennen wir das, wir brauchen nur einen gewoehnlichen Ring der nichtriviale picard gruppe hat. Zum Beispiel einen Dedekindring mit nichttivialer Klassengruppe (Beispiel: $\Pic({\Z[\sqrt{-5}]}) = \Z/2$)} }
 \end{rem2}

\begin{kor}\label{kor_gradedring}
For each manifold $M$ the twisted cohomology $R^*(M)$ forms a graded
ring which is graded over the homotopy category $\ho \PicL_{\underline{R}}(M)$. 
Moreover the assignment 
\begin{equation*}
M \mapsto \big(\ho\PicL_R(M), R^*(M)\big) 
\end{equation*}
forms a functor $\Mf^{op} \to \mathrm{GrRing}$.
\end{kor}

Corollary \ref{kor_gradedring} says more concretely that we have the following structures:
\begin{enumerate}
\item For every {$R$-twist} $E \in  {\ho}\PicL_{\underline{R}}(M)$ {on $M$ we have} an abelian group $R^E(M)$.
\item For every pair $E,E'$ of {$R$-twists} on $M$ we have a multiplication morphism 
 \begin{equation*}
 R^E(M) \otimes R^{E'}(M) \to R^{E + E'}(M)\ , 
 \end{equation*}
 where $E + E'$ denotes the symmetric monoidal pairing in $ {\ho}\PicL_{\underline{R}}(M)$. The multiplication is associative, commutative{,} and has a unit in the obious sense.
\item For every     morphisms $E \to E'$ {between $R$-twists on $M$}   we have  {an iso}morphism 
 \begin{equation*}
 R^E(M) \iso R^{E'}(M) 
 \end{equation*}
 {of abelian groups}.
 This assignment is functorial and compatible with the product in the obvious sense.
\item For every smooth map $f: M \to N$ and {$R$-twist $E$ on $M$} we have a morphism
 \begin{equation*}
 R^E(N) \to R^{f^*E}(M)
 \end{equation*}
 which is compatible with all the other {structures}.
 \end{enumerate}
 Further features of twisted cohomology are:
 \begin{enumerate}\setcounter{enumi}{4}
 \item For every map of commutative ring spectra $u:R\to R^{\prime}$ and {$R$-twist} $E\in \Ho\PicL_{\underline{R}}(M)$ we get an {$R$-twist} $u_{*}E\in  \Ho\PicL_{\underline{R'}}(M)$ and a map
 $u_{*}:R^{E}(M)\to R^{\prime,u_{*}E}(M)$.
 %\item All these structures are natural with respect to smooth maps $M\to M^{\prime}$.
 \item {If $E$ is an $R$-twist on $M$ and $U,V\subseteq M $ are open subsets such that $M= U\cup V$, then we have a long exact Mayer-Vietoris sequence
 \begin{equation}\label{apr2410}\dots \to R^{E}(M)\to R^{E_{|U}}(U)\oplus R^{E_{|V}}(V)\to R^{E_{|U\cap V}}(U\cap V)\to   R^{E[1]}(M)\to \dots \end{equation} which is compatible with the various other structures. This follows from the fact that $\cR$ is a sheaf of graded ring spectra.} 
 %\tfoot{Das ist eine schoene Bemerkung, allerdings folgt das natuerlich strikt genommen nicht aus dem Korollar \uli{pa{\ss}t aber dennoch hierein?}}
 \end{enumerate}

The idea to calculate  $R^{E}(M)$ using a Mayer-Vietoris sequence \eqref{apr2410} 
is to cover the manifold $M$ by open subspaces on which the twist can be trivialized.  The allows to express the twisted cohomology of these subspaces in terms of untwisted cohomology. 
It remains to determine  the maps in this sequence.

\begin{exa2}

In the following we consider this problem when $M$ is a sphere $S^{k}$ for some $k\in \nat$. In this case we can determine the maps explicitly in terms of an invariant of the twist defined in Definition \ref{ulimar2101}. This invariant will also play an important role in our general theory later, e.g. in the proof of Lemma \ref{apr2420}.

\bigskip

We identify  $S^{k}\subset S^{k+1}$ with the equator and choose   a base point $p\in S^{k}$.
We further let $D_{\pm}\subset S^{k+1}$ be the complements of the north- and south poles.

Assume that  $E\in \PicL_{\underline{R}}(S^{k{+1}})$ is a twist which  {locally is a} shift of $R$. Then we 
 can choose an identification
$E_{|p}\cong R[n]$ of $R$-modules for some $n\in \Z$. 

{Since $E$ is homotopy
invariant by Proposition \ref{prop_jul2} we can} extend this identification  to  trivializations $$\phi_{\pm}:E_{|D_{\pm}}\iso \underline{R[n]}_{|D_{\pm}}\ .$$
 
\begin{ddd}\label{ulimar2101}For $k\ge 0$ we define the class  $$\nu(E)\in   \left\{\begin{array}{cc}\pi_{k}(R)&k\ge 1\\\pi_{0}(R)^{\times}&k=0 \end{array}\right.   $$   such that $1\oplus\nu(E)\in R^{0}(S^{k})\cong \pi_{0}(R)\oplus\pi_{k}(\R)$ is the 
image of $1\in R^{0}(D_{+})$ under the map
$$R^{0}(D_{+})\stackrel{\phi_{+}^{-1}}{\iso} \pi_{0}(E[-n]_{|D_{+}}(D_{+}))\stackrel{restr.}{\to} \pi_{0}(E[-n]_{|S^{k}}(S^{k}))\cong \pi_{0}(E[-n]_{|D^{-}}(S^{k}))\stackrel{\phi_{-}}{\iso} R^{0}(S^{k})\ .$$
\end{ddd}

One easily checks that $\nu(E)$ does  not depend on the choices. In fact, the trivializations $\phi_{\pm}$ are  {uniquely determined up to equivalence} by the choice of the trivialization at the point $p$. In the construction of $\nu(E)$ we compose
$\phi_{+}^{-1}$ with $\phi_{-}$ so that this choice drops out. \\

The twisted cohomology
$R^{E}(S^{k{+1}})$ can be calculated using the Mayer-Vietoris sequence \eqref{apr2410}. 
Using the class $\nu(E)$ this sequence can be identified with the explicit sequence
$$\dots \to\begin{array}{c}\pi_{1-n}(R)\\\oplus\\ \pi_{k+1-n}(R)\end{array}\to R^{E}(S^{k})\to \begin{array}{c}\pi_{-n}(R)\\\oplus\\ \pi_{-n}(R)\end{array}\xrightarrow{\left(\begin{array}{c}x\\ y \end{array}\right)\mapsto \left(\begin{array}{c}x-y\\ \nu(E)x\end{array}\right) } \begin{array}{c}\pi_{-n}(R)\\\oplus\\ \pi_{k-n}(R)\end{array}\to  \dots\ .$$
\end{exa2}
 
 \begin{exa2}\label{may0601}
 {The reader may have seen the example of twisted complex $K$-theory. Let $\KU$ denote the complex $K$-theory spectrum.
 Locally trivial unshifted $\KU$-twists $E$ of $S^{3}$ are classified up to equivalence by the invariant $\nu(E)\in  \pi_{2}(\KU)\cong \pi_{3}(\BGL_{1}(\KU))$. Let us identify $\pi_{2i}(\KU)\cong \Z$  for all $i\in \Z$ using  powers of the Bott element. If $\nu(E)$ corresponds to the integer $n\in \Z$, then
$\KU^{E[1]}(S^{3})$ fits into the sequence
$$0\to  \begin{array}{c} \Z\\\oplus\\ \Z\end{array}\xrightarrow{\left(\begin{array}{c}x\\ y \end{array}\right)\mapsto \left(\begin{array}{c}x-y\\ n x\end{array}\right) } \begin{array}{c}\Z\\\oplus\\ \Z \end{array}\to \KU^{E[1]}(S^{3})\to 0\ .$$
This implies the usual calculation: $\KU^{E[1]}(S^{3})\cong \Z/n\Z$, and by a similar argument, $\KU^{E}(S^{3})=0$.
}\end{exa2}

%\begin{exa2}{\nuli{
%One can capture naive equivariant cohomology in terms of twisted cohomology.
%Let $G$ be a topological group which acts on a spectrum $R$. Then we can consider the naive  equivariant cohomology $R_{G}^{*}(M)$ of a $G$-space $M$. }}

%\nuli{For this example we  replace the category of manifolds by the category of spaces in order to include objects like $BG$.
%The action of $G$ on $R$ can be considered as a twist $E\in \Pic_{\underline{R}}^{loc}(BG)$. A $G$-space $M$ gives rise to a space $M_{G}:=EG\times_{G}M\to BG$ over $BG$ and we have a natural isomorphism
%$$R_{G}^{*}(M)\cong R^{E}(M_{G})\ .$$}
%\tho{Brauchen wir das eigentlich?}
%\uli{Nicht unbedingt. Man kann sogar noch erweitern und eine Abbildung
%$J:K^{0}(B)\to \Pic^{loc}_{{\underline{S}}}(B)$ definieren,
%welche wir sp‚Äö√†√∂¬¨√üter auf das differentielle Level heben k\"onnten.}
%}}

%\end{exa2}

 % \begin{example}
% ordinary cohomology
% \end{example}
% 
% \begin{example}
%  K-theory (and gerbes)
% \end{example}

%%%%%%%%%%%%%%%%%%%%%%%%%%%%%%%%%%%%%%%%%%%%%%%%%%%%%%%%%%%%%%%%%%%%%%%%%%%%%%%%
%%%%%%%%%%%%%%%%%%%%%%%%%%%%%%%%%%%%
\section{Twisted differential cohomology}\label{sec_def_twisted}
%%%%%%%%%%%%%%%%%%%%%%%%%%%%%%%%%%%%%%%%%%%%%%%%%%%%%%%%%%%%%%%%%%%%%%%%%%%%%%%%
%%%%%%%%%%%%%%%%%%%%%%%%%%%%%%%%%%%%

In this section we introduce twisted differential cohomology. {The idea of differential cohomology is to combine cohomology classes on manifolds with corresponding differential form data in a homotopy theoretic way. Using differential cohomology one can encode local geometric information.
For example, while the $\K$-theory class of a vector bundle only encodes   homotopy theoretic data, the differential $K$-theory class of a vector bundle with connection contains the information on Chern character forms, see  \cite{MR2192936}, \cite{MR2664467}, \cite{MR2732065}  for a detailed account. 
 Differential extensions of arbitrary generalized cohomology theories were first defined in \cite{MR2192936}. We refer to \cite{bg} and \cite{skript} for the approach to differential cohomology on which the present paper is based. The main goal of the present section is to develop a general theory of {\em twisted} and {\em multiplicative} differential cohomology based on the notion of graded ring spectra. }

\bigskip

A  commutative ring spectrum $R$ represents a multiplicative cohomology theory. As explained in \cite{bg}, \cite{skript},  in order to define the  multiplicative differential $R$-cohomology on manifolds we have to choose a differential refinement of $R$. In Section \ref{ulimar2201}
we have seen that $R$ gives rise to a twisted cohomology theory which is naturally encoded in the sheaf of graded ring spectra $\cR$.     We shall see in the present section that the choice of a differential refinement  of $R$ naturally determines  a sheaf of graded ring spectra (Proposition \ref{ulimar2402}) which encodes the twisted differential $R$-cohomology.  In order to be able to study transformations between different differential cohomology theories we introduce the $\infty$-category of differential  {ring spectra}. 

\bigskip

Let $\Ch$ denote the  {ordinary} symmetric monoidal category of chain complexes.  
%\tfoot{Ich verstehe nicht, dass wir erst alles ueber $\Z$ machen und nicht direkt ueber $\R$. Wir brauchen $\Z$ doch nie in dem ganzen Papier!?}
{If we formally invert the class of quasi-isomorphisms in $\Ch$,} then we obtain an $\infty$-category 
$\Ch_{\infty}$.
The {natural} map  $\iota:\Ch\to \Ch_{\infty}$ is a  lax symmetric monoidal functor. Furthermore, the Eilenberg-MacLane equivalence gives an equivalences   (see \cite[Theorem 8.1.2.13]{HA} for details)  of symmetric monoidal $\infty$-categories
$$H:\Ch_{\infty}\stackrel{\sim}{\to} \Mod_{H\Z} ,$$
{where  $\Mod_{H\Z}$ denotes module spectra over the Eilenberg-Mac Lane spectrum $H\Z$.} 
 By abuse of notation, for $C\in \Ch$ we write
$HC:=H\iota(C)$.
Note that {the category of commutative {algebras} $\CAlg(\Ch)$} in $\Ch$ is the category of commutative  differential graded algebras (CDGAs).
Hence for
a CDGA $A$ we get a commutative {algebra spectrum} $HA\in \CAlg(\Mod_{H\Z})$. We consider $\R$ as a CDGA concentrated in degree $0$ and get the commutative algebra $H\R$. We write $\Ch_{\R}:=\Mod_{\R}$ for the category of chain complexes of real vector spaces. Note that $\CAlg(\Ch_\R)$ is the category of real commutative  differential graded algebras (CDGAs).
The composition of the localization and the Eilenberg-Mac Lane equivalence restricts to a functor
$\CAlg(\Ch_{\R})\to \CAlg(\Mod_{H\R})$.

\begin{ddd}\label{ulimar1501} 
A differential  refinement   of  {a commutative  ring  spectrum} $R$ is a triple $(R,A,c)$
 consisting of a CDGA  $A$ over $\R$ together with an equivalence
\begin{equation*}
 c: R\wedge H\R \iso HA.
\end{equation*}
in $\CAlg(\Mod_{H\R})$.  
\end{ddd}

\newcommand{\hRings}{\widehat{\mathrm{Rings}}}

It has been shown in  \cite{MR2306038} that one can model every $H\R$-algebra by a CDGA. 
In our language this means that the functor
$$\CAlg(\Ch_{\R})\to \CAlg(\Mod_{H\R})\ ,  \quad A\mapsto HA$$
is essentially surjective. %\cite{MR2306038}.\tfoot{In der Quelle wird gezeigt, dass CDGAs alle H$\R$-Moduln modelieren, also ist die Logik hier etwas komisch} 
In particular,  every commutative ring spectrum admits
a differential refinement.

\begin{ddd}\label{may0101}
We define the $\infty$-category of {commutative} differential ring spectra as the pullback
$$\xymatrix{\hRings\ar[rr]\ar[d]&&\CAlg(\Ch_{\R})\ar[d]^{H}\\
 \CAlg(\Sp)\ar[rr]^{\wedge H\R}&&\CAlg(\Mod_{H\R})}$$
\end{ddd}
By construction, the objects of $\hRings$ are differential refinements of commutative ring spectra.

\begin{rem2}\label{apr3045}
 In    many relevant cases
 there exist a {differential refinement of $R$} whose underlying  CDGA   is the graded ring    $\pi_{*}(R)\otimes \R$ with trivial differentials.   {In this case}  
$R \wedge H\R$ is called  formal. For example, if    
  $\pi_{*}(R)\otimes \R$  is free as a commutative $\R$-algebra or arises from a free algebra by inverting elements, then $R \wedge H\R$ is formal. 
  In this case there is an equivalence $c$ which is {uniquely determined  up to homotopy by the property that it induces the canonical identification on homotopy groups {\cite[Sec. 4.6]{skript}.} }
  %\tfoot{Was zitierst du denn hier?}
 \begin{exa2}
\begin{enumerate}
\item
For the {Eilenberg-MacLane spectrum  $H\Z$ of $\Z$  we choose   a real model whose underlying}   CDGA  is      $\R$    concentrated in degree $0$.  \item
For the complex $K$-theory  spectrum  $\KU$   we have a  real model  with underlying  CDGA    given by  $\R[b,b^{-1}]$ with trivial  differentials, where $\deg(b) = -2$. The isomorphism
$c$ is the Chern character.
\item
For the connective spectrum $tmf$ of  topological modular forms  we  can choose a real model with underlying  CDGA  given by $ \R [c_4, c_6]$  with trivial  differentials, where  $\deg(c_4) = 8$ and $\deg(c_6) = 12$. Again, the map $c$ is unique up to equivalence. 
%\nuli{Similarly for the connective version Tmf.}
\end{enumerate}
%\marginpar{\uli{All examples are formal and the  is the canonical one. Should we present a non-canonical example? }\tho{Do we have one?}}
\end{exa2}
Note that there exist non-formal examples. 
% \footnote{\uli{We definitly should give an example. Start from a non-formal $CDGA$.} \tho{Should we really? Is this important?}}
\end{rem2}

The sheaf of smooth real differential forms with the de Rham differential is  a {sheaf of CDGA's on the site of smooth manifolds, i.e. } an object 
${\Omega^*} \in \Sh_{\CAlg(\Ch_{\R})}$  
which associates to every smooth manifold $M$ the smooth real de Rham complex ${(\Omega^*(M),d)}$. For a  CDGA  $A$ over  $\R$  we define
 {the} sheaf of CDGA's of    differential forms with values in $A$  by $$\Omega A : = \Omega\otimes_{ \R } \underline{A}\in \Sh_{\CAlg(\Ch_{\R})}\ .$$ where $\underline{A}$ is the constant sheaf  of CDGA's associated to $A$ (see Definition \ref{apr2630} and Proposition \ref{propju22})
 %\footnote{\uli{Verstehe nicht, warum der Sprung nicht funktioniert!}}.  }
 
In general, a presheaf {with values in the 1-category of} chain complexes $\cM\in \PSh_{\Ch} $ is a sheaf {precisely} if for every {$n\in \Z$} its degree-$n$ component $\cM_{n}$ is a sheaf of abelian  groups. This property can easily be checked in the case of $\Omega$ or $\Omega A$. 

{
We now extend the localization $\iota:\Ch\to \Ch_{\infty}$ to presheaves by post-composition.
Let us consider a sheaf $\cM\in \Sh_{\Ch}$. Then in general 
$\iota(\cM) $ is only a presheaf with values in $\Ch_{\infty}$.   But if $\cM$ is a complex of modules over the sheaf of  rings of smooth real-valued functions on $\Mf$, then it is a sheaf. The argument
 employes the existence of smooth partitions of unity. In particular, $\iota(\Omega)$ and $\iota(\Omega A)$ are sheaves. More generally, the objectwise localization
$\Mod_{\Omega A}\to \Mod_{\iota(\Omega A)}$ preserves sheaves.}

\begin{exa2}
On the other hand,  consider for example the sheaf of complexes $\underline{\R}$ {concentrated  degree zero  {given by locally constant $\R$-valued functions}}.
%\footnote{\tho{Warum $S^0$?}} 
  Then $\iota(\underline{\R})$ is not a sheaf but   $\iota (\underline{\R})\to \iota(\Omega)$ is the sheafification.
\end{exa2}

We now fix a smooth manifold $M$.
% Recall that $\Mod_{\Omega A}(M)$ the symmetric monoidal category of sheaves of sheaves of  $DG$-$\Omega A$-modules on the site $\Mf/M$. 
%Recall further that
%$\iota(\Omega A)\in \Sh_{\CAlg(\Ch_{\R,\infty})}$ denotes the image under localization, and
%that we have a localization $\iota:\Mod_{\Omega A} (M)\to \Mod_{\iota(\Omega A)}(M)$ on the level of modules.
A morphism between objects in $\Mod_{\Omega A} (M)$ becomes an equivalence under the localization
$\iota:\Mod_{\Omega A}(M)\to \Mod_{\iota(\Omega A)}(M)$ if and only if it is a quasi-isomorphism between  complexes of sheaves.
\begin{ddd}
\begin{enumerate}
\item
A sheaf $\cM \in  \Mod_{\Omega A} (M)$   is called $K$-flat if the   functor 
\begin{equation*}
\cM \otimes_{\Omega A|_{{M}}} \dots :  \Mod_{\Omega A}(M)\to
\Mod_{\Omega A}(M)
\end{equation*}
preserves quasi-isomorphisms \cite{MR1465117}.
\item A sheaf $\cM \in  \Mod_{\Omega A} (M)$ 
 is called invertible, if there is an object $\cN \in \Mod_{\Omega A}(M)$
such that $\cM \otimes_{\Omega A|_{{M}}} \cN$ {is isomorphic to}  $\Omega A$. 
\item  {A sheaf $\cM \in  \Mod_{\Omega A} (M)$ is called weakly locally constant if $\iota(\cM)$ is locally constant.}
%\item 
%\uli{By $\Mod^{h}_{\iota(\Omega A)}\subset \Mod_{\iota(\Omega A)}$ we denote the full subcategory of $\iota(\Omega A)$-modules
%which are homotopy invariant. 
%\uli{ A sheaf $\cM\in \Mod_{\Omega A}(M)$  is called weakly homotopy invariant if $\iota(\cM)$ is homotopy invariant.}
% \in \Mod^{h}_{\iota(\Omega A)}(M)$. }
%\item \uli{A sheaf $\cM\in \Mod_{\Omega A}(M)$ is called weakly cartesian, if $\iota(M)$ is cartesian.}
\item
{By $\PicwLf_{\Omega A}\in \Sh_{\CGrp(\cS)}$ we denote}
%By $$\Picfcc_{\Omega A}\subseteq \Picfc_{\Omega A}\subset  \Picf_{\Omega A}\in \Sh_{\sym}$$}we denote
 %\marginpar{\uli{Actung: Das stimmt nicht. Es gilt nur $\Picf_{\Omega A}\in \PSh_{\sym}$}} 
 the Picard-$\infty$-stack (actually a Picard-$1$-stack)    which associates to every  manifold $M$ the {Picard-1-groupoid} $\PicwLf_{\Omega A}(M)$   of invertible, $K$-flat and weakly locally constant objects in $\Mod_{\Omega A}(M)$.   %HOMOLOGICAL ALGEBRA OF HOMOTOPY ALGEBRAS
%VLADIMIR HINICH
\end{enumerate}
\end{ddd}

\begin{exa2}
{Note that $\Omega A$ is weakly  locally constant. In fact, it is weakly constant since we have an equivalence
$\underline{\iota(A)}\iso \iota(\Omega A)$
}
%r and clearly weakly cartesian and cartesian.}
 The groupoid ${\PicwLf_{\Omega A}(M)}$ {is therefore well-defined and} contains all modules
that are shifts of $\Omega A_{ | M} $, i.e. modules of the form $\Omega A[n]_{ | M }$ for $n \in \mathbb{Z}$. 
But $\Omega A[n]_{  |M }$ and $\Omega A[m]_{ | M }$ can be isomorphic if $A$ is periodic, i.e. contains a unit in degree $m-n$. 
\end{exa2}

\begin{rem2}
  Because of the additional conditions of $K$-flatness and {weakly locally constantness} in general the inclusion
$\Picf_{\Omega A}\subseteq  \Pic_{\Omega A}$ 
 is proper as the following example shows. %\footnote{\uli{I do not have an example of a of a $K$-flat homotopy invariant but not weakly cartesian sheaf.}}

We shall give an example of a weakly locally constant   non $K$-flat object in $\Pic_{\Omega A}$. 
Let $L_{\lambda}\to S^{1}$ be the flat one-dimensional real vector bundle with connection $\nabla^{\lambda}$ with holonomy $\lambda\in \R^{{>0}}$ and $\Omega(S^{1},L_{\lambda})$ be the complex
of  smooth forms with coefficients in $L$.  It is a module over the de Rham complex $\Omega(S^{1})\cong \Omega(S^{1},L_{1})$, and we have the rule
$$\Omega(S^{1},L_{\lambda})\otimes_{\Omega(S^{1})} \Omega(S^{1},L_{\lambda^{\prime}})\cong \Omega(S^{1},L_{\lambda\lambda^{\prime}})\ .$$
In particular, $\Omega(S^{1},L_{\lambda})\in \Pic_{\Omega(S^{1})}${,} and its tensor inverse is {given by}
$\Omega(S^{1},L_{\lambda^{-1}})$.
 The complex $\Omega(S^{1},L_{\lambda})$ can be identified with
$$C^{\infty}(S^{1})\xrightarrow{f\mapsto f^{\prime}-\log(\lambda)f} C^{\infty}(S^{1})\ .$$
We have
$$H^{*}(\Omega(S^{1},L_{\lambda}))\cong \left\{\begin{array}{cc} \R\oplus \R[1]&\lambda=1\\
0 &\lambda\not=1\end{array}\right. \ .$$

If $\lambda\in \R^{{>0}}\setminus \{1\}$, then the morphism $0\to \Omega(S^{1},L_{\lambda})$ is a quasi-isomorphism.
Its tensor product with $\Omega(S^{1},L_{\lambda^{-1}})$
is $0\to \Omega(S^{1})$ which is not a quasi-isomorphism. 
It follows that $\Omega(S^{1},L_{\lambda^{-1}})$ is not $K$-flat.
{{Thus} if we set $A:=\Omega(S^{1})$, {then}
$\cM:=\Omega\otimes_{\R} \underline{\Omega(S^{1},L_{2})}\in \Pic_{\Omega A}$ is weakly locally constant but not $K$-flat. %\footnote{\uli{I am not sure that we can make an example of a not weakly homotopy invariant sheaf. Also I do not have an example of a non-cartesian invertible sheaf.}}
}
\end{rem2}
  %Note that the category of CDGAs is the same as the category $\CAlg(\Ch)$ of commutative algebras in chain complexes. We let $\Ch_{\infty}$ denote the symmetric monoidal $\infty$-category
%obtained by localization of $\Ch$ at quasi-isomorphisms. We get an induced map
%$\CAlg(\Ch)\to \CAlg(\Ch_{\infty})$, $A\mapsto A_{\infty}$. We use a similar notation for sheaves with values in these categories.\\
% We let $\Pic_{H\Omega A}\in \Sh_{\sym}(\Mf)$ denote the symmetric monoidal stack of $\Mf$ which associates to each manifold $M\in \Mf$ the symmetric monoidal groupoid $\Pic_{H\Omega A}(M)$  of invertible objects in symmetric monoidal $\infty$-category $ \Mod_{H\Omega A}( M)$   
We observe that  for every manifold $M$
 the transformation
 \begin{equation}\label{ulimar1820}\PicwLf_{\Omega A}(M)\to \Mod^{\loc}_{\iota(\Omega A)}(M)  \ ,    \end{equation}
is symmetric monoidal (and not just lax symmetric monoidal). 
More concretely, the condition of $K$-flatness ensures that the tensor product of sheaves of complexes is equivalent to the derived tensor product.
 Hence it preserves invertible objects.
If we further compose \eqref{ulimar1820} with the symmetric monoidal  Eilenberg Mac-Lane equivalence $ \Mod^{\loc}_{ \iota(\Omega A)  } \stackrel{\sim}{\to}  \Mod^{\loc}_{H \Omega A} $
then we obtain a {symmetric monoidal} transformation 
\begin{equation}\label{apr2701}
 H : \PicwLf_{\Omega A} \to \PicL_{H \Omega A}
\end{equation}
%o the $\infty$-groupoid of invertible sheaves of $H \Omega A$-module spectra over $M$. 
%\nuli{Here we abuse notation and identify 
% the 1-category $\Picf_{\Omega A}(M)$ with the associated $\infty$-category given by its nerve $N\Picf_{\Omega A}(M)$.}\marginpar{\uli{Generelle conventions festlegen und dann konsequent durchziehen.}}

%\uli{By construction, the map \eqref{apr2701} restricts to a map of stacks of cartesian objects 
%$$\Picfc_{\Omega A}(M)\to \Pic^{h,\cart}_{H\Omega A}(M)\ . $$ }

\bigskip

 The de Rham equivalence is a canonical equivalence of sheaves
 \begin{equation}\label{ulimar2001}  H\Omega A \stackrel{\sim}{\to}   \underline{HA} \end{equation}
 in $\CAlg(\Mod_{H\R})$, {see \cite[Sec. 4.6]{skript}.}
 %The de Rham lemma primplies that the sheaf $H\Omega A$ is equivalent to the
%constant sheaf $\underline{HA}$ of $H\R$-algebra spectra induced from $HA$. 

\bigskip
{
Let us fix a commutative ring spectrum $R$ and a differential refinement $\hat R=(R,A,c)$.}
 If we compose the  inverse of the de Rham equivalence {\eqref{ulimar2001}} with the sheafification of the 
equivalence  $c: R \wedge H \R \to HA$ we obtain the  equivalence  \begin{equation*}
  \underline{R} \wedge H\R \iso H\Omega A. 
\end{equation*}
in $\CAlg(\Mod_{H\R})$. {Using this equivalence we will consider
$\underline{R}\wedge H\R$ as an object in $\PicL_{H\Omega A}$.}
 %\ntho{We get a  canonical filler  
%$\kappa:E\to \underline{R}\wedge H\R$ of the diagram
%$$\xymatrix{\PicL_{\underline{R}}\ar[r]^{\dots\wedge H\R}\ar[d]&\PicL_{H\Omega A}\ar[d]\\
%\Sh_{\Sp}\ar@{=}[r]&\Sh_{\Sp}}\ .$$} \marginpar{\uli{Nicht streichen, sondern Nummer geben und im Beweis von 5.12 darauf verweisen.}}

 \begin{ddd}\label{uliapr0701}
A  differential $R$-twist  over a manifold $M$ is  a triple $(E,\cM,d)$ consisting of a topological $R$-twist 
$E \in \PicL_{\underline{R}}(M)$, an  object $\cM \in \PicwLf_{\Omega A}(M)${,} 
and an equivalence
\begin{equation*}
{d: E\wedge H\R \iso H\cM} 
\end{equation*}
{in $\PicL_{H\Omega A}(M)$.} We call  the triple $(E,\cM,d)$   a     differential refinement  of the   topological twist $E$.
\end{ddd}

 \begin{ddd}\label{may0620}
 The $\infty$-stack of differential $R$-twists is defined  by the  pullback  in $\Sh_{\CGrp(\cS)}$ 
\begin{equation}\label{ulimar1704}
\xymatrix{\Tw_{\hat R}\ar[d]\ar[rr]&& \PicwLf_{\Omega A}\ar[d]\\ \PicL_{\underline{R}}\ar[rr]^{{ \wedge H\R}} &&\PicL_{H\Omega A}\ .  }\end{equation}
\end{ddd}
  By construction, the objects of ${\Tw_{\hat R}(M)}$  are differential $R$-twists over $M$.

  %\uli{The definition of twisted differential cohomology theory makes sense without the cartesian assumption. But for cartesian differential twists we have good classification and  existence results.}

\bigskip 
 
% \begin{remark}
%We will provide geometric examples of differential structures on twists in the next sections. We will also show that (almost) all topological twists \marginpar{\uli{zu vage}}
%$E$ admit differential structures and that these structures are unique up to isomorphism.
%\end{remark}
% \begin{lem}
% Every $R$-twist $E \in \Pic_E(M)$ admits a refinement to an abstract differential $R$-twist $(E,\cM,d) \in \Pic_{\hat E}(M)$.
% \end{lem}
% \begin{proof}
% Need to use model theories and results from Shipley or Keller.
% \end{proof}

Now let $(E,\cM,d)$ be a differential twist of a differential spectrum $(R,A,c)$ over $M$. 
 By $\cM^{\ge 0}$ we denote the naive truncation of the sheaf of complexes  to non-negative degrees. We have a  canonical inclusion $i: \cM^{\geq 0} \to \cM$.  {Furthermore, let $\kappa:E\to E\wedge H\R$ be the canonical map.}
\begin{ddd}\label{ulimar1901}
The twisted differential function spectrum {for a differential twist  $(E,\cM,d)$ }is {the} sheaf of spectra defined by the pull-back in $\Sh_{\Sp}(M)$
$$\xymatrix{\Diff(E,\cM,d)\ar[r]\ar[d]&H\cM^{\ge 0}\ar[d]^{i}\\
E\ar[r]^{d\circ \kappa}&H\cM}\ .$$
\end{ddd}
In the square above all objects are considered as sheaves of spectra, i.e.
we omitted to write the {forgetful}  functors 
$\Mod_{\underline{R}}(M)\to \Sh_{\Sp}(M)$ and $\Mod_{H \Omega A}\to \Sh_{\Sp}(M)$  explicitly.\\

 {Note that pullbacks in sheaves of spectra are computed objectwise, i.e. we get a similar pullback after evaluation at $M$ or more generally every manifold over $M$.}  
In order to properly describe the functorial properties of the twisted differential function spectrum we use the same procedure as in the case of the classical twisted cohomology. {Roughly speaking we consider the functor that assigns to a differential twist $(E,\cM,d)$ over $M$ the spectrum $\Diff(E,\cM,d)(M)$ thus defining  a graded ring spectrum.
{
\begin{prop}\label{ulimar2402}
The  construction \ref{ulimar1901} of the twisted differential function spectrum refines to a transformation
$${\hat \cR}\in \Sh_{\GCS}\ , \quad M \mapsto ({\Tw}_{\hat{R}}(M), \Diff(\dots)(M))\ .$$ \end{prop}
}
\begin{proof}
We again consider the functor
$F$ introduced in \eqref{ulimar1401}. We set $$G:={\Tw_{\hat R}}\in \Sh_{\sym}\ .$$

 In order to apply Proposition  \ref{lem_drei} we must construct a lax symmetric monoidal transformation
\begin{equation}\label{ulimar1402}\eta: {\Tw_{\hat R}}\to \Sh_{\Sp}\end{equation}
such that $\eta(M)(E,\cM,d)\cong \Diff(E,\cM,d)\ .$ 
%Therefore we want to show that there 
%is a (lax) \thomas{hinten verbessern} functor 
%\begin{equation}\label{toconstruct}
 %\Pic_{\hat{R}} \to \PSh_{\Mod_{r}}(\Mf/\dots)
%\end{equation}
%which assigns the sheaf of differential sections. 
Therefore let 
$J$ be the indexing category for pullback diagrams, i.e. $J = \Delta[1]
\cup_{\Delta[0]} \Delta[1]$. The limit  {functor}  
\begin{equation*}
\lim_{J}: \Sh_{\Sp}^J \to \Sh_{\Sp}
\end{equation*}
 admits a  lax symmetric monoidal refinement.  {This refinement is constructed using that $\lim_{J}$ is right adjoint to the diagonal functor which admits a canonical monoidal structure. Then we employ the fact that the right adjoint to a monoidal functor inherits a lax monoidal structure as shown in \cite[Corollary 8.3.2.7]{HA}.} Thus in order to construct the transformation \eqref{ulimar1402} it suffices to construct a symmetric monoidal transformation
\begin{equation*}
{\Tw_{\hat{R}}} \to \Sh_{\Sp}^J.
\end{equation*}

By definition we have the following pullback decompositions:
\begin{eqnarray*}
{\Tw_{\hat{R}}} &{\cong}&  \PicwLf_{\Omega A} \times_{\PicL_{H\Omega A}} \PicL_{{\underline{R}}} \\
\Sh_{\Sp}^J &{\cong}& \Sh_{\Sp}^{\Delta[1]} \times_{\Sh_{\Sp}^{\Delta[0]}} \Sh_{\Sp}^{\Delta[1]}
\end{eqnarray*}
Therefore it suffices to write down a transformation {between these} diagrams, i.e. lax symmetric monoidal transformations  
$$f_1: \PicwLf_{\Omega A} \to \Sh_{\Sp}^{\Delta[1]}\ , \quad  f_2: \PicL_{{\underline{R}}} \to  \Sh_{\Sp}^{\Delta[1]}\ ,  \quad 
f_0: \PicI_{H\Omega A} \to \Sh_{\Sp}^{\Delta[0]}$$ and fillers.
The transformation $$f_{0}: \PicL_{H\Omega A} \to \Mod_{H\Omega A}\to \Sh_{\Sp}$$
is a composition of lax symmetric monoidal   transformations.
$$ f_1(\cM) := (H\cM_{\geq 0} \to H\cM) \ , \quad  \cM\in \PicwLf_{\Omega A}(M)$$
and
 $$f_2(E) := (\kappa:E\to {E\wedge H\R})   \ , \quad  E\in {\PicL_{\underline{R}}}(M)\ ,$$
 where we again omit two write various forgetful functors to $\Sh_{\Sp}$.
\end{proof}

\begin{ddd}
Given a manifold $M$ with a differential twist $(E,\cM,d)$ we  define the  twisted differential cohomology 
group  by 
%\marginpar{\uli{Klammern bevorzugt!}}

\begin{equation*}
 \hat R^{(E,\cM,d)}(M) := \pi_0(\Diff(E,\cM,d)(M))  \end{equation*}
\end{ddd}

\begin{kor}\label{cor_june26} For each manifold $M$ the twisted differential cohomology $\hat R^*(M)$ forms a graded
ring, which is graded over the homotopy category $\ho {\Tw_{\hat{R}}}(M)$. 
Moreover the assignment  
%\marginpar{\uli{Hier bemerken, da{\ss} das insbesondere bedeutet, da{\ss} automorphisms von twists wirken!}}
\begin{equation*}
M \mapsto \big(\ho{\Tw_{\hat R}}(M), \hat R^*(M)\big) 
\end{equation*}
forms a functor $\Mf^{op} \to \mathrm{GrRing}$.
\end{kor}
{\begin{theorem}
The construction of the twisted differential function spectrum refines to  
a functor
$$\hRings\to \Sh_{\GCS}\ .$$
\end{theorem} 
}
{
\begin{proof}
 The proof works similar to the proof of Proposition \ref{prop_jun2ee}. We consider the functors 
$$ F:  \hRings\times \Mf^{op} \xrightarrow{\text{pr}_2} \Mf^{op} \xrightarrow{\eqref{ulimar1401}} \GCS$$
and   
$$G: \hRings \times \Mf^{op} \to \sym \ , \quad (\hat R ,M) \mapsto  \Tw_{\hat R}(M)\ ,$$ 
 and the  lax symmetric monoidal transformation $\eta: G\to  \Sh_{\Sp} $ constructed similar to the one in the proof of Propositon \ref{ulimar2402} \end{proof} }

%\begin{remark} \marginpar{\uli{Ist mir wieder zu vage bzw macht die sch\"one Pr\"azision weiter oben kaputt. W\"urde ich hier nicht so sagen.
%}}
%The Proposition in particular shows that for equivalent differential twists $(E,\cM,d)$ and $(E',\cM',d')$ the groups 
%$\hat R^{E,\cM,d}(M)$ and $\hat R^{E',\cM',d'}(M)$ are abstractly isomorphic.
%An isomorphism is induced by every isomorphism of differential twists. 
%Therefore is suffices to consider the twisted 
%differential cohomology groups for one representative in every equivalence class of twists. 
%We will construct nice representatives of differential twists in the next section. 
%We will even show under mild conditions that two pairs $(E,\cM,d)$ and $(E',\cM',d')$ are equivalent precisely if $E$ and $E'$ are equivalent.
%\end{remark}

%%%%%%%%%%%%%%%%%%%%%%%%%%%%%%%%%%%%%%%%%%%%%%%%%%%%%%%%%%%%%%%%%%%%%%%%%%%%%%%%
%%%%%%%%%%%%%%%%%%%%%%%%%%%%%%%%%%%%%%
\section{Properties of twisted differential cohomology}\label{sep0602}
%%%%%%%%%%%%%%%%%%%%%%%%%%%%%%%%%%%%%%%%%%%%%%%%%%%%%%%%%%%%%%%%%%%%%%%%%%%%%%%%
%%%%%%%%%%%%%%%%%%%%%%%%%%%%%%%%%%%%%%%%
%\thomas{Das Kapitel kann auch weg oder in den Appendix, habs jetzt nur mal schnell geschrieben.}
In this section we list the basic properties of differential cohomology. 
The more advanced properties, such as orientation and integration theory will be discussed elsewhere.\\

The following list of structures just decodes the fact that $\Diff$ is a sheaf of graded ring spectra
and correspond to similar structures in the topological case listed in Section \ref{ulimar2201}. We fix a commutative ring spectrum $R$ and {a} differential refinement $\hat R$.

\begin{enumerate}
\item For every differential twist $(E,\cM,d) \in  {\ho}\Tw_{\hat R}(M)$ we have an abelian group $\hat R^{(E,\cM,d)}(M)$.
\item For every pair $(E,\cM,d),(E',\cM^{\prime},d^{\prime}) \in  {\ho} \Tw_{\hat R}(M)$ {of differential twists we have} a multiplication morphism 
 \begin{equation*}
\cup: \hat R^{(E,\cM,d)}(M) \otimes \hat R^{(E',\cM',d')}(M) \to \hat R^{(E,\cM,d)+(E',\cM',d')}(M) ,
 \end{equation*}
 where $(E,\cM,d)+(E',\cM',d')$ denotes the symmetric monoidal pairing in $ {\ho}\Pic_{ \hat R}(M)$. The multiplication is associative, commutative, and has a unit in the obvious sense.
\item For every     morphisms $(E,\cM,d)\to (E',\cM',d')$ in $ {\ho}\Tw_{\hat{R}}(M)$ we have  {an iso}morphism 
 \begin{equation*}
 R^{(E,\cM,d)}(M) \iso R^{(E',\cM',d')}(M) .
 \end{equation*}
 This assignment is compatible with composition and   the product in the obvious sense.
\item For every smooth map $f: M \to N$ and $(E,\cM,d) \in  {\ho}\Pic_{\hat{R}}(N)$ we have a morphism
 \begin{equation*}
 R^{(E,\cM,d)}(N) \to R^{f^*(E,\cM,d)}(M)
 \end{equation*}
 which is compatible with all the other maps.
 %\item All these structures are natural with respect to smooth maps $M\to M^{\prime}$.
 \end{enumerate}
 
 The following structures are typical for differential cohomology and immediately follow from the definition of the twisted differential function spectrum by a pull-back.  In the following proposition, if not said differently, $M$ is a smooth manifold and $(E,\cM,d)\in \Tw_{\hat R}(M)$ is a differential twist.

 \begin{proposition}\label{may0630}
\begin{enumerate}
\item We have a map of sheaves of graded rings spectra
$$I:\widehat \cR\to \cR\ .$$ It induces  a map of abelian groups
$$\widehat R^{(E,\cM,d)}\to R^{E}$$
which is compatible with the products, the action of morphisms between differential  twists and smooth maps between manifolds   in the natural sense. 
 \item
 We have a natural transformation
 $$\curv:\widehat{R}\to \cZ\ ,$$
 where the presheaf of graded rings $\cZ\in \PSh_{\mathrm{GrRing}}$ is given by $$M\mapsto (\PicwLf_{\Omega A}(M)\in \cM\mapsto \cZ^{\cM}(M):=Z^{0}(\cM(M)))\ .$$ It induces   a map of abelian groups
$$\curv: \widehat R^{(E,\cM,d)}(M)\to Z^{0}(\cM(M))$$
which is compatible with the product, the action of morphisms between differential twists and smooth maps between manifolds. 
\item   There is a canonical homomorphism
 \begin{equation*}
 a:   \cM^{-1}(M)/\Im(d) \to  \widehat R^{(E,\cM,d)}(M)
 \end{equation*}
 which is
natural in $M$ and such that it becomes a homomorphism of $\Tw_{\hat R}(M)$-graded groups (not rings!). Moreover we have the rule
$$a(\omega)\cup x=a(\omega\cup \curv(x))$$
in $\hat R^{(E,\cM,d)+(E',\cM',d')}(M) $
for $\omega \in   \cM^{-1}(M)/\Im(d)$ and $x\in \widehat R^{(E',\cM',d')}(M)$.
 \item   The following diagram commutes:
 \begin{equation*}
\xymatrix{
\widehat R^{(E,\cM,d)}(M) \ar[d]^I \ar[r]^{\curv} & Z^{0}(\cM(M)) \ar[d]  \\ R^{E}(M)
 \ar[r] &  H^{0}(\cM(M)) 
} ,
\end{equation*} where the lower horizontal map is induced by  the map
$E{\stackrel{\kappa}{\to}} E\wedge H\R\stackrel{d}{\to} H\cM$.
%Here $H^\cM_{dR}(M,A) :=  H^0\Gamma(\cM; M) = \pi_0\Gamma(H\cM; M)$ and $d_*$ is the morphism induced from $d: E \wedge H\R \to H\cM$.
\item We have the equality $$\curv \circ a = d$$ 
of maps $ \cM^{-1}(M)/\im(d)  \to Z^{0}(\cM(M))$ .
\item The following sequences are exact:
\begin{align*}
R^{E-1}(M) \to \cM^{-1}(M)/\Im(d) \to  \widehat R^{(E,\cM,d)}(M) \to R^E(M) \to 0 \\[0.4cm]
R^{E-1}(M) \to  H^{-1}(\cM(M)) 
\to  \widehat R^{(E,\cM,d)}(M) \to Z^{0}(\cM(M)) \times_{ H^{0}(\cM(M))  } R^{E}(M) \to 0
\end{align*}
\end{enumerate}

 \end{proposition}
\begin{proof}
This follows as in the untwisted case from the construction of $\widehat{R}^{(E,\cM,d)}$ as a pull-back.
For more details we refer to the corresponding results in \cite{bg}, \cite{skript}.
\end{proof}
 
The fact that $\widehat{R}^{(E,\cM,d)}$ is a sheaf implies a Mayer-Vietoris sequence for twisted differential cohomology.   By abuse of notation we denote $\pi_{0}$ of  the evaluation of the sheaf
$E\wedge M\R/\Z$ (which is not a twist) on $M$ by $R^{E\wedge M\R/\Z}(M)$, where $M\R/\Z$ is the Moore spectrum of the abelian group $\R/\Z$.}

\begin{proposition}[Mayer-Vietoris]
Let $M=U\cup V$ be a decomposition of a smooth manifold as a union of two open subsets and $(E,\cM,d)\in \Tw_{\hat R}(M)$ be a differential twist. Then we have the long exact sequence 
%\marginpar{\uli{ Ich w\"urde die Einschr\"ankungen in den Exponenten streichen. So gehts aber auch.}}
{
\begin{equation*}
\xymatrix{
  &
 ... \ar[r] & 
{  R^{(E-{2})\wedge M\R/\Z|_{U}}(U)\oplus    R^{(E-{2})\wedge M\R/\Z|_{V}}(V)}  \ar[r]& 
  R^{(E{-2})\wedge M\R/\Z}({U\cap V})\ar`r[d]`[l]`[llld]`[dll][dll]  \\
  &\hat R^{(E,\cM,d)}(M) \ar[r] & 
  \hat R^{(E,\cM,d)_{|U}}(U)\oplus \hat R^{(E,\cM,d)_{|V}}(V) \ar[r]& 
\hat R^{(E,\cM,d)_{|U\cap V}}(U\cap V) \ar`r[d]`[l]`[llld]`[dll][dll] \\ 
 & R^{E+1}(M) \ar[r] &   R^{(E+1)|_U}(U) \oplus R^{(E+1)|_V}(V)  \ar[r]& ...
 }
\end{equation*}
}
%
%$$\hspace{-1cm}
%%\begin{align*}
%\dots \to R^{E\wedge M\R/\Z}(M)  \to \hat R^{(E,\cM,d)}(M)\to \hat R^{(E,\cM,d)_{|U}}(U)\oplus \hat R^{(E,\cM,d)_{|V}}(V)\to \hat R^{(E,\cM,d)_{|U\cap V}}(U\cap V)\to  R^{E+1}(M)\to\dots 
%$$
%%\end{align*}
 \end{proposition}
Finally we explain the consequence of the fact that
the construction $$\widehat{\Rings}\ni \hat R\mapsto \widehat{\cR}\in \Sh_{\GCS}$$ is functorial.
Let $f:R\to R^{\prime}$ be a morphism of commutative ring spectra. Assume that we have chosen a lift of this morphisms to differential refinements $$\hat f:\hat R=(R,A,c)\to  (R^{\prime},A^{\prime},c^{\prime})= \hat R^{\prime}\ .$$ Then we have a morphism of sheaves {of} graded commutative ring spectra $\widehat{f}:\widehat{\cR}\to \widehat{\cR}^{\prime}$. Let us spell out some pieces of this structure.
If $(E,\cM,d)\in \Tw_{\hat R}$ is a differential twist of $\hat R$, then we get a differential twist
$\widehat{f}_{*}(E,\cM,d)=:(E^{\prime},\cM^{\prime},d^{\prime})$. In greater detail,
$$E^{\prime}\cong E\wedge_{\underline{R}}\underline{R^{\prime}} \ ,\quad 
\cM^{\prime} \cong \cM\otimes_{\Omega A} \Omega A^{\prime}\ .$$
We further get an induced map
$$\widehat R^{(E,\cM,d)}(M)\to \widehat{R}^{\prime (E^{\prime},\cM^{\prime},d^{\prime})}(M)$$
which is compatible with corresponding maps on the level of forms and underlying twisted cohomology classes. Furthermore, these maps preserve the products and are functorial and natural with respect to 
morphisms of differential twists and smooth manifolds. We refrain from writing out these details.

\begin{exa2}
Here is a typical example. We consider the Atiyah-Bott Shapiro orientation $f: MSpin^{c}\to \KU$.
Since both spectra are formal we can take canonical differential refinements. We can further choose a differential ABS-orientation  $\hat{f}:\widehat{MSpin^{c}}\to \widehat{\KU}$.   In this way
we obtain a transformation
$$\widehat{A}:\widehat{\mathcal{MS}pin^{c}}\to \widehat{\mathcal{{K}}}$$
of sheaves of graded ring spectra.
 \end{exa2}

\part{De Rham models}

%%%%%%%%%%%%%%%%%%%%%%%%%%%%%%%%%%%%%%%%%%%%%%%%%%%%%%%%%%%%%%%%%%%%%%%%%%%%%%%%
%%%%%%%%%%%%%%%%%%%%%%%%%%%%%%%%%%%%%%
\section{Twisted de Rham complexes}\label{sep0602nnnn}
%%%%%%%%%%%%%%%%%%%%%%%%%%%%%%%%%%%%%%%%%%%%%%%%%%%%%%%%%%%%%%%%%%%%%%%%%%%%%%%%
%%%%%%%%%%%%%%%%%%%%%%%%%%%%%%%%%%%%%%%%
 
In Definition \ref{uliapr0701} we have introduced differential twists {on a smooth manifold $M$} as triples $(E,\cM,d)$ consisting of
a topological twist $E$ together with a sheaf $\cM$ of $K$-flat, invertible, and weakly  locally constant DG-$\Omega A$-modules and an   {equivalence $d:E\wedge H\R\iso H\cM $
of $H\Omega A$-modules.} In this section
we show how examples of such sheaves $\cM$  can be constructed from differential geometric data on $M$.

%\nuli{This is a first step towards the existence of differential refinements for a given topological twist.}

\bigskip

We consider a  CDGA $A$.
The following condition on  $A$ will play an important role below.

\begin{ddd}The CDGA $A$ is called split
 if the differential $d:A^{-1}\to A^{0}$ vanishes.
\end{ddd}

By $$\PicLf_{\underline{ A}}\in \Sh_{\CGrp(\cS)}$$ we denote {the Picard-{1} stack}
%\tfoot{Wir sind nicht konsistent mit den Bindestrichen bei Picard-stacks, das sollten wir am Ende noch vereinheitlichen. \uli{Habe kein anderes Vorkommen gefunden.}} 
which associates to a manifold $M$ the groupoid $\Picff_{\underline{ A}}(M)$
of  $K$-flat, {locally constant} sheaves of {invertible} $A$-modules on $\Mf/M$.  Our basic example of such a sheaf
is  $\underline{A}_{|M} $.

%\bigskip

%\uli{\begin{lem}\label{apr2730}
%A homotopy invariant sheaf of $A$-modules is
%automatically cartesian.
%\end{lem}
%\proof
%Let $X$ be a homotopy invariant sheaf of $A$-modules. We must show that
%$\tilde X\to X$ is an isomorphism (see Definition \ref{apr2720}). Since we are in a $1$-categorial setting
%we can check this on the level of stalks. Let $f:N\to M$ be an object in $\Mf/M$ and $n\in N$ be a point. 
%Then the stalk of $\tilde X$ at $n$ is given by
%$$\tilde X_{n}\cong \colim_{U} X(U)\ ,$$
%where $U$ runs over the poset of open subsets containing $f(n)$. 
%We can express the stalk of $X$ at $n$ as
%$$X_{n}\cong  X(n)\cong  X(f(n)) \cong   \colim_{U} X(U)\ ,$$
%where we use the homotopy invariance at the first isomorphism.
%The natural map $\tilde X_{n}\to X_{n}$ translates to the identity on the right-hand sides of these isomorphisms.\hB}

%\uli{
%We let $$\Picffc_{\underline{A}}\in \Sh_{\sym}$$ be the subsheaf $\Picff_{\underline{ A}}$ of  such that
%$\Picffc_{\underline{A}}(M)\subseteq \Picff_{\underline{ A}}(M)$ is a subgroupoid of cartesian objects. 
%}

%\uli{
%\begin{rem2}
%Here is an example of a non-cartesian object in 
%$\Picff_{\underline{A}}(M)$. \footnote{\uli{Hab noch keins}  }
 %\end{rem2}}

Note that ${\PicLf_{\underline{A}}(M)}$ is an ordinary groupoid.
In the following we calculate {the homotopy groups} $\pi_{i}({\PicLf_{\underline{A}}(M))}$ for $i=0,1$. For a manifold $M$
let $\Pi_1(M)$ denote the fundamental groupoid of $M$. The objects are points in
$M$, and the morphisms are homotopy classes of smooth path{s}%\footnote{\uli{Mehrzahl von path?}Das W√∂rterbuch sagt paths} 
in $M$. 
{By  definition the Picard groupoid  %\tfoot{Hier sieht man natuerlich die schwaeche unserere Notation, dass wir nicht $\Picfff_A(\Mf)$ schreiben}
$${\Picfff_{A}}:={\PicLf_{ \underline{A}}(*)}$$
is the groupoid of $K$-flat invertible $A$-modules.}
%\footnote{\uli{Das ist eigentlich eine zu beweisende Aussage. $\Picff_{\uli{\underline{A}}}(*)$ k\"onnte gr\"o{\ss}er sein.  
{Covering theory provides a} natural identification
$${\PicLf_{\underline{A}}(M)}\cong \Fun(\Pi_{1}(M),{{\Picfff_A}})\ .$$

\bigskip

  For every $X\in \Picfff_A$ we have a canonical identification of abelian groups $$\Aut_{\Picfff_A
  }(X)\cong Z^{0}(A)^{\times}\ .$$

  {To see this  first of all note that the action of $A$ on $X$ provides an identification  $A^{0}\cong \End_{A^{\sharp}}(X^{\sharp})$, where $(\cdots)^{\sharp}$ indicates the operation of forgetting the differential.
  The condition that $a\in A$ preserves the differential translates to $a\in Z^{0}(A)$. Finally, if $a$ induces an isomorphism, then 
     $a\in   Z^{0}(A)^{\times}$.
 }
  %{  
%For an abelian group $G$ let $BG$ denote the {Picard groupoid} with one object and morphisms $G$. } 
 %{The discussion above shows that \nuli{for a {split} CDGA $A$}
%\footnote{\uli{Verstehe ich jetzt gar nicht mehr. Split als Groupoid ja, aber nicht als Picardgr. Sehe nicht, wie das Differential $d:A^{-1}\to A^{0}$ √ºberhaupt eine Rolle spielt. Scheint ein ganz anderes Problem zu sein. $Ext(\pi_{0}(...),Z^{0}(A))=0$. }} 
 %   $$\Picfff_A\cong  \pi_{0}(\Picfff_A)\times B(Z^{0}(A)^{\times})\ .$$
%This yields the following assertion:

\begin{prop}\label{ulimar2160}
If $M$ is {a connected smooth manifold}  with a given base point $m \in M$, then we have {natural
isomorphisms}
\begin{equation*}
 \pi_0({\PicLf_{\underline{A}}(M))} \cong \pi_0({\Picfff_A}) \oplus {H^1(M;Z^{0}(A)^{\times})}% H(Z^{0}(A)^{\times})^{1} (M)  
\end{equation*} and $$\pi_{1}({\PicLf_{\underline{A}}(M)} ,X)\cong Z^{0}(A)^{\times}$$
for all $X\in {\PicLf_{\underline{A}}(M)}$.

\end{prop}
%Concretly this means that in order to write down a locally constant sheaf of $A$ modules one has to pick a K-flat $A$-module 
%$N$ over each connected component of $M$ and a representation of the fundamental group on $N$.
%
%\begin{corollary} \label{map}
%The functor $\cM$ induces a homomorphism
%\begin{equation*}
%\pi_0(\cM) : \pi_0({\Picf_A}) \oplus H^1(M,Z^0(A)^\times) \oplus F^2H^1(M;A)  \to \pi_0(\Picf_{\Omega A}(M)).
%\end{equation*} 
%\end{corollary}
%\begin{proposition}\label{prop_inj}
%The functor
%\begin{equation*}
%\cM: \Picf_\cA(M) \times \mathbb{Z}^1_{dR}(M,A) \to \Picf_{\Omega A}(M)
%\end{equation*}
%induces an isomorphism on $\pi_1$.\\
%\end{proposition}
%

{For { a sheaf}  $\cA\in \PicLf_{\underline{A}}(M)$ we define the sheaf of DG-$\Omega A_{|M}$-modules}
%a symmetric monoidal transformation
%$$\Picff_{\underline{A}}\to \Sh_{\sym}$$ by
$$
%\Picf_{\underline{A}}(M)\ni \cA\mapsto
 \cM(\cA):=
\Omega_{|M}\otimes_{\R} \cA \in \Mod_{\Omega A}(M)\ .$$
\begin{lem}\label{apr2601}
We have $\cM(\cA) {\in \PicwLf_{\Omega A}(M)}$.    \end{lem}
\begin{proof}
{We must show that $\cM(\cA)$ is \begin{enumerate} \item $K$-flat, \item invertible,  
\item and weakly locally constant.
%
 %\item weakly homotopy invariant, \item weakly cartesian, and \item{cartesian.} 
 \end{enumerate}}

\bigskip

{1.} We have a canonical isomorphism $\cM(\cA)\cong \Omega A_{|M}\otimes_{A}\cA$.
If $X\to Y$ is a quasi-isomorphism in  $ \Mod_{\Omega A}(M)$, then we can identify the induced {morphism}
$\cM(\cA)\otimes_{\Omega A_{|M}} X\to \cM(\cA)\otimes_{\Omega A_{|M}} Y$ with  
$\cA\otimes _{A}X\to \cA\otimes_{A}Y
$. Here we consider $X$ and $Y$ as sheaves of $A$-modules by restriction along the inclusion $\underline{A}\mapsto \Omega A_{|M}$ This shows that $K$-flatness of $\cA$ implies $K$-flatness of $\cM(\cA)$.  

\bigskip

{2.} If $\cA\in \PicLf_{\underline{ A}}(M)$, then there exists $\cB\in  \PicLf_{ \underline{ A}}(M)$ such that   $\cA\otimes_{ A}\cB \cong \underline{A}_{|M}$.  We see that then  $\cM(\cA)\otimes_{\Omega A_{|M}}\cM(\cB)\cong \Omega A_{{|M}}$. {Hence    $\cM(\cA)$ is invertible}. \\

{3.} Let $U\subseteq M$ be such that $\cA_{|U}\cong \underline{X}_{|U}$ for some $K$-flat invertible $A$-module $X$. Then we have $\cM(\cA)_{|U}\cong \Omega A_{|U}\otimes_{A}{\underline{X}_{|U}}\cong \Omega_{{|U}}\otimes_{\R} {\underline{X}_{|U}}$. The equivalence $\underline{\iota(\R)}\iso \iota(\Omega)$ induces the equivalence
${\underline{\iota(X)}_{|U}}\iso \iota(\Omega_{{|U}}\otimes_{\R} {\underline{X}_{|U}})$.
\end{proof}

We thus get a   symmetric monoidal transformation {of Picard stacks}
$$\cM:\PicLf_{\underline{A}}\to \PicwLf_{\Omega A}   \ .$$
%TThe tensor product over $\R$ is an exact functor.
%he decreasing filtration 
%$$\dots\subseteq F^{\ge 2}\Omega \subseteq F^{\ge 1} \Omega\subseteq F^{\ge 0}\Omega=\Omega $$
%by form degree induces a decreasing filtration of $\cM(\cA)$ by $K$-flat
%$\Omega A$-modules 
%$$\dots\subseteq F^{\ge 2}\cM(\cA) \subseteq F^{\ge 1} \cM(\cA)\subseteq F^{\ge 0}\cM(\cA)=\cM(\cA)\ . $$\\
We consider a manifold $M$ and an object {$\cA\in  \PicLf_{\underline{A}}(M)$}.   We use the symbol $d$ in order to denote the differentials of $\Omega A$ and of $\cM(\cA)$. {An} element  $\omega\in \Omega A(M)$  induces {a} multiplication  map $\omega:\cM(\cA)^{\sharp}\to \cM(\cA)^{\sharp}$ of modules over the sheaf of graded commutative algebra $\Omega A^{\sharp}_{|M} $ ({recall that    $\sharp$ stands for forgetting}  differentials). Furthermore we have the identity
$d\omega =[d,\omega]$. If $\omega \in \Omega A^{1} (M)$, then $\omega^{2}=0$, and if $\omega$ is in addition closed, then 
$(d+\omega)^{2}=d^{2}+[d,\omega]+\omega^{2}=0$.
Hence 
a cycle $$\omega\in Z^1(\Omega A(M))$$ can be 
used to deform this differential $d$ to $d+\omega$. 
\begin{ddd}\label{may0301}
We let $\cM(\cA,\omega)\in  \Mod_{\Omega A}( M)$ {denote} the sheaf of $\Omega A_{|M}$-modules  $\cM(\cA)$   with new differential $d+\omega$.\\
\end{ddd}
More precisely, if $f:N\to M$ is an object of $\Mf/M$, then
$\cM(\cA,\omega)(N\to M)$ is the complex $\cM(\cA)(N\to M)$ with the new differential $d+f^{*}\omega$.

\bigskip

In the following we analyse when $\cM(\cA,\omega)$ is invertible, weakly locally constant{,}  and $K$-flat.

\begin{lem}\label{apr2602}
Let  $\cA\in \PicLf_{\underline{A}}(M)$ and $\omega\in Z^1(\Omega A(M))$. {Then}
%\uli{Then   the following assertions hold true.}
%\begin{enumerate}
%\item 
% 
$\cM(\cA,\omega)$ is invertible.
%\item $\cM(\cA,\omega)$ is weakly locally constant.
%\item  $\cM(\cA,\omega)$ is weakly cartesian and cartesian.
%\end{enumerate}
\end{lem}
\proof 
For $\cA,\cA^{\prime}\in  \PicLf_{A}(M)$ and    $\omega,\omega^{\prime}\in Z^1(\Omega A(M))$
we have the rule  \begin{equation}\label{sep0604}\cM(\cA,\omega)\otimes_{\Omega
A_{|M}} \cM(\cA^{\prime},\omega^{\prime})=
\cM(\cA\otimes_A \cA^{\prime},\omega+\omega^{\prime})\ . \end{equation}
Hence if $\cB\in \PicLf_{\underline{A}}(M)$ is such that $\cA\otimes_{A}\cB\cong \underline{A}_{|M}$, then by  \eqref{sep0604} we get $$\cM(\cA,\omega)\otimes_{\Omega A_{|M}}  \cM(\cB,-\omega)\cong \Omega A_{|M}\ .$$ \hB 

%\uli{2. We now show that $\cM(\cA,\omega)$ is weakly locally constant.}

%

%\uli{2.} We now show that $\cM(\cA,\omega)$ is weakly homotopy invariant.   
%\uli{By Remark \ref{apr2501} it suffices to show that  the natural map $\cM(\cA,\omega)\to I\cM(\cA,\omega)$ is a quasi-isomorphism.} 
%  \uli{Since $\cA$ is locally constant we have the lower horizontal isomorphism in the diagram.
% $$ \xymatrix{\cM(\cA,\omega)\ar[r]^{\cong} \ar[d]&\cM(\underline{A},\omega)\otimes_{A}\cA\ar[d]\\I\cM(\cA,\omega)\ar[r]^{\cong} &I\cM(\underline{A},\omega)\otimes_{A}\cA
% }$$ Since $\cA$ is $K$-flat it suffices to show that $\pr^{*}:\cM(\underline{A},\omega)\to I\cM(\underline{A},\omega)$ is a quasi-isomorphism. We are going to show that $i^{*}:I\cM(\underline{A},\omega)\to \cM(\underline{A},\omega)$ is a homotopy inverse of $\pr^{*}$, where $i^{*} $ denotes the map induced by the inclusion of $0$ into the unit interval $I:=[0,1]$. 
% }
%  
%\uli{ We first define a  map $h:I\Omega^{*}\to I\Omega^{*}[-1]$ by $$h(\alpha):= \int_{I\times I/_{\pr_{2}}I} \mu^{*} \alpha\ ,$$
%where $\mu:I\times I\to I$ is given by $\mu(s,t):= st$. 
%One can check that  $dh-hd=\id_{I\Omega^{*}}-\pr^{*}i^{*}$.
%}   

%\uli{Note that $I\cM(\underline{A},\omega)$ is the sheaf $I\Omega A=I\Omega^{*}\otimes_{\R} A$ with differential $d+\pr^{*}\omega$. We now observe that
%$h\otimes_{\R}\id_{A}$ provides the chain homotopy from $\pr^{*}i^{*}$  to $\id_{I\cM(\underline{A},\omega})$.}

%\bigskip 

%\uli{3. This is the same argument as in Lemma \ref{apr2601}.} 
%\hB 
Let $\beta\in F^{1}\Omega A{^0}(M)$, {where $F^{p}\Omega A(M) {\subseteq} \Omega A(M)$ denotes the subcomplex of forms with differential {form} degree $\geq p$.}  Then $\beta$ is nilpotent and {the}
multiplication by $\exp(\beta)$ provides an isomorphism
\begin{equation}\label{may0310}
\exp(\beta):\cM(\cA,\omega+d\beta)\iso \cM(\cA,\omega)\ .
\end{equation}
 \begin{lem}\label{ulimar2130}  
  We assume that $A$ is split.  Let $\cA\in \PicLf_{\underline{A}}(M)$  and  $\omega\in Z^{1}(F^2\Omega A(M))$. Then  the $\Omega A_{{|M}}$-module $\cM(\cA,\omega)$ is $K$-flat {and weakly locally constant.}
\end{lem}
\proof
$K$-flatness {and the condition of being weakly locally constant} can be checked locally. {
Let $[\omega]$ denote the cohomology class of the cycle $\omega$.}
%Note that $\underline{A}\to \Omega A$ is a quasi-isomorphism. 
If $m\in M$, then
$[\omega]_{|m}=[\omega_{|m}]=0$ {by the filtration condition}. Let $U\subseteq M $ be  a contractible  neighbourhood of $m$. {Since {the cohomology of} $\Omega A$ is homotopy invariant, }we {get} that $\omega_{|U}=d\beta$ for some $\beta\in \Omega A^{0}(U)$. 
Let us write $\beta=\sum_{i\in \nat} \beta^{i}$, where $\beta^{i}\in (\Omega^{i}\otimes A^{-i})(U)$
and set
$\beta^{\ge {1}}:=  \sum_{i={1}}^{\infty} \beta^{i}$. {Since $A$ is split} we still have
$d\beta^{\ge {1}}=\omega_{|U}$.We get an isomorphism $\exp(-\beta):\cM(\cA,\omega)_{|U}\cong \cM(\cA)_{|U}$ in $\Mod_{\Omega A}(U)$.
Since $\cM(\cA)_{|U}$ is $K$-flat and {weakly locally constant by Lemma \ref{apr2601}}, so is $\cM(\cA,\omega)_{|U}$.    \hB

%The sheaf of CDGAs  $\Omega A$ has a decreasing natural filtration 
%$$\dots\subseteq F^{2}\Omega A\subseteq F^{1}\Omega A\subset F^{0}\Omega A=\Omega A$$
%by form degree.
%Let us decompose $\omega\in Z^{1}(\Omega A(M))$ as
%$$\omega=\omega_{0}+\omega_{1}+\omega_{\ge 2}\ ,$$
% where $\omega_{0}\in (\Omega^{0}\otimes \cA^{1})(M)$\ , $\omega_{1}\in (\Omega^{1}\otimes \cA^{0})(M)$ and $\omega_{2}\in F^{2}\Omega A(M)$. Assume that $M$ is connected. Then there exists $\eta\in Z(A)^{1}$ such that $\omega_{0}=1\otimes \eta$. If we add
% $\eta$ to the differential of $\cA$, then we get $\cA^{\prime}\in \Picf_{A}(M)$ such that
% $\cM(\cA,\omega_{0})$  
% 
% 

%One easily checks that this
%again defines a new sheaf of complexes, which is an invertible $\Omega A$-module. We can
%restrict to the filtration $F^{2}$ since elements in $\Omega^{1}(M, A)$ can
%be incorporated in the structure of $\cA$.

%\begin{definition}
%We let $\cM(\cA,\omega)$ be the invertible sheaf of 
%$\Omega A$-modules $\cM(\cA)$ over $M$ with new differential $d+\omega$.
%\end{definition}

%We have the rule  \begin{equation}\label{sep0604}\cM(\cA,\omega)\otimes_{\Omega
%A} \cM(\cA^{\prime},\omega^{\prime})=
%\cM(\cA\otimes_A \cA^{\prime},\omega+\omega^{\prime})\ . \end{equation}
%If $\beta\in (F^{1}\Omega A)^0(M)$ is such that $d\beta\in F^{2}\Omega A(M)$,
%then $\beta$ is nilpotent and 
%multiplication by $\exp(\beta)$ provides an isomorphism
%\begin{equation*}
%\exp(\beta):\cM(\cA,\omega+d\beta)\to \cM(\cA,\omega)\ .
%\end{equation*}
A homomorphism of abelian groups $d:X\to Y$ defines a {Picard groupoid} with objects $Y$ and morphisms $x:y\to y^{\prime}$ for $y,y^{\prime}\in Y$ and $x\in X$ with  $y+dx=y^{\prime}$. 
\begin{rem2}\label{apr3001}{As Picard-$\infty$-groupoid it is equivalent to $\Omega^{\infty}H(X\to Y)$, where we consider $X\to Y$ as a chain complex with $Y$ in degree $0$.}
\end{rem2}

In order to make the following definition we assume that $A$ is split.
\begin{ddd}\label{apr2901}We let $\mathbf{Z}^1_{A}\in \Sh_{\CGrp(\cS)}$ 
%\marginpar{\uli{Achtung: $\mathbf{Z}^1_{
% kein Stack}}
denote the {Picard-$\infty$ stack induced} by the map 
   \begin{equation}\label{apr0530}F^{1}\Omega A^{0} \stackrel{d}{\to} Z^{1}(F^{2}\Omega A) \end{equation}
 of sheaves of abelian groups.
 \end{ddd}
 Indeed, the condition that $A$ is split ensures that $d$ increases the filtration.
\begin{prop}\label{msymmetricmonoidal}
Assume that $A$ is split. Then
the assignment $(\cA, \omega) {\mapsto} \cM(\cA,\omega)$ refines to a   faithful transformation \begin{equation}\label{ulimar2161n}
 \cM:{\PicLf_{\underline{A}}} \times \mathbf{Z}_{A}^1 {\to \PicwLf_{\Omega A} }\end{equation}
 %\marginpar{\uli{man sollte beide Seiten stackifizieren}}

 between symmetric monoidal stacks
%Here $\mathbb{Z}^1_{dR}(M,A)$ denotes the symmetric monoidal category with objects given by elements in $Z^1(F^2\Omega A)(M)$ and morphisms given be elements of $(F^{1}\Omega A)^0(M)$  seen as a category in the obvious way.
\end{prop}
\proof
The formula in the statement of the proposition describe the action of the functor on objects.
Let $(\phi,\beta):(\cA,\omega)\to (\cA^{\prime},\omega^{\prime})$ be a morphism in $({\PicLf_{\underline{A}} }\times \mathbf{Z}_{A}^1)(M)$. Then
$\cM(\phi,\beta):\cM(\cA,\omega)\to \cM(\cA^{\prime},\omega^{\prime})$ is given by
$$\cM(\cA,\omega)\stackrel{\id_{\Omega A_{|M}}\otimes \phi}{\longrightarrow} \cM(\cA^{\prime},\omega)\stackrel{\exp(\beta)}{\longrightarrow} \cM(\cA^{\prime},\omega^{\prime})\ ,$$
where $\omega^{\prime}=\omega+{d}\beta$.
The symmetric monoidal structure on $\cM$ {given by the canonical maps.}
\bigskip

Finally we must show that $\cM$ is faithful. Assume that 
\begin{equation}\label{ulimar2170n}\cM(\phi,\beta)=\cM(\phi^{\prime},\beta^{\prime})\ .\end{equation} We conclude that
 $\phi=\phi^\prime$ by considering the restriction of the equality  \eqref{ulimar2170n} to all points $m\in M$. It remains to show that $\beta=\beta^{\prime}$.

 The ring of endomorphisms of an object in $ {\PicLf_{\Omega A}(M)}$  {is naturally isomorphic to the ring  $Z^{0}(\Omega A(M))$}.
 %\footnote{\uli{Hier scheinen wir die Annahme ``cartesisch'' zu brauchen.}}
The automorphism $\cM(\phi,\beta^{\prime})^{-1}\circ \cM(\phi ,\beta)=\cM(\id,\beta-\beta^{\prime})$  is given by multiplication with the cycle $\exp(\beta-\beta^{\prime})\in 1+ F^{1}\Omega A^{0}(M)$. The inverse of the exponential is given by the logarithm
$$\log:1+F^{1}\Omega A^{0}(M)\to F^{1}\Omega A^{0}(M)\ ,\quad \log(\gamma){:=}-\sum_{n=1}^{\infty} \frac{(\gamma-1)^{n}}{n}\ .$$
Hence we can recover the difference $\beta-\beta^{\prime}$ from the composition $\cM(\phi,\beta^{\prime})^{-1}\circ \cM(\phi,\beta)$.  If $\cM(\phi,\beta^{\prime})^{-1}\circ \cM(\phi,\beta)=\id$, then $
\beta=\beta^{\prime}$. \hB

Let $k\in \nat$ and $\omega\in  Z^{1}(F^{1}\Omega A(S^{k+1}))$. The sheaf $\cM(\underline{A},\omega)\in {\PicwLf_{\Omega A}(S^{k+1})}$ gives rise to an {object} $H\cM(\underline{A},\omega)\in {\PicL_{\underline{HA}}(S^{k+1})}$.
In the following we calculate the invariant $$\nu(H\cM(\underline{A},\omega))\in {\pi_{k} (HA)\cong}\widetilde{HA}^{0}(S^{k})$$ defined in Definition \ref{ulimar2101}. {Here and below  $\tilde{\textcolor{white}{-}}$ indicates reduced cohomology.}\\

The composition of the de Rham equivalence with the suspension isomorphism provides an isomorphism
$$\sigma:\tilde H^{1}(\Omega A(S^{k+1}))\iso \widetilde{HA}^{1}(S^{k+1})\iso \widetilde{HA}^{0}(S^{k})\ .$$

\begin{lem}\label{ulimar2132}
We assume that $A$ is split.
Then we have
$\nu(H\cM(\underline{A},\omega))=\sigma([\omega])$.
\end{lem}
\begin{proof}
We use the notation introduced in order to state Definition \ref{ulimar2101}.
We can find $\beta_\pm \in
\Omega A^0(D_{\pm})$ with $d\beta_\pm = \omega_{|D_\pm}$.  We can normalize $\beta$ such that
$\beta_{\pm|p}=0$. Since $A$ is split, we then, as in the proof of Lemma \ref{ulimar2130} or Lemma \ref{pseudo_surjective}, {can} modify $\beta$ further such that $\beta_{\pm}\in F^{1}\Omega A^{0}(D_{\pm})$.
Then by the usual description of the suspension isomorphism in  de Rham cohomology we have
$$\sigma([\omega])=[\beta_{+,|S^{k}}-\beta_{-,|S^{k}}]\in \widetilde{H}^{0}(\Omega A(S^{k}))\cong \widetilde{HA}^{0}(S^{k})\ .$$
%$[\beta_0|_X - \beta_1|_X] \in H^0(X,A)$ represents the class $S[\omega]$.\\

The trivializations $\phi_{\pm}:\cM(\underline{A},\omega)_{|D_{\pm}}\iso \Omega A_{|D_{\pm}}$ are now given by $\phi_{\pm}:=\exp(-\beta_{\pm})$. According to Definition \ref{ulimar2101} we have
$$\nu(H\cM(\underline{A},\omega))=[\exp(\beta_{+,|S^{k}}-\beta_{-,|S^{k}})]-1=[\beta_{+,|S^{k}}-\beta_{-,|S^{k}}]  \ .$$ 
\end{proof}

%%%%%%%%%%%%%%%%%%%%%%%%%%%%%%%%%%%%%%%%%%%%%%%%%%%%%%%%%%%%%%%%%%%%%%%%%%%%%5
\section{Uniqueness of differential refinements of twists}\label{may0510}

Let $(R,A,c)$ be a {differential refinement (Definition \ref{ulimar1501}) of a commutative ring spectrum $R$.}
In this section we show the following theorem.

\begin{thm}\label{unqiue} Let $M$ be a smooth manifold and
assume that $A$ is split. Then the {forgetful functor which  {projects to }the underlying topological twist induces an injective homomorphism} 
\begin{equation*}
\pi_{0}(\Tw_{ \hat R}(M)) \to  \pi_{0}({\PicL_{\underline{R}}}(M)) \ .
\end{equation*}
 In particular, if {a  $R$-twist on $M$} admits a  differential refinement, then this refinement is unique up to {(non-canonical)} equivalence.
\end{thm}
{The proof of this theorem will be finished  after some preparations at the end of this section.}
%As a preparation we need the following result on invertible  and $K$-flat $\Omega A$-modules.
%We start with some algebraic preparations.
%%%%%%%%%%%%%%%%%%%%%%%%%%%%%%%%%%%%%%%%%%%%%%%%%%%%%%%%%%%%%%%%%%%%%%%%%%%%%5
%In this section we want to show that under mild conditions on the CDGA A, the differential refinement of a topological twist $E$ is, if it exists, unique in the appropriate sense.
%Let $A$ be a CDGA over $\R$ and $\Omega A$ be the associated sheaf of CDGAs over $\Mf$.

%\begin{definition}\label{defrigid} 
%An element $X\in \Picf_{\Omega A}(M)$ is called rigid the following implication
%hold true:\\

%If $Y\in \Picf_{\Omega A}(M)$ and $f:X\to Y$ is a morphism in $\Mod_{\Omega A}(M)$ such that
%$\iota(f):\iota(X)\to \iota(Y)$ is an equivalence, then $f:X\to Y$ is an isomorphism.\\
%\end{definition}
% 
% 
%\tho{In order to prove the Theorem we need some Lemmas.}
\begin{lemma}\label{pseudo_surjective}
If $A$ is split, then   \begin{equation*}
\xymatrix{
 Z^0\big(\Omega A(M)\big)^\times  \ar[r]\ar[d] & H^0\big(\Omega A(M)\big)^\times \ar[d] \\
Z^0\big(\Omega A(M)\big) \ar[r] &  H^0\big(\Omega A(M)\big)
}
\end{equation*} is a pullback diagram. 
\end{lemma}
\proof 
We first observe that it suffices to show the assertion under the additional assumption that   $M$ is connected.

\bigskip

 We consider a cycle $\omega\in Z^{0}(\Omega A(M))$ such that $[\omega]\in H^{0}(\Omega A(M))^{\times}$. We then must show that $\omega\in Z^{0}(\Omega A(M))^{\times}$.
 
 \bigskip

We pick a point $m \in M$. Then $[\omega]_{|m}\in H^{0}(A)^{\times}$.   We write
$\omega=\sum_{i\in \nat}\omega^{i}$, where $\omega^{i}\in (\Omega^{i}\otimes A^{-i})(M)$.
Since $A$ is split we have $d_{dR}\omega^{0}=0$ and therefore $\omega^{0}=1\otimes x$, where
$x\in Z^{0}(A)$ represents $[\omega]_{|m}\in H^{0}(A)$. Since $A$ is split we have $Z^{0}(A)\cong H^{0}(A)$. Since $[\omega]_{|m}$ is a unit we conclude that $x$ is a unit. This implies that
 $\omega$ is a unit, since $\omega-\omega^{0}$ is nilpotent.
\hB

  \begin{lemma}\label{ulimar1702}
{We consider an object}  $Y\in\PicwLf_{\Omega A}(M)$. {Then we}
 have an isomorphism
$$\pi_{0}\Big(\Map_{  \Mod_{\iota{(}\Omega A{)}}(M)}\big( \iota(\Omega A)_{|M},\iota(Y)\big)\Big)\cong  H^{0}\big(Y(M)\big)\ .$$
\end{lemma}
\proof The evaluation at $M$ fits into an adjunction
$$\cF:\Ch_{\infty} \leftrightarrows \Mod_{\iota(\Omega A){(M)}}:\mbox{evaluation at $M$}$$
{such that  $\cF(\iota(\Z))=\iota(\Omega A)_{{|M}} $}. This gives the equivalence
$$\Map_{  \Mod_{\iota(\Omega A)}(M)}(\iota(\Omega A)_{|M},\iota(Y))\cong 
   \Map_{\Ch_{\infty}}(\iota(\Z),\iota(Y(M)))\ .$$
%\marginpar{\uli{F\"ur die zweite \"Aquivalenz w\"urde ich gerne ein Argument sehen.}}
We now use the fact that
$\pi_{0}(\Map_{\Ch_{\infty}}(\iota(\Z),\iota(Y(M))))\cong H^{0}(Y(M))$.
\hB

  In order to formulate the following Lemma we use the fact that $\Mod_{\Omega A}(M)$ has a natural enrichment to a  $DG$-category.  {In particular the notion of chain homotpy is well defined in $\Mod_{\Omega A}(M)$}.
%  
%   \ntho{In particular, \uli{for 
%  $Y,X\in \Mod_{\Omega A}(M)$}  we can  consider the chain complex
%  $\underline{\Map}_{\Mod_{\Omega A}(M)}(X,Y)$ so that  $\Hom_{\Mod_{\Omega A}(M)}(X,Y)\cong Z^{0}(\underline{\Map}_{\Mod_{\Omega A}(M)}(X,Y))$. 
%  Two morphisms $f,f^{\prime}\in  \Hom_{\Mod_{\Omega A}(M)}(X,Y) $ are called chain homotopic if they represent the same cohomology class in 
%  $H^{0}(\underline{\Map}_{\Mod_{\Omega A}(M)}(X,Y))$.}
%
\begin{lemma}\label{lem:eins}
Assume that {$X,Y \in \PicwLf_{\Omega A}(M)$.} Then {the following assertions hold true:}  \begin{enumerate}
\item
The map 
%\tfoot{Hab ein paar Klammern bei $\iota(...)$ weggemacht, damit das besser auf die Seite passt und nicht immer blau
%gemacht. 
%\uli{Hab sie wieder hingemacht, wenns kein Unheil gibt. Einheitlichkeit der Notation wegen Wiedererkkenung!}}
\begin{equation}\label{ulimar1910}
 \Hom_{\Mod_{\Omega A}(M)}(X,Y) \to \pi_0\big(\Map_{\Mod_{\iota(\Omega A)}(M)}(\iota( X),\iota (Y) )\big) 
\end{equation} 
is surjective.
\item
If $f,f^{\prime} \in \Hom_{\Mod_{\Omega A}(M)}(X,Y)$ {are equivalent in $\Mod_{\iota(\Omega A)}(M)$}, i.e. they induce the same {point} in  $ \pi_0\big(\Map_{\Mod_{\iota(\Omega A)}(M)}(\iota (X),\iota (Y))\big)$,
then $f$ and $f^{\prime}$ are chain homotopic.
\end{enumerate}
\end{lemma}
\proof
{1.} We first assume that $X = \Omega A_{|M}$.  {Then we have the isomorphisms} 
$$
\Hom_{\Mod_{\Omega A}(M)}(\Omega A_{|M},Y)%\cong Z^{0}(\underline{\Map}_{\Mod_{\Omega A}(M)}(\Omega A_{|M} ,Y))\\
\cong Z^{0}(Y(M))\ .$$
{Furthermore, we have the isomorphism}
$$ \pi_0\big(\Map_{\Mod_{\iota(\Omega A)}(M)}(\iota(\Omega A)_{|M},\iota(Y))\big) \stackrel{Lemma \ref{ulimar1702}}{\cong} H^{0}(Y(M))\ .
$$
The assertion now follows from the fact that the projection $Z^{0}(Y(M))\to H^{0}(Y(M))$ is surjective.

\bigskip

In the general case we have the following commuting diagram\begin{equation*}\hspace{-0.3cm}
\xymatrix{
\Hom_{\Mod_{\Omega A}(M)}(X,Y) \ar[d]^{\otimes X^{-1}}\ar[r] & \pi_0\big(\Map_{\Mod_{\iota(\Omega A)}(M)}(\iota (X),\iota (Y))\big)\ar[d]^{\otimes \iota(X^{-1})} \\
\Hom_{\Mod_{\Omega A}(M)}(\Omega A_{{|M}}, X^{-1} \otimes_{\Omega A_{|M}} Y) \ar[r] & \pi_0\big(\Map_{\Mod_{\iota(\Omega A)}(M)} (\iota(\Omega A)_{|M}, \iota(X^{-1} \otimes_{\iota(\Omega A)_{|M}} Y))\big)
} 
\end{equation*}
The vertical maps are both bijective. Hence the  upper horizontal map is surjective since the lower one is.

\bigskip

{2.} For the second assertion, {as for 1.},  we can reduce to the case where $X=\Omega A_{|M}$.  In this case the map \eqref{ulimar1910} reduces to the projection $Z^{0}(Y(M))\to H^{0}(Y(M))$ and the assertion is obvious. 
\hB 
%We use that $\Hom_{\Omega A \text{-Mod}}(X,X) \cong \Hom_{\Omega A \text{-Mod}}(\Omega A,\Omega A)$ and 
%$\Hom_{H\Omega A \text{-Mod}}(HX,HX) \cong \Hom_{H\Omega A \text{-Mod}}(H\Omega A,H\Omega A)$ to reduce this statment 
%to the case $X = \Omega A$ where it is obvious.  \marginpar{\uli{Sehe ich nicht!}}
%\end{enumerate}
%\end{proof}
{
\begin{proposition}\label{ulimar1703}
If $A$ is  split, then the canonical  functor
$$\iota:{\PicwLf_{\Omega A}(M)}\to \Ho\PicL_{\iota{(}\Omega A{)}}(M)$$
is full.
\end{proposition}
}
\proof %(of Proposition \ref{ulimar1703}):
%Lemma \ref{lem:eins} allows us to formulate a variant of Whiteheads theorem.
%\begin{lemma}\label{lem:zwei}
%Let $f: X \to Y$ be a quasi-isomorphism between invertible, $K$-flat $\Omega A$-modules over $M$.  Then $f$ is a chain-homotopy equivalence. 
%\end{lemma}
%\begin{proof}
Assume we are given objects $X,Y$ in $\PicwLf_{\Omega A}(M)$ and a morphism $F: \iota (X) \to \iota (Y)$ in the homotopy category 
%\footnote{\uli{Mehr Klammern bei $\iota(\dots)$ eingef\"ugt. Nicht alle rot}}
 $\Ho\PicL_{\iota(\Omega A)}(M)$. By Lemma  \ref{lem:eins} (1) we can conclude that there is a morphism $f: X \to Y$ in ${\Mod_{\Omega A}(M)}$ such that $\iota (f) = F$. It remains to show that $f$ is an isomorphism. 
 
Since $F$ is invertible there exists an inverse 
 $G:\iota(Y)\to \iota(X)$ {in $\Ho\PicL_{\iota(\Omega A)}(M)$}{, i.e. $G\circ F$ and $F\circ G$ are equivalent to the identities. By the first assertion of Lemma \ref{lem:eins} there exists a morphism $g:Y\to X$ such that $\iota(g){=}G$. By the second assertion of the  Lemma the compositions
$g\circ f$ and $f\circ g$ are chain-homotopic to the corresponding identities.

{We can} identify $\Hom_{\Mod_{\Omega A}(M)}(X,X)\cong Z^{0}(\Omega A(M))$ as a  ring 
and $H^{0}(\Omega A(M))$ with chain-homotopy equivalence classes of endomorphisms such that
multiplication corresponds to composition. Since the chain homotopy class of $g\circ f$ is invertible, it follows from Lemma \ref{pseudo_surjective} that $g\circ f$ itself is invertible. Similarly we see that
$f\circ g$ is invertible.
This implies that $f$ is an isomorphism. 
\hB
{
\begin{rem2}
More generally, the  
assertion of Proposition \ref{ulimar1703} holds true
for $K$-flat invertible $\Omega A$-modules, i.e. the assumption of being weakly locally constant can be dropped.
\end{rem2}
}

 \begin{proof}(of Theorem \ref{unqiue}):
By Proposition \ref{ulimar1703} and the long exact sequence in homotopy 
the fibre of $${\PicwLf_{\Omega A}(M)\stackrel{\iota}{\to} \PicL_{\iota(\Omega A)}(M)\cong \PicL_{H\Omega A}(M)} $$  is connected. 
  By the Definition  \ref{ulimar1704} of $\Tw_{\hat R}(M)$ as a pull-back the fibre of the map   $${\Tw_{ \hat R}(M) \to \PicL_{\underline{R}}(M)}$$ 
   is connected, too. 
\end{proof}

%\begin{remark}
%The theorem implies in particular that 
%the twisted differential cohomology group $\hat{R}^{(E,\cM,d)}(M)$ depends up to isomorphism
%only on the topological twist $E$.
%\end{remark}

%%%%%%%%%%%%%%%%%%%%%%%%%%%%%%%%%%%%%%%%%%%%%%
\section{Existence of differential refinements of twists}\label{sec:existence}
%%%%%%%%%%%%%%%%%%%%%%%%%%%%%%%%%%%%%%%%%%%%%%%%%%%%%%%%%%%%%%%%%%%%%%%%%%%%%%%%
%%%%%%%%%%%%%%%%%%%%%%%%%%%%%%%%%%%%%%%%%%%%%%

\newcommand{\Gr}{\mathrm{Gr}}

We consider a differential {refinement} $(R,A,c)$  of a commutative ring spectrum $R$  {(Definition \ref{ulimar1501})}. In the statement of the following theorem we use the fact that if $W$ is an $R$-module {spectrum}, then  {its} {realification} $W\wedge H\R$ has the strucure of {an} $HA$-module via $c$.
 In this section we show the following theorem.
\begin{theorem}\label{ulimar2170}
We assume that the {CDGA} $A$ is split. Let $M$ be a connected smooth manifold with a base point $m\in M$ and $E $ be a {topological} $R$-twist on $M$.  {Then the} twist $E$ admits a  differential refinement  $(E,\cM,d)$ if and only if there exists an invertible $K$-flat {DG}-$A$-module $X$ such that {there exists an equivalence  {$E(\{m\})\wedge H\R\cong HX$}  of $HA$ modules.} 
\end{theorem}

\begin{rem2}
 {
Note that  the verification of the   condition in the theorem boils down to check whether     the spectrum $E(\{m\})\wedge H\R$ allows a real chain complex model. In fact, one can always find a DG-$A$-module which models this spectrum. But we require in addition that this DG-$A$-module is strictly invertible and $K$-flat, which is not automatic. }

The following twists satisfy the condition:
\begin{enumerate}
\item shifts of the spectrum $R$;
\item twists    classified by a map to $\BGL_1(R)$;
\item twists that are pointwise of the form of a shift of $R$. These are exactly the twists  classified by a  map  to the spectrum of periodic units considered by Sagave \cite{Sagave}. 
\end{enumerate}

%\tho{However note that all shifts of the spectrum $R$ satisfiy the condition of the theorem. Moreover,    all twists    classified by a map to $\BGL_1(R)$ satisfy the condition. More generally all twists that are pointwise of the form of a shift of $R$. These are exactly the twists corresponding to the maps to the spectrum of periodic units considered by Sagave \cite{Sagave}. }
\end{rem2}

The proof of this theorem {requires some work}   and will be finished at the end of this section.
The main point is to find the corresponding object $\cM$  in {$\PicwLf_{\Omega A}(M)$}.
The idea is to take, as a first approximation, the sheaf  $\cM(\underline{X})$, {where $X$ is the DG{-A-}module realizing the twist over the basepoint $m \in M$.} 
This approximation is, {by definition{,}} correct at the base point $m$. In a second step we try to {modify this sheaf} {$\cM(\underline{X})$} by tensoring
with a sheaf of the form $\cM(\cA,\omega)$ where $\cA\in {\PicLf_{\underline{A}}}(M) $ satisfies $\cA_{|m}\cong A$. 
In order to ensure the existence of this correction we must understand the  map 
$${\pi_{0}(H\cM(-))}: \pi_{0}({\PicLf_{\underline{A}}}(M)\times \mathbf{Z}^{1}_{A}(M))\to \pi_{0}{(\PicL_{\underline{HA}}}(M))$$ %\footnote{\uli{Sollte man nicht sowas wie $(\underline{\hspace{0.2cm}})$ statt $(-)$ setzen?}} 
in some detail. The calculation of this map is 
stated in Corollary \ref{ulimar2162}.\\

We consider a spectrum $p$. %\tfoot{$p$ ist hart fuer ein Spektrum, wie waere es mit $B$? \uli{
%Sehe ich ein, aber $B$ ist wegen Barkonstruktion auch nicht besser. Die Wahl von $p$ ist weil es in der Anwendung von einem Picardstack kommt, der mit $P$ ganz gut bezeichnet ist.}
 %}.
  For every {topological} space $X$ we have the Atiyah-Hirzebruch filtration $$\dots\subseteq F^{2}p ^{*}(X)\subseteq F^{1}p^{*}(X)\subseteq F^{0}p^{*}(X)=p^{*}(X)$$
of the $p$-cohomology groups. By definition, for $k\in \nat$ a class $c\in p^{*}(X)$ belongs to  $F^{k}p^{*}(X)$ if and only if $f^{*}c=0$ for every map $f:Y\to X$ from a  $k-1$-dimensional complex $Y$. The associated graded group $\Gr\,  p^{*}(X)$ is calculated by the Atiyah-Hirzebruch spectral sequence (AHSS).  If $p$ is a {commutative} ring spectrum, then the Atiyah-Hirzebruch filtration  turns the cohomology $p^{*}(X)$  into a filtered graded commutative {ring}. 

\bigskip

For $k\in \Z$
we let $p\langle k \rangle \to p$ denote the $k$-connected covering of $P$. We have
$\pi_{i}(p\langle k \rangle)=0$ for $i\le k$ and  an isomorphism $\pi_{i}(p\langle k \rangle)\iso \pi_{i}(p)$ for all $i>k$.
The following assertions are checked easily using the naturality of the AHSS. 
If  $X$ is a space, then the natural map
$p\zero\to p$ induces isomorphisms \begin{equation}\label{ulimar2010}p\zero^{0}(X)\iso F^{1}p^{0}(X)\ , \quad p\zero^{1}(X)\iso F^{2}p^{1}(X)\ .\end{equation}
Similarly, the maps $  p\one \to p$ and $p\langle -1 \rangle \to p$ {induce  isomorphisms}
\begin{equation}\label{ulimar2015}p\one^{0}(X) { \iso } F^{2}p^{0}(X)\ ,  \quad  p\langle -1 \rangle^{1}(X)\iso  F^{1} p^{1}(X) \ . \end{equation}
 The sheaf {of} CDGAs $\Omega A$ has a decreasing filtration 
$$\dots\subseteq F^{2}\Omega A\subseteq F^{1} \Omega A\subseteq F^{0}\Omega A=\Omega A\ ,$$
where $F^{k}\Omega A:=\prod_{i\ge k} \Omega^{i}\otimes A$. 
For every manifold $M$ we obtain an induced filtration on the cohomology groups $H^{*}(\Omega A(M))$.
Recall that we have the de Rham equivalence \eqref{ulimar2001}. 
\begin{lem}\label{ulimar2014}
The de Rham equivalence induces  is an isomorphism of filtered  commutative graded algebras
$$H^{*}(\Omega A(M))\iso HA^{*}(M)\ .$$ 
\end{lem} 
\proof {The main new point is that the de Rham equivalence preserves the filtration. This 
  will be shown in Lemma \ref{may0260} in a more general situation. \hB

Let $p$ be a spectrum and ${P:= \Omega^{\infty}p\in \CGrp(\cS)}$ be its infinite loop space. Then we obtain a constant Picard-$\infty$ stack}
{$\underline{P}\in \Sh_{\CGrp(\cS)}$, see Definition \ref{apr2630}.}
Note that   $$\pi_{i}( \underline{P}(M))\cong p^{-i}(M)$$ for all $i\ge 0$. 
%\uli{For $\cF\in \Sh_{\CGrp(\cS)}$ we set $F:=\cF(*) \in \CGrp(\cS)$.
% Then we have a canonical equivalence
%$\underline{F}\cong \hat C_{F}$ (see appendix of \cite{bg}). Moreover, \footnote{\uli{Hier ist der Unterschied zwischen $p^{*}F$ und $\underline{F}$!}} we have a canonical morphism
%$\hat C_{F}\to \cF$ induced by the maps $M\to *$.
%\begin{ddd}If this map is an equivalence, then
%we say   that $F$ represents $\cF$. \end{ddd}
Recall Definition \eqref{ulimar2002} of the {Picard-$\infty$ stack} $\PicL_{\underline{R}}$. 
We have the equivalence $$\Pic_{R}\cong {\PicL_{\underline{R}}}(*)\ ,$$ see Proposition \ref{prop_jul2}.  We further let $\pic_{R}$ be the connective spectrum such that $\Omega^{\infty}\pic_{R}:=\Pic_{R} \ .$
{Proposition \ref{prop_jul2} also shows that} we have  equivalences resp. an isomorphism \begin{equation}\label{uliapr0670}
{\PicL_{\underline{R}}}(M)\cong \Pic_{R}^{M^{top}} \cong  \Omega^{\infty}\Map(\Sigma^{\infty}_{+}M^{top},\pic_{R})\ ,\quad 
\pi_{0}(\PicL_{\underline{R}}(M))\cong \pic_{R}^{0}(M)\ ,
\end{equation}
i.e. $R$-twists are classified by a cohomology group. 
In particular we get an induced  decreasing Atiyah-Hirzebruch filtration   $$\dots\subseteq F^{2}\pi_{0}(\PicL_{\underline{R}}(M))\subseteq F^{1}\pi_{0}(\PicL_{\underline{R}}(M))\subseteq \pi_{0}(\PicL_{\underline{R}}(M))$$
%\footnote{\uli{zu viele $\cart$'s. Mache jetzt nicht alle mehr rot.}} 
on the abelian group of isomorphism classes  of  $R$-twists on $M$.\\

 The cohomology theory of the spectrum of units $\gl_{1}(R)$ 
 of the commutative ring spectrum  $R$  is characterized by the natural isomorphism \begin{equation}\label{ulimar2011}
\gl_1(R)^0(X) \cong R^0(X)^\times 
\end{equation}
for all spaces $X$. We have an equivalence of connective spectra
\begin{equation}\label{ulimar2030}
\pic_{R} \zero    \cong  \bgl_{1}(R):=\gl_{1}(R)[1]\ .\end{equation}
In calculations below we use the following fact:%\footnote{\uli{Nur um zu verstehen. Eigentlich ist das doch die Definition von \eqref{ulimar2030}}}
\begin{lem}\label{ulimar2131} For $k\ge 0$ the  isomorphism
$$\pic_{R}\zero^{0}(S^{k+1})\stackrel{\eqref{ulimar2030}}{\cong} \bgl_{1}^{0}(S^{k+1})\stackrel{susp.}{\cong} \widetilde{\gl_{1}^{0}(S^{k})}\stackrel{\eqref{ulimar2011}}{\cong} \left\{\begin{array}{cc}\pi_{k}(R)&k\ge 1\\\pi_{0}(R)^{\times}&k=0 \end{array}\right.$$
is given by
$$\pic_{R}\zero^{0}(S^{k+1})\ni E\mapsto \nu(E)  \ ,$$
where $\nu(E)$ is defined in Definition \ref{ulimar2101}.
\end{lem}
Since $\pic_{R}$ is a connective spectrum, we have  for every connected space $X$   with base point  canonical isomorphisms 
\begin{equation}\label{ulimar2140}\pic_{R}^{0}(X)\cong \pi_{0}(\pic_{R})\oplus \gl_{1}(R)^{1}(X)
 \ , \quad  \pic^{-1}_{R}(X)\cong  \gl_{1}(R)^{0}(X) \ .\end{equation}

\bigskip

By \eqref{ulimar2010}
 we have identifications
 \begin{equation}\label{ulimar2020}
%R\zero^{0}(X)\cong F^{1}R^{0}(X)\ , \quad 
\gl_{1}(R)\zero^{0}(X)\stackrel{\eqref{ulimar2011}}{\iso} 1+F^{1}R^{0}(X)\ , \quad F^{2}\gl_{1}(R)\one^{0}(X)\iso 1+F^{2}R^{0}(X) \uli{\ .} \end{equation} 

\bigskip

In general it is quite complicated to understand the spectrum of units $\gl_{1}(R)$ from the structure of $R$. But if $R$ is rational, then one can use a homotopy theoretic version of the exponential map
in order to relate  $R$ with its units. The details are as follows.

\begin{lem}\label{ulimar2150}
Assume that $Q $ is {a commutative algebra in $H\Q$-modules.} Then there exists an
equivalence of spectra \begin{equation*}
\exp_Q: Q\zero \iso \gl_1(Q)\zero\
\end{equation*}
which induces for every space $X$ the exponential map 
\begin{equation} \label{ulimar2012}
Q\zero^0(X) \to \gl_1(Q)\zero^0(X) \qquad \beta \mapsto \sum_{n=0}^\infty \frac 1
{n!} \beta^n \ .
\end{equation}
Moreover $\exp_Q$ is unique up to homotopy with this property.
\end{lem}
 \proof
The spectra   $Q\zero$, $\gl_{1}(Q)$ and $\gl_{1}(Q)\zero$ are $H\Q$-modules.
A map between connective $H\Q$-modules is determined up to homotopy by a map of $H$-spaces between the underlying infinite loop spaces {(this follows from the Milnor-Moore theorem \cite{MilnorMoore})}. Furthermore, a map of rational $H$-spaces is determined up to equivalence by the transformation of abelian group-valued functors on the category of finite $CW$-complexes represented by them {since there are no phantom maps between rational spectra}. Therefore in order to define and characterize the exponential map $\exp_{Q}$ it suffices to check that \eqref{ulimar2012}
is a well-defined transformation between abelian group-valued functors on finite $CW$-complexes.\\

 Every element $\beta\in Q\zero^{0}(X)\cong F^{1}Q^{0}(X)$ is nilpotent. Therefore the power series 
 reduces {to} a finite sum
  and determines an element  $$ \sum_{n=0}^\infty \frac 1{n!} \beta^n \in 1+F^{1}Q^{0}(X)\cong \gl_{1}(Q)\zero^{0}(X)\ .$$ 
 Additivity is checked as usual.
    \hB 

\begin{rem2}
The exponential map constructed in the lemma is only determined up to homotopy and not up to contractible choices. But we will see in {Remark \ref{uuu100}}  that a choice of chain complex modelling the rational spectrum $Q$ allows {one} to make a prefered choice which is well defined up to contractible choices.
\end{rem2}
   % \uli{The homotopy mentioned in the proof of Lemma \ref{ulimar2150} is not unique. The group
%$$[Q\zero, gl_1(Q)\zero[-1]]\cong [Q\zero,Q\zero[-1]]\cong \prod_{n\ge 1}\Hom(\pi_{n}(Q),\pi_{n+1}(Q))$$ acts on the set of homotopies.}
%\bigskip

 \begin{lem}\label{kor_iso_rational}
For every manifold $M$ we have a natural isomorphism
\begin{equation*}
Exp: F^{2}H^{1}(\Omega A(M)) \to F^{2}\pi_0(\PicL_{\underline{HA}}(M))
\end{equation*}
\end{lem}
\begin{proof}
By Lemma \ref{ulimar2014} the de Rham equivalence induces an isomorphism
$$F^{2}H^{1}(\Omega A(M))\iso F^{2}HA^{1}(M)\ .$$
%or equivalently by down-shift, an isomorphism
%$$ F^{2}H^{1}(\Omega A(M)[1])\iso F^{2} HA[1]^{0}(M)\ .$$
By \eqref{ulimar2010} we have an isomorphism
$$F^{2} HA^{1}(M)\cong HA\zero^{1}(M) \ .$$
 The exponential map
$\exp_{HA } $ induces an isomorphism
$$\exp_{HA }:HA \zero^{1}(M)\iso \gl_{1}(HA)\zero ^{1}(M)\ .$$
%Furthermore we have isomorphisms 
%$$gl_{1}(HA[1]\one)^{0}(M)\stackrel{\eqref{ulimar2015}}{\iso}   F^{2}gl_{1}(HA[1]\one)^{0}(M) \stackrel{\ref{ulimar2020}}{\iso} 1+F^{2}HA[1]\one^{0}(M)\ .$$
Finally we have the isomorphism
{$$
\xymatrix{
\gl_{1}(HA)\zero ^{1}(M) \ar[r]^\sim &
\bgl_{1}(HA)\one^{0}(M) \ar[r]^\sim_{\eqref{ulimar2030}} &
 \pic_{HA}\one^{0}(M) \ar[d]_{\eqref{ulimar2015}}^{\eqref{uliapr0670}}\\
&& F^{2}\pi_{0}\big(\PicL_{\underline{HA}}(M)\big)\ .
}
$$}
The isomorphism $Exp$ is defined as the composition of the isomorphisms above. 
\end{proof}
In {Proposition \ref{msymmetricmonoidal}}\ref{sep0602nnnn}} we have  constructed a symmetric monoidal transformation
$$\cM:\mathbf{Z}^{1}_{A}\to \PicwLf_{\Omega A}\ , \quad \omega\mapsto \cM(\underline{A},\omega)\ .$$
Its composition with the localization $\PicwLf_{\Omega A}\to \PicL_{\iota{(}\Omega A{)}}$ and Eilenberg-MacLane equivalence $\PicL_{\iota{(}\Omega A{)}}\iso \PicL_{\underline{HA}}$ yields a symmetric monoidal transformation
$$H\cM:\mathbf{Z}^{1}_{A} \to \PicL_{\underline{HA}}\ .$$
It immediately follows from the definition of  $\mathbf{Z}^{1}_{A}$ that $\pi_{0}(\mathbf{Z}^{1}_{A}(M))\cong F^{2}H^{1}(\Omega A(M))$.
\begin{prop} \label{thm_commutes} We assume that $A$ is split.
For every manifold $M$ 
the following diagram commutes
\begin{equation}\label{ulimar2040}\xymatrix{F^{2}H^{1}(\Omega A(M))\ar[d]^{Exp}\ar[r]^{\cong}&\pi_{0}(\mathbf{Z}^{1}_{A}(M))  \ar[d]^{H\cM} \\F^{2}\pi_{0}(\PicL_{\underline{HA}}(M))\ar[r]&
\pi_{0}(\PicL_{\underline{HA}}(M))  }\ ,\end{equation}
where the lower horizontal {map is} the {inclusion}.
\end{prop}
 The proposition follows from the following two Lemmas.
 \begin{lem}
 In order to show Proposition \ref{thm_commutes} it suffices to show that the diagram \eqref{ulimar2040} commutes in the case $M=S^{k}$ for all $k\in \nat $.
 \end{lem}
 \proof
 Using some natural isomorphisms discussed above we can rewrite the diagram \eqref{ulimar2040} as follows:
$$\xymatrix{HA\zero^{1}(M)\ar[d]^{Exp}\ar@{=}[r]& HA\zero ^{1}(M) \ar[d]^{H\cM} \\ \pic_{HA}\one^{0}(M) \ar[r]&
\pic^{0}_{HA}(M) }\ .$$
This makes clear that all corners are functors represented by rational $H$-spaces.
The horizontal maps  represented by maps of $H$-spaces by construction.
The transformations of abelian group valued functors $Exp$ and $H\cM$ are as well represented  by a map between $H$-spaces which are uniquely determined up to homotopy. Finally, the resulting diagram 
\begin{equation}\label{uliapr0620}\xymatrix{HA\zero[1]\ar[d]^{Exp}\ar@{=}[r]& HA\zero[1] \ar[d]^{H\cM} \\ \pic_{HA}\one \ar[r]&
\pic_{HA} }\end{equation}
of maps between rational $H$-spaces commutes up to homotopy if and only if it commutes on the level of homotopy groups. \hB
\begin{lem}\label{apr2420}
The diagram 
 \eqref{ulimar2040} commutes for $M=S^{k}$ for all $k\in \nat$.
\end{lem}
\proof
Note that $H\cM$ takes values in the subgroup $$F^{1}\pi_{0}(\PicL_{\underline{HA}}(S^{k+1}))\cong \pic_{{HA}}\zero^{0}(S^{k+1})\ .$$ By Lemma \ref{ulimar2131} the equivalence class of the element $H\cM(\underline{A},\omega))\in F^{1}\pi_{0}(\PicL_{\underline{HA}}(S^{k+1}))$ is completely determined by the class $\nu(H\cM(\underline{A},\omega))\in \pi_{k}(HA)$ defined in Definition \ref{ulimar2101}. \\

On the one hand, by Lemma \ref{ulimar2132} we have
$\nu(H\cM(\underline{A},\omega))=\sigma([\omega])$.
On the other hand, by Lemma  \ref{ulimar2131}  and the construction of $Exp$ we have the following equality in $HA^{0}(S^{k+1})$:
$$1+\sigma^{-1} (\nu(Exp([\omega]))) = \exp([\omega])=1+[\omega]\ .$$
We conclude that $\nu(Exp([\omega])) =\sigma([\omega])$. \hB
%\uli{The filler in the diagram 
%\begin{equation}\label{uliapr0620}\xymatrix{HA\one[1]\ar[d]^{Exp}\ar[r]& HA\zero[1] \ar[d]^{H\cM} \\ \pic_{HA}\one \ar[r]&
%\pic_{HA} }\ .\end{equation}
%is unique up to the action of the group
%$$[HA\one[1],\pic_{HA}[-1] ]\cong \prod_{n\ge 2} \Hom(H_{n-1}(A),H_{n}(A))\ . $$
%We can redefine $Exp$ by chosing a lift in
%$$\xymatrix{&&\pic_{HA}\one\ar[d]\\HA\one[1]\ar@{.>}[urr]\ar[r]&HA\zero[1]\ar[r]^{H\cM}&\pic^{0}_{HA}}$$
%which exists since $HA\one[1]$ is one-connected. In this way we fix a filler in \eqref{uliapr0620}.
% }
%It follows from 

%

 \begin{prop}\label{ulimar2161}
 We assume that $A$ is split. Then there exists a canonical equivalence of spectra
   \begin{equation*}
\gl_1(HA) \cong  H (H^{0}(A)^\times) \times HA\zero \ .
\end{equation*}
\end{prop}
\begin{proof}
We consider the  fibre sequence 
\begin{equation*}
\gl_1(HA) \zero \to \gl_1(HA) \to  H (H^{0}(A)^{\times}) \ .\end{equation*}
Since $A$ is split, we have a natural map of CDGAs
$$  H^{0}(A)\cong Z^{0}(A)\to A$$
which induces a split of the fibre sequence
$H (H^{0}(A)^{\times})\to {\gl_1(HA)}$. The required equivalence is now constructed using the exponential  equivalence Lemma \ref{ulimar2150}:
$$\gl_{1}(HA)\cong H (H^{0}(A)^\times) \times \gl_{1}(HA)\zero\stackrel{\id  \times \exp_{HA}^{-1} }{\longrightarrow}   H (H^{0}(A)^\times)\times HA\zero\ .$$
\end{proof}
Using \eqref{ulimar2140}, \eqref{uliapr0670}, and \eqref{ulimar2010} we get:
\begin{kor}\label{ulimar2162}
We assume that $A$ is split. 
For every connected manifold $M$ with a fixed base point we have natural isomorphisms
$$\pi_{0}(\PicL_{\underline{HA}}(M))\cong \pi_{0}(\pic_{HA})\oplus  {H(H^{0}(A)^{\times})^{1}(M)}  \oplus {F^{2}HA^1(M)}  $$ and 
%\tfoot{Ich sehe ein, dass das inkosistent is, gefaellt mir aber viel besser eigentlich....}
$$\pi_{1}(\PicL_{\underline{HA}}(M),X)\cong       H^{0}(A)^{\times} \oplus  F^{1}  {HA^0(M)}  $$
for every $X\in \PicL_{\underline{HA}}(M)$.
\end{kor}

\begin{kor}\label{ulimar2167}
We assume that $A$ is split.
For a connected manifold $M$ with {a fixed base point} the transformation %\eqref{ulimar2161}  
$$\pi_{0}(H\cM): \pi_{0}(\PicLf_{\underline{A}}(M) \times \mathbf{Z}_{A}^1(M)) \to \pi_{0}(\PicL_{\underline{HA}}(M))
$$  is given in  terms of components of the decompositions obtained in  Proposition
\ref{ulimar2160} and Corollary \ref{ulimar2162} as the map
$$\pi_{0}(\Picfff_{A})\oplus H(Z^{0}(A)^{\times})^{1}(M)\oplus F^{2}HA^{1}(M)\to 
\pi_{0}(\pic_{HA})\oplus  H(H^{0}(A)^{\times})^{1}(M)\oplus   F^{2}HA^{1}(M) $$
with the following description.
\begin{enumerate}
\item The first component maps the isomorphism class of $X\in \Picfff_{A}$ to the isomorphism class {of} $HX$ {in} $\pi_{0}(\pic_{HA})$. 
\item Since $A$ is split we have an isomorphism 
$Z^{0}(A)\cong  H^{0}(A)$ which induces the second component
$$H (Z^{0}(A)^\times) ^{1}(M)\cong  H (H^{0}(A)^\times) ^{1}(M)\ .$$
\item The third component is the identity by Proposition \ref{thm_commutes} and 
the construction of the split in Proposition \ref{ulimar2161}.\\
\end{enumerate}
The transformation 
$$\pi_{1}(H\cM): \pi_{1}(\PicLf_{\underline{A}}(M) \times \mathbf{Z}_{A}^1(M)) \to \pi_{1}(\PicL_{\underline{HA}}(M))
$$ 
is given in terms of components 
$$Z^{0}(A)^{\times} \oplus Z^{0}(F^{1}\Omega A(M)) \to H^{0}(A)^{\times}\oplus F^{1}HA^{0}(M)$$
as follows.
\begin{enumerate}
\item The first component is the isomorphism  $Z^{0}(A)^{\times}\iso H^{0}(A)^{\times}$.
\item The second component maps the cycle
% \marginpar{\uli{Wenn man stackifiziert, wird dies sogar ein iso.}}
$\omega \in Z^{0}(F^{1}\Omega A(M))$ to its de Rham cohomology class $[\omega]\in F^{1}HA^{0}(M)$.
\end{enumerate}
In particular, the functor {$\Ho H\cM:\Ho\mathbf{Z}^{1}_{A} \to \Ho\PicL_{\underline{HA}}$}
 is full.
\end{kor}

\proof (of Theorem \ref{ulimar2170}):
We choose $X\in \Picfff_{A}$ such that {$HX\cong E(m)\wedge H\R$} as $HA$-modules. Then we have
{$E(m)\wedge H\R -\underline{HX}(m)\cong 0$} in ${\Pic_{HA}}$. By Corollary \ref{ulimar2167}
there exists an object $\cA\in \PicLf_{\underline{A}}(M)$ with
$\cA{(m)}\cong A$ and $\omega\in Z^{1}(F^{2}\Omega A(M))$ such that
{$ E\wedge H\R\cong \underline{HX}+\cM(\cA,\omega)\in \PicL_{\underline{HA}}(M)$.}
 Then there exists an equivalence
$$d:E\wedge H\R\cong \cM(\cA\otimes_{A} \underline{X},\omega)\ .$$
Hence we obtain a  differential refinement 
$$(E, \cM(\cA\otimes_{A} \underline{X},\omega),d)\in {\Tw_{\hat R}(M)}$$ of $E$.

\hB
\part{Examples}
%%%%%%%%%%%%%%%%%%%%%%%%%%%%%%%%%%%%%%%%%%%%%%%%%%%%%%%%%%%%%%%%%%%%%%%%%%%%%5
\section{Differentially simple spectra} \label{realization}
%%%%%%%%%%%%%%%%%%%%%%%%%%%%%%%%%%%%%%%%%%%%%%%%%%%%%%%%%%%%%%%%%%%%%%%%%%%%%

For a general commutative ring spectrum $R\in \CAlg(\Sp)$ it seems to be difficult to single out a good choice of a differential extension $(R,A,c)$.  In the present section we introduce the notion of a differentially simple commutative ring spectrum. This class contains all examples which have been considered in applications so far. For a differentially simple commutative ring spectrum we have a very good canonical choice
of a differential extension and control of differential twists.

\begin{definition}\label{good}
We call a commutative ring spectrum $R$  differentially simple, if it has the following properties:
\begin{enumerate}
\item
$R$ is formal, i.e. there exists an equivalence  $$c:R\wedge H\R\iso H(\pi_{*}(R)\otimes \R)$$ in $\CAlg(\Mod_{H\R})$ that induces the identity on homotopy groups.
\item
$R$ is rationally connected, i.e. $\pi_{0}(R)\otimes \R\cong \R $.
\item  Rationally, the ring spectrum
$R$ has   no exotic twists, i.e. the canonical assignment
\begin{equation*}
\Z \to \pi_0(\Pic_{HA}) \qquad n \mapsto  HA[n]
\end{equation*}
surjects onto the image of $\pi_0(\Pic_R) \to \pi_0(\Pic_{HA})$.%\footnote{\uli{$()$ eingef\"ugt.}}
\end{enumerate}
\end{definition}

In general,  the hard piece to check in the above definition is condition (3). However when $R$ is connective, then condition 3. is automatic as follows from the following result.

\begin{proposition}[{Gepner, Antieau \cite[Theorem 7.9]{antgep}}]\label{ulimar2181}
If a commutative ring spectrum $E$ is connective and $\pi_0(E)$ is a local ring, then ${\pi_{0}(\Pic_{E})} = \Z$.%\footnote{\uli{$\pi_{0}$ eingef\"ugt}}
\end{proposition}
\begin{corollary}\label{thoapr18}
If $R$ is connective and satisfies condition (1) and (2) in Definition \ref{good}, then it also satisfies condition (3).
\end{corollary}
\proof If $R$ satisfies condition (1) and (2) in Definition \ref{good}, then according to Proposition  \ref{ulimar2181} the group  $\pi_{0}(\Pic_{HA})$  consists of the isomorphism classes of the shifts $HA[n]$, and these can be realized by the isomorphism classes of the shifts $A[n]$ in $ \pi_{0}(\Pic_{A})$. \hB 

Recall also that every ring spectrum  {whose real homotopy is a free graded commutative $\R$-algebra}  is formal. 
\begin{example}
The following spectra are differentially simple:
\begin{itemize}
\item
ordinary cohomology $H\Z$
{\item
the sphere spectrum $\mathrm{S}$}
\item
connective complex $K$-theory $ku$ and periodic $K$-theory ${\K}$
\item
connective real $K$-theory $ko$ and the periodic version ${\K O}$
\item
{connective topological modular forms $\mathrm{tmf}$ and the periodic version $\mathrm{TMF}$}
\end{itemize}
\end{example}
\begin{proof}
All examples follow from the last corollary except for the periodic $K$-theories and $\mathrm{TMF}$. But it follows from results of {Lennart Mayer and Akhil Mathew} \cite{2013arXiv1311.0514M} that 
$\pi_0(\Pic_{{\KU}}) = \Z/2$ and $\pi_0(\Pic_{{\K O}}) = \Z/8$ and $\pi_0(\Pic_{\mathrm{TMF}}) = \mathbb{Z}/ 576$. {In particular, all classes are represented by shifts.}  
\end{proof}

%\begin{rem2}
%It seems also true that the periodic version of tmf, denoted TMF is differentially simple. This was communicated to us by M. Hopkins. However the proof that 
%$\pi_0(Pic_{\mathrm{TMF}}) \cong \Z/576$ has not yet \uli{been} appeared in the literature.
%\end{rem2}

\begin{thm}
Assume that $R$ is a differentially simple spectrum and
$(R,A,c)$ is the canonical differential extension with $A=\pi_{*}(R)\otimes \R$ and equivalence
$c:R\wedge H\R\iso HA$  in $\CAlg(\Mod_{H\R})$.  Then we have the following assertions:
\begin{itemize}
\item
Every topological $R$-twist $E$ on a manifold $M$ admits  a differential refinement $(E,\cM,d)$ which is unique up to {(non-canonical)} equivalence.
\item
The sheaf $\cM$ of $\Omega A$-modules can be chosen  in the form
\begin{equation*}
\cM=\cM(L\otimes_{{\R}} \underline{A}, w) {[n]}
\end{equation*}
where $L$ is an invertible locally constant sheaf of $\R$-vector spaces {(i.e. { the sheaf of parallel} sections of a flat real line bundle)}  and $\omega \in Z^1(F^{2}\Omega A(M))$.
 \item
 The isomorphism class of $L$ and the de Rham class of $\omega$ are uniquely determined by the topological twist $E$. 
\end{itemize}
\end{thm}
\begin{proof}
The uniqueness in the first statement follows from Theorem \ref{unqiue} and the existence from Theorem \ref{ulimar2170}. The second part follows from the fact that $Z^0(A)^\times \cong {\R}^{\times}$
%\fuli{Verstehe diese Folgerung nicht.  Wieso folgt aus $\pi_{0}(R)\otimes \R\cong \R$, da{\ss} $Z^{0}(A)\times\cong \Z/2\Z$.
%Ich sehe das Theorem nur unter der Voraussetzung $\pi_{0}(R)=\Z$ oder was \"Ahnliches.} and  Corollary \ref{ulimar2167}.
\end{proof}

%%%%%%%%%%%%%%%%%%%%%%%%%%%%%%%%%%%%%%%%%%%%%%%%%%%%%%%%%%%%%%%%%%%%%%%%%%%%%5
\section{From differential $\bgl_{1}$ to twists}\label{may0511}
%%%%%%%%%%%%%%%%%%%%%%%%%%%%%%%%%%%%%%%%%%%%%%%%%%%%%%%%%%%%%%%%%%%%%%%%%%%%%5

The main goal of the present section is to construct differential twists from classes in differential $\bgl_{1}(R)$-cohomology.
A more precise statement will be given as Theorem  \ref{uliapr0720}.
 after the introduction of necessary {notation}.    \bigskip

First we  review ordinary (non-twisted, non-multiplicative) differential cohomology in the set-up of \cite{bg}, \cite{skript}. Let $p$ be a spectrum, 
$B$ be a  {chain} complex {of real vector spaces} and $e:p\wedge H\R\to HB$ be an equivalence in $\Mod_{H\R}$. The triple $(p,B,e)$ is {called a} differential data 
 {(this is the non-multiplicative version of a differential ring spectrum as introduced in Definition \ref{ulimar1501})}. For every $n \in \Z$ we define a differential function spectrum $\Diff^n(p,B,e)\in \Sh_{\Sp}$  as the pull-back
 \begin{equation}\label{apr0501}
 \xymatrix{
 \Diff^n(p,B,e)\ar[d]\ar[r]&H(\Omega B^{\ge n})\ar[d]\\
\underline{p}\ar[r]^{e}&\underline{HB}\ ,
}\end{equation}
where the right vertical map is induced by the inclusion {of chain complexes} $\Omega B^{\ge n}\to \Omega B$ and the de Rham equivalence $H\Omega B\cong \underline{HB}$.

\bigskip

For a smooth manifold $M$ the homotopy groups of the differential function spectrum  are given by 
\begin{equation}\label{apr0505}
\pi_{i}(\Diff^n(p,B,e)(M))\cong 
\left\{\begin{array}{cc} 
p^{-i}(M) & i < -n\\
p\R/\Z^{-i-1}(M)&i>-n\\
\widehat{p}^{-n}(M)&i=-n
\end{array}\right.\ ,\end{equation}
where {$$\widehat{p}^{n}(M):=\pi_{-n}(\Diff^n(p,B,e)(M))$$ is, by definition,}  the $n$th differential ${p}$-cohomology of $M$. We refer to \cite{bg} for further details. 
%Note that we have an equivalence
%$$H(\Omega B^{\ge n})\cong H(L( Z^{n}\Omega B ))$$ in $\Sh_{\Sp}$, where $L:\PSh_{\Ch_{\infty}}\to \Sh_{\Ch_{\infty}}$ denotes the sheafification.  Hence we can replace
%$H(\Omega B^{\ge n})$ by $ H(L(Z^{n}(\Omega B)))$ in the right upper corner of \eqref{apr0501}.
%\newcommand{\bZ}{\mathbf{Z}}
{For notational reasons we temporarily use the notation $q_{\ge 0}:=q\langle -1 \rangle$ for connective covers.}
For $n \in \Z$ and $k \in \mathbb{N}$ we define
$$\Diff^n_k(p,B,e) := (\Diff^{n-k}(p,B,e)[n])_{\geq 0}\ .$$
 Note that $\Diff^n_k(p,B,e)$ is a sheaf of connective spectra, but not a sheaf of spectra.   
We now consider a commutative ring spectrum $R$ and a differential refinement   $(R,A,c)$. We assume that $A$ is split. { In this case we consider   the complex 
 \begin{equation}\label{apr3004}B :=   \dots \to A^{-2}\to A^{-1}\to {0 \to 0 \to ...}\ ,
 \end{equation} where $A^{-1}$ sits in degree $-2$. The complex $B$   will serve as a real model of $\bgl_{1}(R)\one$.} 
In the following theorem we use the term canonical differential data as in \cite{bg}.
Recall {that} every spectrum admits   canonical differential data. But  in general such data is not unique.
So the main achievement  in the first part of the theorem is contained in the word {\em prefered}. 
{
\begin{theorem}\label{uliapr0720}
{Let $(R,A,c)$ be a differential data with $A$ split.}
The spectrum $\bgl_{1}(R)\one$ admits a prefered canonical differential {data}%\tfoot{Ich bin mir nicht sicher, aber ist data strikt genommen nicht plural? Manche benutzen das auch als singular. Kommt oft vor im Papier...\uli{Hat der Gepner so vorgeschlagen. Der wirds wohl im Gef\"uhl haben.}}
$$ \widehat{\bgl_{1}(R)\one}:=\big(\bgl_{1}(R)\one,B,e^{-1}\big)\ .$$ Furthermore, there exists a canonical map of  {Picard-$\infty$ stacks}
\begin{equation}\label{apr3040}{ \Omega^{\infty}\Diff^0_1( \widehat{\bgl_{1}(R)\one}) \to \Tw_{\hat R}}\end{equation}

%\tho{If $R$ is differentially simple then there is a prefered canonical refinement 
%$$ \widehat{\pic_R}:=\big(\pic_R,B,e^{-1}\big)\ .$$ 
%and }
\end{theorem}
}

{The remainder of the present section is devoted to the proof of this theorem. {In particular, the map $e$ will be defined in Proposition \ref{thoapril162}.}
Let us start with more details about the spectra $\Diff^{n}_{k}(p,B,e)$ in the general situation.}

\begin{lem}\label{thoapr16}
\begin{enumerate}
 \item 
The homotopy groups of $\Diff^n_k(p,B,e){(M)}$ are given by 
$$\pi_{i}(\Diff^{n}_k(p,B,e){(M)})\cong
\left\{\begin{array}{cc} 
0& i < 0\\
p^{n-i}(M) &  0 \leq i < k \\
\widehat{p}^{n-k}(M)& i= k\\
p\R/\Z^{n-i-1}(M)& i > k \end{array}\right.\ .$$ 
\item
{The sheaf of connective spectra}
$\Diff^n_k(p,B,e)$ can be written as the pullback in {$\Sh_{\Sp_{\ge 0}}$} 
\begin{equation}\label{apr11}
 \xymatrix{
 \Diff^n_k(p,B,e) \ar[d]\ar[r]   &   {H(\Omega B\langle n-k,\dots,n\rangle)[n]  )}\ar[d] \\
{(\underline{p[n]})}_{\geq 0}\ar[r]^{e} &  ({\underline{ (HB)[n]}})_{\geq 0}
}
\end{equation}
where $\Omega B\langle n-k,\dots,n\rangle$  is   {the sheaf in  $\Sh_{\Ch^{\le n},\infty}$ } 
$$ {\Omega B}^{n-k} \to {\Omega B}^{n-k + 1} \to \ldots \to {\Omega B}^{n-1} \to Z^n(\Omega B)\ ,$$
{where $Z^{n}(\Omega B)$ sits in degree $n$}.
 \item
There are natural morphisms 
$$ \Diff^n_0(p,B,e) \to \Diff^n_1(p,B,e) \to \Diff^n_2(p,B,e) \to ...$$
%with $\underrightarrow{\lim} \Diff^n_k(P,B,e) = \underline{\Sigma^n P} $ \marginpar{\uli{Brauchen wir diese Aussage?}\tho{Nein, k\"onnen wir l\"oschen sp‚Äö√Ñ√∂‚àö√ë‚àö‚àÇ‚Äö√†√∂‚Äö√Ñ‚Ä†‚Äö√†√∂‚Äö√†√á¬¨¬®¬¨¬Æ‚Äö√†√∂‚àö¬∫ter}}
\end{enumerate}
\end{lem}
\begin{proof}
{1.} The   statement follows immediately from the calculation of  homotopy groups {of $\Diff^{n}(p,B,e)$} as listed in \eqref{apr0505}. 

\bigskip

{2.} Note that shifting and taking connective {covers} both commutes with  pullbacks. Thus applying the two operations to the pullback diagram \eqref{apr0501} yields {\eqref{apr11}}.

\bigskip

{3.} This result finally follows from the description of $\Diff^n_k(p,B,e)$ as a   pullback together with the obvious inclusions	%\footnote{\uli{Habe shifts in $B$ eingebaut, damit die Notation pa{\ss}t.
%Eigentlich w\"are so was wie $(\Omega B[-n])\langle -n,k-n\rangle$ besser.}} 
$${
\Omega B\langle  n,\dots,n\rangle \to 
\Omega B\langle  n-1,\dots,n\rangle    \to \Omega B\langle  n-2,\dots,n\rangle \to ...}
$$
\end{proof}
 {
\begin{example}
Let $P = H\Z$ be the ordinary Eilenberg-Mac Lane spectrum,   $B:=\R[0]$ and $e:H\Z\wedge H\R\to HB$ be the canonical map. Then the spectrum
$\Diff^n_k(P,B,e)$ has the following interpretation for various values of $n$ and $k$:
\begin{enumerate}
 \item $\Diff^2_0(P,B,e)(M)$ is a 1-type as we see from the homotopy groups. As such it can be identified with the groupoid of line bundles with connection.
 \item $\Diff^2_1(P,B,e)(M)$ is   the 1-type, namely the groupid of line bundles without connection.
 \item $\Diff^3_0(P,B,e)(M)$ is a 2-type. It can be identified with the $2$-groupoid of gerbes with connection.
 \item $\Diff^3_1(P,B,e)(M)$ is a 2-type which can be identified with the $2$-groupoid of gerbes with connection and invertible gerbe bimodules.
 \item $\Diff^3_2(P,B,e)(M)$ is a 2-type which can be identified with the $2$-groupoid of gerbes without connection.
\end{enumerate}
\end{example}
%\marginpar{\uli{Diese Aussagen w\"aren zu begr\"unden. Kann man jedenfalls nicht mit dem Text des papers einsehen. Ich w\"urde die nicht machen.}}
}
\newcommand{\bDiff}{\mathbf{Diff}}

We now turn our attention to the case $\Diff^0_1(p,B,e)$ {appearing in the theorem}. 
We give a categorial interpretation of the morphism
$$ {H(\Omega B\langle -1,\dots,0 \rangle) \cong H(\Omega B\langle -1,\dots,0 \rangle)_{\ge 0}  \to  H(\Omega B)_{\ge 0}\cong (\underline{HB})_{\geq 0}}\ .$$
%$$\bZ^{0}(\Omega B)\to \tau^{\le}\Omega B$$ in $\Sh_{\Ch_{\infty}^{\le}}$, 
 {which is the right vertical map in the diagram  \eqref{apr11} that defines $\Diff^0_1(p,B,e)$. }
Note that ${(\underline{HB})_{\geq 0}}$ is a homotopy invariant sheaf {of connective spectra.} % like any smooth function spectrum.
The canonical inclusion {of homotopy invariant sheaves into all sheaves fits into an adjunction 
$$ i^{h}:\Sh^h_{\Sp_{\le 0}} \leftrightarrows \Sh_{\Sp_{\le 0}}\ ,$$%\footnote{\uli{$i$ klingt so nach inclusion!}}
and we call $i^{h}$ the homotopification. Given a sheaf $\cF$ we shall also refer to the unit $\cF\to i^{h}\cF$ as homotopification.}

\begin{lem}\label{thoapr17}
The morphism\
%footnote{\uli{Hier sollte man sich das $H$ sparen. Sp\"ater nimmt man eh sogar $\Omega^{\infty}H$. Ich w\"urde in $\Sh_{\Ch_{\le 0,\infty}}$ arbeiten.}} 
{$ H(\Omega B\langle -1,\dots,0\rangle) \to (\underline{HB})_{\geq 0}$} is equivalent to the homotopification.  
\end{lem}
\begin{proof}
{Follows immediately from Lemma 7.15 of \cite{diffsheaf} together with the fact {that} the Eilenberg-{MacLane} functor from chain complexes to spectra commutes with homotopification, and the de Rahm equivalence.} %\footnote{\tho{Alten Beweis auskommentiert}}
\end{proof}

We now {specialize to the situation of the theorem and set $p:=\bgl_{1}(R)\one$ and $B$ as in \eqref{apr3004}. } 
 Note that
$$ \Omega B\langle -1,\dots,0\rangle\cong F^{1} \Omega A^{0}\to Z^{1}(F^{2}\Omega A) \ .$$ 
In Definition \ref{apr2901} this two-step complex gave rise to the  Picard stack  $\mathbf{Z}^1_A$.  {By Remark \ref{apr3001} we  have  a} canonical equivalence 
\begin{equation}\label{apr2910}\Omega^{\infty} H(\Omega B\langle -1,\dots,0\rangle)\cong  \bZ_{A}^{1}\ \end{equation} of Picard-$\infty$-stacks.

\bigskip

   Recall that we constructed a transformation {(restrict \eqref{ulimar2161n} to $\{\underline{A}\}\times \bZ_{A}^{1}$)}
$$\cM:\mathbf{Z}^{1}_{A}\to \PicwLf_{\Omega A}\ , \quad \omega\mapsto \cM(\underline{A},\omega)\ .$$
of  {Picard-$\infty$-stacks.}
\begin{prop}\label{thoapril162}
There {is a} {canonical} equivalence of spectra
\begin{equation}\label{apr3003}e:  HB \to  \bgl_{1}(R)\one\wedge H\R \end{equation}
 and a commutative diagram in $\Sh_{\CGrp(\cS)}$
\begin{equation}\label{aprdi}
{\xymatrix{
\mathbf{Z}^{1}_{A} \ar[rr]^-\cM \ar[d] && \PicwLf_{\Omega A} \ar[d] \\
\underline{\Omega^{\infty} HB }\ar@{..>}[rr]\ar[rd]^-{\Omega^{\infty}e} & &  \PicL_{H\Omega A}\ar[d]^{\cong}%_{Lemma \ref{ulimar2031}}
\\ &\underline{\Omega^{\infty}(\bgl_{1}(R)\one\wedge H\R}) \ar[ur]\ar[r]&\underline{\Pic_{HA}}
}}\ .
\end{equation}
 
\end{prop}
\begin{proof}
From Lemma \ref{thoapr17}, {the equivalence of categories  $\Sh_{\Sp_{\ge  0}}\cong \Sh_{\CGrp(\cS)}$
and \eqref{apr2910}} we conclude that the left vertical morphism in diagram \eqref{aprdi} is the homotopification. Note that the sheaf {$\PicL_{H\Omega A}\cong \PicL_{\underline{HA}}$} in the lower right corner of diagram \eqref{aprdi} is homotopy invariant.% by Lemma \ref{ulimar2031}.
{ We get the dotted arrow 
$$ \underline{\Omega^{\infty} HB} \to  \PicL_{H\Omega A} $$ and the upper square 
from the universal property of the {homotopification}.
}

\bigskip
{
The constant functor 
$$\CGrp(\cS)\to \Sh_{\CGrp}\ , \quad  X\mapsto \underline{X}$$
is fully faithful. Hence the map
$$\underline{\Omega^{\infty}HB}\to \PicL_{H\Omega A}\cong \underline{\Pic_{HA}}$$
induces the arrow marked by $!$ in the following diagram in $\CGrp(\cS)$
\begin{equation}\label{apr3002}\xymatrix{ &\Omega^{\infty}(\bgl_{1}(R)\one\wedge H\R)\ar[d]
\\\Omega^{\infty} HB\ar@{..>}[ur]^{\Omega^{\infty}e} \ar[r]^{!}&\Pic_{HA} } .\end{equation}
We now observe that
$$ \bgl_{1}(R)\one\wedge H\R\cong \bgl_{1}(R\wedge H\R)\one\cong \bgl_{1}(HA)\one\cong {\pic}_{HA}\one$$
and that{, under this identification,} the right vertical arrow in \eqref{apr3002} is just the structure map of a $1$-connected covering. Since $\Omega^{{\infty}} HB$ is one-connected the triangle \eqref{apr3002} is unique up to equivalence. It induces corresponding triangle in \eqref{aprdi}. 
}%\textcolor{green}{From the results in Section \ref{sep0602} we see that upper right composition even factors through the subsheaf  $\underline{bgl_{1}(R)\one\wedge H\R} \subset   \uli{\PicIc_{H\Omega A}}$} thus we get the map
%$$
%E: \underline{HA\zero[-1] } \to \underline{bgl_{1}(R)\one\wedge H\R}
%$$
%which is unique in the stated sense. The source and the target are both homotopy invariant sheafs, thus $E$ is induced by a map 
%$$
%e: HA\zero[-1]  \to bgl_{1}(R)\one\wedge H\R
%$$
%of spectra which is uniquely determined by $E$. It remains to show that $e$ is an equivalence. But this follows from Proposition \ref{thm_commutes}.
\end{proof}

 \begin{ddd}\label{refbgl1}
Let $(R,A,c)$ be a differential ring spectrum. Then the equivalence \eqref{apr3003}  determines a prefered canonical differential data
$$ \widehat{\bgl_{1}(R)\one} := (\bgl_{1}(R)\one,B,e^{-1})$$ 
\end{ddd}
{
\begin{rem2}\label{uuu100}
In homotopy the map $e$  induces the map
$Exp$ discussed {in} Proposition \ref{thm_commutes}. But note that this proposition only determines the homotopy class of $e$, while here we fix the  map itself up to contractible choice. 
This fact is important to make the domain of the transformation \eqref{apr3040} well-defined.
\end{rem2}
}

We now construct 
the asserted canonical map of {Picard-$\infty$ stacks}
$${ \Omega^{\infty}\Diff^0_1( \widehat{\bgl_{1}(R)\one}) \to \Tw_{\hat R}}\ .$$
According to Lemma \ref{thoapr16} %\footnote{\uli{Weiss nicht warum der Link ins Leere geht!}} 
the {Picard-$\infty$ stack  $\Omega^{\infty}\Diff^0_1( \widehat{\bgl_{1}(R)\one})$} is defined as the pullback of the diagram
\begin{equation*}
 { \underline{\Omega^{\infty} \bgl_{1}(R)\one} \to \underline{\Omega^{\infty} HB } \leftarrow \mathbf{Z}^1_A}
\end{equation*}
Recall that  {$\Tw_{\hat R}$} was defined {as the pull-back of Picard-$\infty$-stacks}
\begin{equation*}
{\PicL_{\underline{R}} \to \PicL_{H\Omega A}   \leftarrow  \PicwLf_{\Omega A}  \ .}
\end{equation*}
Thus it suffices to  {exhibit a canonical morphism between} the two diagrams. {To this end we consider the following diagram:} \begin{equation*}{
\xymatrix{
\underline{\bgl_{1}(R)\one} \ar[r] \ar@{.>}[dr]\ar[dd]& \underline{HB}\ar[d]_{\cong}^e&  \mathbf{Z}^1_A \ar[l]\ar[dd]^\cM \\
 & \underline{\bgl_{1}(R)\one\wedge H\R} \ar[d]& \\
\PicL_{\underline{R}} \ar[r] & \PicL_{H\Omega A} &   \PicwLf_{\Omega A}   \ar[l] 
}}\ .
\end{equation*}
The right-hand part  comes from {\eqref{aprdi}. The dotted arrow is the realification map which induces the left upper horizontal map since $e$ is an equivalence. The left lower square commutes since it is equivalent to the sheafification of the realification applied to the  one-connective covering map $\Pic_{R}\one\to \Pic_{R}$. 
}

\bigskip 
 
{This finishes the proof of Theorem \ref{uliapr0720}} \hB 
 
\begin{kor}{
There is a canonical map of Picard-$\infty$ stacks
\begin{equation}\label{fjfjkwehfjwkehfkwefi8wefuwefwef} \Omega^{\infty}\Diff^0_0( \widehat{\bgl_{1}(R)\one}) \to \Tw_{\hat R}\end{equation}}
\end{kor}
\begin{proof}
Compose the canonical map {$$\Omega^{\infty}\Diff^0_0( \widehat{\bgl_{1}(R)\one}) \to \Omega^{\infty}\Diff^0_1( \widehat{\bgl_{1}(R)\one})$$} from Lemma \ref{thoapr16} with the canonical map {$$ \Omega^{\infty}\Diff^0_1( \widehat{\bgl_{1}(R)\one}) \to  \Tw_{\hat R}$$} from Theorem \ref{uliapr0720}
\end{proof}

\begin{rem2}
 {
It is a natural question   whether  one can extend the map  \eqref{fjfjkwehfjwkehfkwefi8wefuwefwef} to a map from a 
  differential version of $\Tw_R$. In fact this can be done, but one has to make some additional choices. We will refrain from doing so here. }
\end{rem2}
\section{Examples of twists}
%%%%%%%%%%%%%%%%%%%%%%%%%%%%%%%%%%%%%%%%%%%%%%%%%%%%%%%%%%%%%%%%%%%%%%%%%%%%%%%%%%%%%%%%%%%%%%%%%%%%%%%%%%%%%%%%%%%%%%%%%%%

We fix an integer  $n\in \nat\setminus\{0\}$ and  consider the Picard-$\infty$ groupoid $K(\Z,n)$ (corresponding to the respective Eilenberg MacLane Space).
%\ntho{$:=\Omega^{\infty}((H\Z)[n])$. Of course, its underlying space is the Eilenberg-MacLane space 
%with the single non-trivial homotopy group $\pi_{n}(K(\Z,n))\cong \Z$, but this presentation makes  %Picard-$\infty$ groupoid structure clear.} 
Applying the symmetric monoidal functor
$\Sigma^{\infty}_{+}$ from spaces to spectra we
 obtain a commutative ring spectrum $\Sigma_{+}^{\infty} K(\Z,n)$.

\begin{prop}
The commutative ring spectrum $\Sigma^\infty_+ K(\Z,n)$ is differentially simple (Definition \ref{good}). In particular it admits a canonical refinement to a differential spectrum 
$$\widehat{\Sigma^\infty_+ K(\Z,n)} := (\Sigma^\infty_+ K(\Z,n),\Sym_\R\R([n]),\kappa)$$
\end{prop}\proof
We have isomorphisms of rings
$$\pi_{*}(\Sigma_{+}^{\infty} K(\Z,n)\wedge H\R)\cong H_{*}(\Sigma_{+}^{\infty} K(\Z,n),\R)\cong H_{*}( K(\Z,n);\R)\cong \Sym_{\R}(\R[n])\ . $$ 
%\footnote{\uli{Verschiebungen sind Kohomologisch. $\R[n]^{-n}=\R$}}
Since the homotopy of $K(\Z,n)\wedge H\R$ is a polynomial algebra this ring spectrum is formal by  Remark \ref{apr3045}. It also obviously satisfies the remaining conditions for being differentially simple according to Definition \ref{good}. 
\hB  

From the universal property of the functor
%\footnote{\uli{zitieren}} 
$\GL_{1}$ we get 
  a canonical map of Picard-$\infty$ groupoids $$K(\Z,n)_{+}\to \GL_{1}(\Sigma^{\infty}_{+} K(\Z,n))\ ,$$
 which induces a map of connective spectra
 $$(H\Z)[n+1]\to \bgl_{1} (\Sigma^{\infty}_{+} K(\Z,n))\ .$$ 
 Since $(H\Z)[n+1]$ is one-connected we get a canonical factorization over a map
  \begin{equation}\label{apr3050}(H\Z)[n+1]\to \bgl_{1} (\Sigma^{\infty}_{+} K(\Z,n))\one\ . \end{equation} 
 From Theorem  \ref{uliapr0720} we get a prefered  canonical differential data
 $$\widehat{\bgl_{1} (\Sigma^{\infty}_{+} K(\Z,n)\one}:=(\bgl_{1} (\Sigma^{\infty}_{+} K(\Z,n)\one,(\Sym_{\R}\R([n])[1]) ,e^{-1}) \ .$$
 The spectrum $(H\Z)[n+1]$ has a canonical differential data
 $$\widehat{(H\Z)[n+1]}:=((H\Z)[n+1],\R[n+1],c)\ .$$ 
 %By \cite[Sec. 2.4]{bg} 
 \begin{lem} The map of spectra \eqref{apr3050}
  induces a unique equivalence class of maps of canonical differential data 
 $$\widehat{(H\Z)[n+1]}\to \widehat{\bgl_{1} (\Sigma^{\infty}_{+} K(\Z,n)})\ .$$
 \end{lem}
 \proof
 By \cite[Sec. 2.4]{bg}  there exists such a map whose equivalence class  is unique up to the action of 
 $$\prod_{k\in \Z} \Hom((\R[n+1])^{-k},(\Sym_{\R}\R([n])[1])^{-k-1})\ .$$
 Since this group is trivial the assertion follows. \hB

 This   map of differential data induces an  equivalence class of maps of differential function spectra
 $$\Diff^{0}_{1}(\widehat{(H\Z)[n+1]})\to \Diff^{0}_{1}(\widehat{\bgl_{1} (\Sigma^{\infty}_{+} K(\Z,n)})\ .$$
 
 If we compose this map with the map given in  Theorem \ref{uliapr0720}, then we get the following maps
  between Picard-$\infty$-stacks
 $$\Omega^{\infty} \Diff^{0}_{0}(\widehat{(H\Z)[n+1]}) \to  \Omega^{\infty} \Diff^{0}_{1}(\widehat{(H\Z)[n+1]})\to \Tw_{\widehat{  \Sigma^{\infty}_{+} K(\Z,n)} }\ .$$ 
 
 Let $M$ be a smooth manifold.
We obtain a commutative diagram of ordinary Picard categories 
$$\xymatrix{\widehat{H\Z}^{n+1}(M)\ar@/_3.5cm/[ddd]^{\curv}\ar[r]^{I}&H\Z^{n+1}(M) \ar[r]^{(i)}& \pi_{0}(\PicL_{\underline{ \Sigma^{\infty}_{+} K(\Z,n)}}(M))\\\Ho \Omega^{\infty} \Diff^{0}_{0}(\widehat{(H\Z)[n+1]})(M)\ar[r]^{\psi}\ar[d]^{!} \ar[u]^{\pi_{0}}& \Ho \Omega^{\infty} \Diff^{0}_{1}(\widehat{(H\Z)[n+1]})(M)\ar[r]\ar[d]_{(ii)}^{!}\ar[u]^{\pi_{0}}& \Ho \Tw_{\widehat{  \Sigma^{\infty}_{+} K(\Z,n)}}{(M)}\ar[d]\ar[u]\\
\Ho \Omega^{\infty} HB\langle 0,\dots,0\rangle(M)\ar[r]\ar[d]^{\cong}&\Ho \Omega^{\infty} HB\langle -1,\dots,0\rangle(M)\ar[r]\ar[d]^{\cong}&\PicwLf_{\Omega \Sym_{\R}(\R[n])}(M)\ar@{=}[d]\\ Z^{n+1}(\Omega(M))\ar[r] &(\Omega^{n}(M)\to Z^{n+1}(\Omega(M)))\ar[r]^{(iii)} &\PicwLf_{\Omega \Sym_{\R}(\R[n])}(M)}\ ,$$
where the arrows marked by $!$ are induced by the upper horizontal  in \eqref{apr11}.
%\tfoot{Das ist echt hart zu verstehen. Sollten wir nicht vielleicht noch die Notation erkl‚àö¬ßren?  \uli{Sehe nicht, was hier noch nicht erkl\"art ist.}}

 \begin{rem2}\label{may0201}
Let us spell out what this diagram means.   Let us call an object $$\hat x\in \Ho \Omega^{\infty} \Diff^{0}_{0}(\widehat{(H\Z)[n+1]})(M)$$ a differential cycle for differential cohomology class $[\hat x] \in \widehat{H\Z}^{n+1}(M)$. Let $\curv([\hat x])\in Z^{n+1}(\Omega(M))$ and $I([\hat x])\in H\Z^{n+1}(M)$ denote the curvature and the underlying integral cohomology class of $[\hat x]$.
\begin{enumerate}
\item   {We choose an identification} $\Sym_{\R}(\R[n])\cong  \R[b]$, {where $b\in  \Sym_{\R}(\R[n])$ a non-zero element  of degree $-n$.}  The functor $(iii)$ maps a form $\omega\in Z^{n}(\Omega (M))$ to the sheaf of $DG$-$\Omega \R[b]_{|M}$-modules  $(\Omega \R[b]_{|M},d+b\omega)$ and a form $\beta \in \Omega^{n}(M)$ to the map
$$\exp(-\beta b):(\Omega \R[b]_{|M},d+b\omega)\to (\Omega \R[b]_{|M},d+b\omega^{\prime})\ ,$$ where $\omega^{\prime}:=\omega+d\beta$.
\item The second horizontal line associates  to a  differential cycle $\hat x$ a differential twist which we denote by  $(E,\cM,d)$. 
\item The equivalence class of this differential twist, or what is the same by Theorem
\ref{unqiue}, the equivalence class of the underlying  $  \Sigma^{\infty}_{+} K(\Z,n)$-twist $E$,
 is determined by the underlying integral cohomology class $I([\hat x])$.
\item  The differential cycle   $\hat x$ determines a closed form $\omega\in \curv([\hat x])\in Z^{n+1}(\Omega)$.
The lower horizontal line of the diagram  maps this form to the  sheaf of $DG$-$\Omega \R[b]_{|M}$-modules
 $$\cM:= (\Omega\R[b]_{|M},d+b \omega)$$
which is part of the datum of the differential twist $ (E,\cM,d)$. 
\item
We now observe that the differential cohomology class $[\hat x]\in \widehat{H\Z}^{n+1}(M)$ determines the equivalence class of the  $\Sigma^{\infty}_{+} K(\Z,n)$-twist $E$ and the complex $\cM$.
{This} additional information contained in the cycle $\hat x$ is needed to fix the map $d$.%\footnote{\uli{Habe es zuerst mit dem de Rham differential verwechselt.}}
\item Let us now fix a second differential cycle $\hat x^{\prime}$ which gives rise to a differential twist
$(E^{\prime},\cM^{\prime},d^{\prime})$, where $\cM^{\prime}:=(\Omega \R[b]_{|M},d+b \omega^{\prime})$
with $\omega^{\prime}:=\curv([\hat x^{\prime}])$. We further consider a morphism
$f:\psi(\hat x)\to \psi(\hat x)^{\prime}$. This morphism induces an equivalence morphism of differential twists $(E,\cM,d)\to (E^{\prime},\cM^{\prime},d^{\prime})$.  
\item 
Using the diagram we can explicitly describe
the induced map $\cM\to \cM^{\prime}$. The map $(ii)$ associates to the morphism $f$ a form
$\beta\in \Omega^{n}(M)$ such that $\omega+d\beta=\omega^{\prime}$.  Therefore the morphism $\cM\to \cM^{\prime}$ is multiplication by $\exp(-b\beta)$.
If the morphism $f$ is in the image of $\psi$, then $\cM=\cM^{\prime}$ and it acts as the identity.
\item Since $\pi_{1}(\Diff^{0}_{{1}}(\widehat{(H\Z)[n{+1}]})(M))\cong \widehat{H\Z}^{n}(M)$
the second line induces   an  action of the differential cohomology group  $\widehat{H\Z}^{n}(M)$ 
by automorphisms on differential $\Sigma^{\infty}_{+} K(\Z,b)$-twist. On $\cM$ this action is given by multiplication by $\exp(-b \ \curv(y))$. 
\end{enumerate}
 \end{rem2}

\newcommand{\bX}{\mathbf{X}}
\begin{prop}\label{prop_refine}
Let $R$ be a commutative ring spectrum and  $\widehat R:=(R,A,c)$ be a differential refinement of $R$.
A morphism of commutative ring spectra  
$$
g: \Sigma^\infty_+ K(\Z,n) \to R
$$
  can  be extended to a morphism of differential ring spectra
 $$
\widehat g: \widehat{\Sigma^\infty_+ K(\Z,n)} \to \widehat R$$
whose equivalence class is unique up {to} the action of the group $A^{n-1}/\im(d)$.
  \end{prop}
\proof  
 The space of morphism of differential ring spectra from $\widehat{\Sigma^\infty_+ K(\Z,n)}$ to $\widehat R$  is given by a pull-back 
 $$\xymatrix{\Map_{\widehat{\Rings}}(\widehat{\Sigma^\infty_+ K(\Z,n)},\widehat R)\ar[d]\ar[r]&\Map_{\CAlg(\Ch_{\R})}(\R[b],A)\ar[d]\\
 \Map_{\CAlg(\Sp)}(\Sigma^\infty_+ K(\Z,n),R)\ar[r]& \Map_{\CAlg(\Mod_{H\R})}(H\R[b],HA)
}$$ 
We must calculate the fibre of the left vertical map at $g$ which can be identified with the fibre $\bX$ of
the right vertical map at the image of $g$. Since $\CAlg(\Ch_{\R})$ is a $1$-category the right upper corner is discrete. Since $\R[b]$  is freely generated by an element $b$ of cohomological degree $-n$ we get 
$$\pi_{k}(\Map_{\CAlg(\Ch_{\R})}(\R[b],A))\cong \left\{\begin{array}{cc} Z^{-n}(A)&k=0\\
0&k\ge 1\end{array}\right.$$
and
\begin{eqnarray*}
\Map_{\CAlg(\Mod_{H\R})}(H\R[b],HA)&\cong&
\Map_{\CAlg(\Ch_{\R,\infty})}(\iota(\R[b]),\iota(A))\\&\cong&
\Map_{\Ch_{\R,\infty}}(\iota(\R[n]),\iota(A))\\
&\cong&\Omega^{\infty} ((HA)[-n])
 \end{eqnarray*}
and therefore
$$ \pi_{k}( \Map_{\CAlg(\Mod_{H\R})}(H\R[b],HA){)}\cong H^{-n-k}(A)\ .$$
We deduce that for $k\ge 1$
$$\pi_{k}(\bX)\cong H^{-n-k}(A)$$
and $\pi_{0}(X)$ fits into the exact sequence
$$0\to H^{-n-1}(A)\to \pi_{0}(\bX)\to B^{n}(A)\ \to 0\ .$$
In particular, $\bX$ is not empty, and its components are parametrized by $\pi_{0}(X)\cong A^{-n-1}/\im(d)$. 
%We consider the space $\bX$ of maps of differential rings $\widehat{\Sigma^\infty_+ K(\Z,n)} \to \widehat R$ refining $g$, i.e. the space of commutative diagrams
%$$\xymatrix{ \widehat{\Sigma^\infty_+ K(\Z,n)}\ar[d]^{\kappa}\ar[r]^{g}&R\ar[d]^{c}
%\\ H\R[b]\ar[r]^{H(\tilde g)}&HA
%}$$ in $\CAlg(\Sp)$.
%Then we have a fibre sequence
%$$\Omega^{1}\Map_{\CAlg(\Sp)}(\Sigma^\infty_+ K(\Z,n),HA)\to \bX \to \Map_{\CAlg(\Mod_{H\R})}(H\R[b],HA) ,$$
%where the second map is the evaluation of the diagram at the lower horizontal line, and
%the fibre is the space of fillers. We calculate these two spaces explicitly.
%We have
%\begin{eqnarray*}
%\Map_{\CAlg(\Mod_{H\R})}(H\R[b],HA)&\cong&
%\Map_{\CAlg(\Ch_{\R})}(\R[b],A)\\&\cong&
%\Map_{\Ch_{\R}}(\R[n],A)\\
%&\cong&\Omega^{\infty} ((HA)[n])
% \end{eqnarray*}
%In particular, for $k\in \nat$ we have
%$$\pi_{k}(\Map_{\CAlg(\Mod_{H\R})}(H\R[b],HA))\cong H^{n-k}(A)\cong \pi_{k-n}(R)\ .$$
% We furthermore calculate  
%\begin{eqnarray*}
%\Map_{\CAlg(\Sp)}(\Sigma^\infty_+ K(\Z,n),HA)&\stackrel{adjunction}{\cong}&\Map_{\CGrp(\cS}(K(\Z,n), \GL_{1}(HA))\\
%&\stackrel{n\ge 1}{\cong}&\Map_{\CGrp(\cS)}(K(\Z,n),\GL_{1}(HA)\one)\\
%&\stackrel{\exp}{\cong}&\Map_{\CGrp(\cS)}(K(\Z,n), \Omega^{\infty}HA\one )\\
%&=&\Map_{\Sp} ((H\Z)[-n], HA\one)\\
%&=&\Omega^{\infty}(HA\one[n]) 
%\end{eqnarray*}
%We get
%$$\pi_{0}(\Omega^{1}\Map_{\CAlg(\Sp)}(\Sigma^\infty_+ K(\Z,n),HA)\cong \pi_{1}(\Omega^{\infty}(HA\one[-n]))\cong \pi_{1+n}(R)\otimes \R\ .$$
%In view of the long exact sequence in homotopy this calculation implies that $$\bX\to \pi_{0}(\bX)\cong  \pi_{-n}(R)$$
%is an equivalence. \hB
\hB 
%The map $g_{*}:\pi_{*}(\widehat{\Sigma^\infty_+ K(\Z,n)} )\to \pi_{*}(R)$ together with the identification
%$$\pi_{*}(\widehat{\Sigma^\infty_+ K(\Z,n)} )\otimes \R\cong \R[b]$$
%induces a map of algebras $g_{\R,*}:\R[b]\to H_{*}(A)$. Any choice of a cycle $\tilde b\in A^{-n}$
%for $g(b)$ determines a map of CDGA's
%$\R[b]\to A$. By Definition \ref{may0101} a map of differential ring spectra  $\widehat{g}$ over $g$ is a commutative diagram
%$$\xymatrix{ \widehat{\Sigma^\infty_+ K(\Z,n)}\ar[d]^{\kappa}\ar[r]^{g}&R\ar[d]^{c}
%\\ H\R[b]\ar[r]^{H(\tilde g)}&HA
%}$$ Let $\bX$ denote the space of such diagrams. Then we have a fibre sequence
%$$\Omega^{1}\Map_{\CAlg(\Sp)}(\Sigma^\infty_+ K(\Z,n))\to \bX \to \Map_{\CAlg(\Mod_{H\R})}(H\R[b],HA) .$$
% 
% 
%Follows from the fact that $\mathrm{GrSym}_\Q(x)$ is free on $x$.
%\end{proof}
\begin{kor}\label{may0205}
If $A^{-n-1} = 0$ in the above setting, then the extension $\widehat g$ is unique {up to non-canonical equivalence}.
\end{kor}
%Assume now that  there is an interesting map of spectra $\Sigma^\infty_+K(\Z,n) \to R$. Then Proposition \ref{prop_refine} shows how one can produce differential twists, and therefore also twisted differential $R$-cohomology groups given 
%concrete differential geometric data (such as gerbes with connection) on $M$. Putting everything together we have shown:
Assume that we have  a morphism of commutative ring spectra  
$
g: \Sigma^\infty_+ K(\Z,n) \to R$ {and} a differential refinement $(R,A,c)$ with $A^{-1-n}= 0$. Then by composing the morphism $\Tw_{\widehat{\Sigma^\infty_+ K(\Z,n)}}\to \Tw_{\widehat R}$ induced by the map $
\widehat g$ from Proposition \ref{prop_refine} we obtain
 a morphism of Picard-$\infty$-stacks  $${\underline{\Z}}\times \Omega^{\infty}\Diff^{0}_1(\widehat{H\Z[n+1]}) \to \Tw_{\widehat R}(M)\ ,$$ where the additional $\Z$-factor introduces degree shifts.
  Restricting the grading of the twisted function spectrum (Definition \ref{ulimar1901})  to $\Z\times \Omega^{\infty}\Diff^{0}_1(\widehat{H\Z[n+1]}{)}$
 thus yields a sheaf of graded ring spectra 
 $$M \mapsto ({\underline{\Z}}\times \Omega^{\infty}\Diff^{0}_1(\widehat{H\Z[n+1]}) (M), \Diff(\dots)(M))$$
 and a functor 
 \begin{equation*}\Mf^{op} \to \mathrm{GrRing}\ ,\quad 
M \mapsto \big(\Z\times \Ho\Omega^{\infty}\Diff^{0}_1(\widehat{H\Z[n+1]})(M) , \widehat R^{\dots}(M)\big) \tfoot{Ich darf hier wahrscheinlich nichts ueber Gerben und Bimoduln etc sagen, oder?}
\end{equation*}

 \begin{rem2}\label{may0206}
Using Remark \ref{may0201} (and the notation introduced there) we obtain the following.
\begin{enumerate}
\item For every differential cycle  $\hat x \in \Omega^{\infty} \Diff^{0}_{0}(\widehat{H\Z[n+1]})(M)$ and integer $k\in \Z$ we obtain an abelian  group $\widehat R^{k+\psi(\hat x)}$.
\item The curvature map for this group takes values in 
$Z^{k}(\Omega A,d+\omega b)$, where $\omega=\curv([\hat x])${,} and we use that $A$ is an $\R[b]$-module.
\item For every morphism $f:\psi(\hat x)\to \psi(\hat x^{\prime})$ we get an isomorphism
$ \psi(f):\widehat R^{k+\phi(\hat x)}\to\widehat R^{k+\phi(\hat x^{\prime})}$ and a form $\beta\in \Omega^{n-1}$ such that $d\beta+\omega=\omega^{\prime}${,} and
$$\xymatrix{\widehat R^{k+\psi(\hat x)}\ar[r]^{\psi(f)}\ar[d]^{\curv}&\hat R^{k+\psi(\hat x^{\prime})}\ar[d]^{\curv}\\
Z^{k}(\Omega  A,d+\omega b)\ar[r]^{\exp(-\beta b)}&Z^{k}(\Omega A,d+\omega^{\prime} b)
}$$ commutes.
\item For a pair of differential cycles $\hat x,\hat x^{\prime}$ and integers $k,k^{\prime}\in \Z$ we have a diagram
$$
\xymatrix{\widehat R^{k+\psi(\hat x)}\otimes \widehat R^{k^{\prime}+\psi(\hat x^{\prime})}\ar[r]\ar[d]^{\curv\otimes \curv}&\widehat R^{k+k^{\prime}+\psi(\hat x)+\psi(\hat x^{\prime})}\ar[d]^{\curv}\\
Z^{k}(\Omega A,d+\omega b)\otimes Z^{k^{\prime}}(\Omega A,d+\omega^{\prime}b)\ar[r]^/4mm/{mult}&Z^{k+k^{\prime}}(\Omega A,d+\omega b+\omega^{\prime}b)}\ .$$
\end{enumerate}

 \end{rem2}
%Combining this construction with the following.
% 
%\begin{corollary}
%In the setting of the Theorem we obtain a twisted differential cohomology group
%$\widehat{R}^{\mathcal{G}}(M)$
%for every object $\mathcal{G} \in \Ho(\Diff^{n+1}_1(H\Z)(M))$. Moreover the assignment

%forms a functor $\Mf^{op} \to \mathrm{GrRing}$.
%\end{corollary}

%\newcommand{\bb}{\mathbf{b}}

\begin{exa2}\label{may0220}
A trivial example where the above applies is $R:=\Sigma^{\infty}_{+} K(\Z,n)$ and $g$  the identity. But the classical example is complex $K$-theory. It is known by \cite{2011arXiv1106.5099A} that there is an {essentially} unique equivalence class of maps of Picard-$\infty$-groupoids
$K(\Z,2)\to \GL_{1}(\KU)$ (see also 1. in Example \ref{may0240}). We therefore get a canonical map
$\Sigma^{\infty}_{+} K(\Z,2)\to \KU$ of commutative ring spectra. Since $\KU$ is even, Corollary
\ref{may0205} applies. We therefore get a differential twisted $K$-theory
$\widehat{\bKU}^{\dots}(M)$ which is graded by $\Z\times \Omega^{\infty} \Diff^{0}_{1}(\widehat{H\Z[3]})(M)$.
The structures spelled out in Remark \ref{may0206} have partially been constructed using explicit models in 
\cite{MR2518992}, \cite{MR2672802}. In particular, the real approximation is given by the usual periodic de Rham complex
$\cM(\underline{\R[b,b^{-1}]},\omega b)$, where $\omega=\curv([\hat x])\in Z^{3}(\Omega)$
and $\deg(b)=-2$.

\end{exa2}

\begin{exa2}\label{may0240}
Other examples of  interesting maps of ring spectra $\Sigma^\infty_+K(\Z,n) \to R$ can be constructed    following ideas of Ando, Blumberg and Gepner \cite{MR2681757}.
We consider the Picard-$\infty$-groupoid $BO:= \Omega^{\infty} \KO\zero$.
Assume {that we are given} a sequence of maps between Picard-$\infty$-groupoids 
\begin{equation}\label{may0210}
K(\Z,n) \to X \to BO
\end{equation}
such that the composition is nullhomotopic.   The Thom spectrum construction
$$M:\CGrp(\cS)/BO\to \CAlg(\Sp)\ , (X\to BO)\mapsto MX$$
functorially associates to every Picard-$\infty$-groupoid over $BO$ a commutative ring spectrum. Since we assume that 
$K(\Z,n) \to  BO$ is null-homotopic we have an equivalence $MK(\Z,n)\cong \Sigma^{\infty}_{+} K(\Z,n)$. Hence the sequence \eqref{may0210} induces a map
\begin{equation*}
\Sigma_+^\infty K(\Z,n) \to MX
\end{equation*}

In the following we discuss examples of sequences \eqref{may0210}  from maps between classical groups. 
\begin{enumerate}
\item The central extension $$1\to U(1)\to Spin^{c}\to SO\to 1$$ together with the identification $BU(1)\cong K(\Z,2)$ and the map $BSO\to BO$
induces the sequence
$$K(\Z,2)\to BSpin^{c}\to BO\ .$$
The construction produces a morphism of commutative ring spectra
$$\Sigma_+^\infty K(\Z,2) \to MSpin^{c}\ .$$
Note that $\pi_{3}(MSpin^{c})\otimes\Q=0$ so that we can apply Corollary \ref{may0205}.
If we compose this morphism with the Atiyah-Bott-Shapiro orientation
$A:MSpin^{c}\to \KU$, then we get the map $$\Sigma_+^\infty K(\Z,2) \to \KU$$ discussed in Example \ref{may0220}. 
\item We have a fibre sequence
$$K(\Z,3)\to BString\to BSpin\ .$$
It provides a morphism of commutative ring spectra
$$\Sigma^{\infty}_{+}K(\Z,3)\to MString\ .$$
If we compose this with the string orientation  $MString\to tmf$  of the spectrum of topological modular forms  $tmf$ (constructed by Ando, Hopkins and Rezk) we get
a map
$$\Sigma^{\infty}_{+}K(\Z,3)\to tmf$$
of commutative ring spectra.

\bigskip

Since $\pi_{4}(MString)\otimes \Q=0$ and $\pi_{4}(tmf)\otimes \Q=0$ we can apply Corollary \ref{may0205} to both cases.

Further note that $\pi_{3}(MString)\otimes \Q=0$ and $\pi_{3}(tmf)\otimes \Q
=0$. Therefore in both cases for every differential $4$-cycle the curvature of the corresponding twisted $MString$ or $tmf$-cohomology takes values in the cycles of the   untwisted de Rham complexes $(\Omega (\pi_{*} (MString)\otimes \R),d)$ or
$(\Omega (\pi_{*} (tmf)\otimes \R),d)$, and morphisms between differential cycles act by the identity on the form level.
\end{enumerate}
\end{exa2}

%\uli{
%\begin{exa2}
%Differential $n$-cycles can be constructed geometrically as follows:   ......
%\end{exa2}}
%\begin{exa2}
%\begin{enumerate}
%\item
%Consider the case $n=2$ and $R = KU$ complex K-theory. The Atiyah-Bott-Shapiro orientation of complex $K$-theory refines to a 
%morphism of commutative ring spectra
%\begin{equation*}
%\mathrm{MSpin^\C} \to KU
%\end{equation*} 
%as first shown by Joacim \cite{Joachim}. Thus we obtain a composite map $\Sigma^\infty_+ K(\Z,3) \to KU$ or also to $ku$.
%\item
%The String orientation of $tmf$ is a map of commutative ring spectra
%\begin{equation*}
%\mathrm{MString} \to tmf.
%\end{equation*}
%which was constructed by Ando, Hopkins and Rezk   Thus we obtain a map $\Sigma^\infty_+ K(\Z, 3) \to tmf$.
%\end{enumerate}

%\end{exa2}

%%%%%%%%%%%%%%%%%%%%%%%%%%%%%%%%%%%%%%%%%%%%%%%%%%%%%%%%%%%%%%%%%%%%%%%%%%%%%%%%
%%%%%%%%%%%%%%%%%%%%%%%%%%%%%%%%%%%%
\section{The twisted Atiyah-Hirzebruch spectral sequence}\label{may0512}
%%%%%%%%%%%%%%%%%%%%%%%%%%%%%%%%%%%%%%%%%%%%%%%%%%%%%%%%%%%%%%%%%%%%%%%%%%%%%%%%
%%%%%%%%%%%%%%%%%%%%%%%%%%%%%%%%%%%%

We consider a commutative ring spectrum  $R$ and a differential refinement  $\widehat R = (R, A, c)$.
Assume that we are given an $R$-twist $E\in \PicL_{\underline{R}}(M)$ which is trivial on the 1-skeleton of $M$, i.e. $E$ is classified by a map $M\to \BGL_{1}(R)\one$. We have shown in Section $\ref{sec:existence}$ that there exists a twisted de Rham complex $\cM = \cM(\underline A, \omega)$ with $\omega\in Z^{1}(F^{2}\Omega A (M))$ such that $E$ refines to a differential twist $(E,\cM,d)$. 
In this section we demonstrate how the cohomology class of  $\omega$ can effectively be determined by the differentials in the Atiyah-Hirzebruch spectral sequence associated to the twist $E$. 

\bigskip

% and a topological twist $E$ which we assume for simplicity is trivial on the 1-skeleton of $M$. We have shown in Section $\ref{sec:existence}$ that one can always find a twisted de Rham complex
%$\cM = \cM(\underline A, \omega)$ such that $E$ refines to a differential twist $(E,\cM,d)$. 

%In practice it might however be hard to determine the differential form $\omega$, depending how the twist $E$ is realized. In this section we demonstrate how $\omega$ is determined by the differentials in the canonical spectral sequence associated to the twist $E$. \\

More generally, consider $\cA\in \PicLf_{\underline{A}}(M)$ and
 the  sheaf $\cM(\cA,\omega)$  of  $DG$-$\Omega A_{|M}$-modules introduced in Definition \ref{may0301}. It has a decreasing filtration
\begin{equation*}
F^{p}\cM(\cA,\omega):=\Omega^{\ge p}_{|M} \otimes_{\R} \cA 
\end{equation*}
which turns it in into a sheaf of filtered $DG$-modules over the similarly filtered sheaf of $DG$-algebras $\Omega A_{|M}$.
This filtration induces a filtration of the complex of global sections $\cM(\cA,\omega)(M)$ and therefore a spectral sequence  $({}_{\cM(\cA,\omega)}E_{r},d_{r})$ which converges to 
 $H^{*}(\cM(\cA,\omega)(M))$. 
The spectral sequence has the structure of a  module
over the corresponding spectral sequence $({}_{\Omega A}E_{r},d_{r})$ for $\Omega A_{|M}$.  

\bigskip

We let $\sing(M)$ denote the smooth singular complex of $M$ and 
$$|\sing(M)|:=\bigcup_{n\in \nat} \sing(M) [n]\times \Delta^{n}/\sim$$ %\footnote{\uli{Nehme jetzt alle Simplizes, auch die degenerierten.}}
%\tfoot{Ist es wichtig hier mit nicht-degenerierten simplizes zu argumentieren?}
be its geometric realization. 
% where $\sing(M)[n]^{red}\subseteq \sing(M)[n]$ is the subset of non-degenerated simplices.
 We have a canonical evaluation map
$\phi:|\sing(M)|\to M$ such that
$\phi(\sigma,t)=\sigma(t)$, where $\sigma:\Delta^{n}\to M$ is a point in $\sing(M)[n]$ and
$t\in \Delta^{n}$. We consider $|\sing(M)|$ as a piecewise smooth manifold on which we can consider differential forms and sheaves of spectra. The map $\phi$ is piecewise smooth. 

\bigskip

The piecewise smooth manifold $|\sing(M)|$ 
  has an increasing filtration
$$\emptyset=\sing(M)_{-1}\subseteq |\sing(M)|_{0}\subseteq |\sing(M)|_{1}\subseteq \dots$$ by skeleta.

\bigskip

Let $X\in\Sh_{\Sp}(M)$. Then we can consider the pull-back
$\phi^{*}X\in \Sh_{\Sp}(|\sing(M)|{)}$\footnote{This is actually an abuse of notation since $|\sing(M)|$ is not an object of $\Mf$. Nevertheless we think that is is clear from the context what is ment here.}.
 The filtration of $|\sing(M)|$ induces, by definition, the Atiyah-Hirzebruch spectral sequence $({}_{X}E_{r},d_{r}) $ which converges to $\pi_{*}(X(M))$. 

\bigskip

\begin{lem}\label{may0260}
We have a canonical isomorphism of spectral sequences
$({}_{\cM(\cA,\omega)}E_{r},d_{r})$ and 
$({}_{H\cM(\cA,\omega)}E_{r},d_{r})$ for $r\ge 2$. 
\end{lem}
\proof
The filtration of $|\sing(M)|$
 induces a decreasing filtration of the  complex 
$$\hspace{-0.5cm}\dots\subseteq \tilde F^{1}\cM(\phi^{*}\cA,\phi^{*}\omega)(|\sing(M)|) \subseteq \tilde  F^{0}\cM(\phi^{*}\cA,\phi^{*}\omega)(|\sing(M)|)= \cM(\phi^{*}\cA,\phi^{*}\omega)(|\sing(M)|)\ ,$$
where $\tilde F^{k}\cM(\phi^{*}\cA,\phi^{*}\omega)(|\sing(M)|)$ is the kernel of the restriction to $|\sing(M)|_{k-1}$. We let $(\tilde E_{r},\tilde d_{r})$ be the associated spectral sequence.

\bigskip

The piecewise smooth map $\phi$ induces  {a morphism} of complexes
$$\phi^{*}:\cM(\cA,\omega)(M)\to \cM(\phi^{*}\cA,\phi^{*}\omega)(|\sing(M)|)$$
and satisfies $$\phi^{*} F^{p}\cM(\cA,\omega)(M)\subseteq \tilde F^{p}(\cM(\phi^{*}\cA,\phi^{*}\omega)(|\sing(M)|))\ .$$
We therefore get a map of spectral sequences
$( {}_{\cM(\cA,\omega)}E_{r},d_{r})\to (\tilde E_{r},\tilde d_{r})$.

\bigskip

We check that it is an isomorphism on the second page and therefore on all higher pages. 
To this end we calculate
calculate $\tilde E_{2}$. We have $$\tilde E_{0}^{p,q}\cong F^{p}\cM(\phi^{*}\cA,\phi^{*}\omega)^{q+p}(|\sing(M)|)/F^{p-1}\cM(\phi^{*}\cA,\phi^{*}\omega)^{q+p}(|\sing(M)|)\ .$$ This complex calculates the relative sheaf cohomology $$H^{q+p}(\sing(M)_{p},\sing(M)_{p-1}{;} \cM(\phi^{*}\cA,\phi^{*}\omega))\ .$$ This cohomology group  is a product over the cohomology of all  %\footnote{\uli{habe das Wort non-degenerated mehrmals gestrichen}} 
$p$-simplices in $|\sing(M)|$ relative to their boundaries with coefficients in the restrictions of $\cM(\phi^{*}\cA,\phi^{*}\omega)$.
Since the data $\phi^{*}\cA$ and $\phi^{*}\omega$ can be trivialized on the simplicies, each    $p$-simplex contributes a factor of $H^{q}(A)$. Using the fact that $\omega$ belongs to $F^{2}(\Omega A(M))$ we now observe that $(\tilde E_{1}^{p,q},\tilde d_{1})\cong (C^{p}(\sing(M),H^{q}(\cA)),d_{1})$
is exactly the reduced singular complex associated to the local system $H^{*}(\cA)$.
Hence we get 
$$\tilde E_{2}^{p,q}\cong H^{p}(\sing(M){;}H^{q}(\cA))\ .$$

\bigskip

We now consider the spectral sequence $({}_{\cM(\cA,\omega)} E_{r},d_{r})$ in greater detail.
 We have \begin{equation*}
{}_{\cM(\cA,\omega)} E_{0}^{p,q}\cong (\Omega^{p}\otimes \cA^{q})(M)\ , \quad
d_{0}=1\otimes d_{\cA}\ .
\end{equation*}
We get
\begin{equation*}
{}_{\cM(\cA,\omega)} E_{1}^{p,q}\cong (\Omega^{p}\otimes H^{q}(\cA))(M)\ ,\quad
d_{1}=d_{\Omega}\otimes 1\ .
\end{equation*}
It follows that
\begin{equation*}
{}_{\cM(\cA,\omega)} E_{2}^{p,q}  \cong H^{p}(M;H^{q}(\cA))\ .
\end{equation*}

\bigskip

Since the Eilenberg-MacLane equivalence is an equivalence between sheaves of complexes and sheaves of $H\Z$-modules the descent spectral sequences for 
$({}_{H\cM(\cA,\omega)}E_{r},d_{r})$  and $(\tilde E_{r},\tilde d_{r})$
 associated to the skeletal filtration  of $|\sing(M)|$ are isomorphic.
 
 \bigskip

This finally implies the assertion. 
\hB

\begin{lem}\label{may0411}
The spectral sequence $({{}_{\Omega A}}E_{r},d_{r})$
%\tfoot{Bin verwirrt ‚àö¬∫ber die Notation, warum steht hier keine $\Omega \cA$ oder so?}
 degenerates at the second page.
\end{lem}
\proof
If $X\cong \underline{Q}_{|M}$ is the constant sheaf generated by a spectrum $Q\in \Sp$,  then 
$({}_{X}E_{r},d_{r}) $ is the classical Atiyah-Hirzebruch spectral sequence.
It is a well-known fact, that it degenerates at the second page if $Q$ is a rational spectrum. 
Since $\Omega A\cong \underline{HA}$ and $HA$ is rational we conclude
that $({}_{\Omega A}E_{r},d_{r})\cong ({}_{\underline{HA}}E_{r},d_{r})$ degenerates at the second page. 
 \hB

%We now consider the spectral sequence $({}_{\cM(\cA,\omega)} E_{r},d_{r})$ in greater detail.
% We have \begin{equation*}
%{}_{\cM(\cA,\omega)} E_{0}^{p,q}\cong (\Omega^{p}\otimes \cA^{q})(M)\ , \quad
%d_{0}=1\otimes d_{\cA}\ .
%\end{equation*}
%We get
%\begin{equation*}
%{}_{\cM(\cA,\omega)} E_{1}^{p,q}\cong (\Omega^{p}\otimes H^{q}(\cA))(M)\ ,\quad
%d_{1}=d_{\Omega}\otimes 1\ .
%\end{equation*}
%It follows that
%\begin{equation*}
%{}_{\cM(\cA,\omega)} E_{2}^{p,q}  \cong H^{p}(M;H^{q}(\cA))\ .
%\end{equation*}

\bigskip

We write 
$\omega=\sum_{n\ge 2}\omega^{n}$, where $\omega^{n} \in (\Omega^{n}\otimes
\underline{A}^{1-n})(M)$.
Let $r\in \nat $ be such that $\omega^{n}=0$ for  {all} $n<r$. Then we have
$(1\otimes d_{A})\omega^{r}=0$ and $(d_{\Omega}\otimes
1)\omega^{r}=(-1)^{r}(1\otimes d_{A}) \omega^{r+1}$. We see that $\omega^{r}$
represents a class
$[\omega^{r}]\in H^{r}(M;\underline{H^{1-r}(A)})\cong {}_{\Omega A}E_{2}^{r,1-r}$. 
 
 \bigskip
 
 In the following Lemma we write $d(\omega)_{r}$ for the differential of ${}_{\cM(\underline{A},\omega)} E_{r}^{p,q}${,} 
 and $d_{r}$ for the differential of ${}_{\Omega A} E_{{r}}^{p,q}$.
\begin{lem}\label{sep0607}
For $2\le i\le r$ we have $${}_{\cM(\underline{A},\omega)} E_{i}^{p,q}  \cong {}_{\Omega A} E_{i}^{p,q} , $$
and $d(\omega)_{i}=0 $ for all $2\le i<r$.   Furthermore,  $d(\omega)_r =d_r +[\omega^{r}]$.
\end{lem}
\proof
This is a consequence of the fact that we have a canonical identification of
the filtered $\Z$-graded vector spaces $\cM(\underline{A},\omega)(M)^{\sharp}\cong \Omega A(M)^{\sharp}$ 
and the difference of the differentials $d(\omega)-d=\omega$ increases the filtration by $r$ steps.  \hB

 \begin{lem}  \label{may0420}
 The  following assertions are equivalent:
 \begin{enumerate}
 \item
We have an  isomorphism $\cM(\underline{A},\omega)\cong \cM(\underline{A},0)$ of
$DG$-module sheaves over $\Omega A$.
\item The class $[\omega] \in H^1(\Omega A(M))$ vanishes. 
\item   The spectral sequence $({}_{\cM(\underline{A},\omega)} E_{r},d(\omega)_{r})$ degenerates at the second
page.     \end{enumerate}
 \end{lem}
\proof
We show $2.\Rightarrow 1.\Rightarrow 3.\Rightarrow 2.$

\bigskip

If $[\omega]=0$, then $\cM(\underline{A},\omega)\cong \cM(\underline{A},0)$ by
\eqref{may0310} . 

\bigskip

We now assume that $\cM(\underline{A},\omega)\cong
\cM(\underline{A},0)$.
In this case, as a consequence of the Lemmas \ref{may0260} and \ref{may0411}, the spectral sequence $({}_{\cM(\underline{A},\omega)} E_{r},d(\omega)_{r})$ degenerates at the second page.

\bigskip

We now assume that the spectral sequence $({}_{\cM(\underline{A},\omega)} E_{r},d(\omega)_{r})$   degenerates at the second page. We argue by contradiction and 
 assume that $[\omega]\not=0$.
Then there exists a maximal  $r\ge 2$ with $[\omega]\in F^{r}H^{1}(M,A)$. We can
assume that $\omega\in Z^1(F^{r}\Omega A(M))$ and
$[\omega^{r}]\in H^{r}(M,\underline{H^{1-r}(A)})={}_{\Omega A}E_{2}^{r,1-r}$ does not vanish.  The unit $1\in A$
induces a  unit $1\otimes 1\in \Omega^{0}\otimes \underline{A}$ with image $1
\in H^{0}(\Omega A(M))$. This unit is  detected by a
unit
$1_r\in E_r^{0,0}$. By Lemma \ref{sep0607} we have
$0=d(\omega)_r (1_r)=[\omega^{r}]$. This is a contradiction. \hB

\bigskip

We can now describe an algorithm which allows to determine the real approximation $\cM=\cM(\underline{A},\omega)$ of a given $R$-twist $E\in \PicL_{\underline{R}}(M)$. The main point is to determine the cohomology class of $\omega$ such that
$$(E\wedge H\R)\wedge_{H\Omega A} H\cM(\underline{A},-\omega)\cong H\Omega A\ .$$

\bigskip

Since we assume that $E$ is trivialized on the one-skeleton of $M$ we have an isomorphism
$${}_{(E\wedge H\R)\wedge_{H\Omega A} H\cM(\underline{A},-\omega)}E_{2}^{p,q}\cong H^{p}(M,\underline{H^{q}(A)})\cong {}_{\Omega A}E^{p,q}_{2}\ .$$
We write $\tilde d_{r}$ for the differential of the spectral sequence on the left-hand side.

\bigskip

Assume that we have found $ \omega\in Z^{1}(F^{2}\Omega A(M))$
such that this isomorphism extends to an isomorphism of the pages $2,\dots,r$. 
Then we consider the difference 
$(\tilde d_{r}-d_{r})(1_{r})=:x\in {}_{\Omega A}E^{r,1-r}_{2}$.
It follows from the module structure over ${}_{\Omega A}E_{r}$ that then
$\tilde d_{r}-d_{r}=x$. There exists $\kappa\in Z^{1}(F^{r}\Omega A(M))$ such that the cohomology class of $\kappa$ is detected by $x$. After replacing  $\omega$ by $\omega+\kappa$ we can assume that $x=0$ and our isomorphism of pages persists to the ${(r+1)}$th page.

\bigskip

We now proceed by induction on $r$. Since $M$ is finite-dimensional, after finitely many iteration{s}
we have found a form $\omega\in Z^{1}(F^{2}\Omega A(M))$ such that the spectral sequence for $
(E\wedge H\R)\wedge_{H\Omega A} H\cM(\underline{A},-\omega)$
degenerates at the second term. It now follows from Lemma \ref{may0420}
that $E\wedge H\R\cong H\cM(\underline{A},\omega)$. 
\hB

\begin{exa2}
We illustrate this procedure in the classical case of twisted complex $K$-theory which has already been considered in Example \ref{may0220}. In this case the procedure terminates after the first iteration and reproduces a well-known fact.

The map of Picard-$\infty$  groupoids $K(\Z,2)\to \GL_{1}(\KU)$ induces a map of spaces $K(\Z,3)\to \BGL_{1}(\K)$ by one-fold delooping. This map classifies a twist
$E\in \PicL_{\underline{\KU}}(K(\Z,3))$.

\bigskip

 The Atiyah-Hirzebruch spectral sequence $({}_{E}E_{r},d_{r})$
has been investigated in \cite[Sec. 4]{MR2307274}. Let $\eta\in H^{3}(K(\Z,3);\Z)$ denote the canonical class. It induces a class $$\eta b\in H^{1}(K(\Z,3);\R[b,b^{-1}])\ .$$   The calculation of Atiyah-Segal shows that ${}_{E\wedge H\R}E_{3}^{p,q}\cong H^{p}(K(\Z,3);\R[b,b^{-1}]^{q})$ and $d_{3}=\eta b$. Note that rationally $K(\Z,3)$ behaves like $S^{3}$. One checks that that ${}_{E\wedge H\R}E_{4}=0$. 

\bigskip

Let now $f:M\to K(\Z,3)$ be a map from a smooth manifold representing a class $x\in H^{3}(M;\Z)$.
Then, if we take a form $\omega\in Z^{3}(\Omega(M))$ such that $[\omega]=x\otimes \R$,  there exists an equivalence
$$f^{*}E\wedge H\R\cong H\cM(\underline{\R[b,b^{-1}]},\omega b)\ .$$
Thus our procedure reproduces the complex found in Example \ref{may0220}. %\footnote{\uli{Habe $\bb$ durch $b$ ersetzt zur Vereinheitlichung.}}

\end{exa2}

\appendix

\part{Appendix}

\section{Technical facts about graded ring spectra}\label{aapend}
%%%%%%%%%%%%%%%%%%%%%%%%%%%%%%%%%%%%%%%%%%%%%%%%%%%%%%%%%%%%%%%%%%%%%%%%%%%%%5

Here we will prove some technical facts about graded ring spectra {(defined in Section \ref{may0501})}. More precisely we show that
the constructions that are required to define twisted and differential twisted cohomology admit
good {functorial} properties. \\

Let $\UFib $ denote the {total space of the} universal cocartesian fibration over $\Cat_\infty$. 
See \cite[Section 3.3.2]{HTT} for more details.
Roughly speaking objects of $\UFib$ are given  by $\infty$-categories with a chosen {base point}.
However the morphisms in $\UFib$ are not required to preserve 
this object up to isomorphism, but only up to a possible non-invertible morphism.
Thus the  overcategory $(\Cat_\infty)_{*/}$ is a subcategory of $\UFib$
with the same objects but less morphisms.
We will also refer to the category   $\UFib$ as the lax slice of  $\Cat_\infty$
under the point. 

\begin{rem2}
Note that $\Cat_\infty$ and $\UFib$ are neither small nor large. The reason is that $\Cat_\infty$ contains all large categories. Thus we need to work with a third universe, which we call very large. Then $\Cat_\infty$ is very large. But the fibres of the universal fibration are large. Thus $\Cat_\infty$ classifies cocartesian fibrations $X \to S$ between $\infty$-categories $X$ and $S$ {which} itself are allowed to be very large, but with essentially large fibres $X_s \subset X$ for all $s \in S$.

As a convention from now on, we consider all $\infty$-categories to be large, unless we explicitly allow them to be very large.
\end{rem2}

\begin{lemma}\label{lem_eins}
For every (large) $\infty$-category $C$ the assignment $c \mapsto C_{/c}$ defines a
functor $C \to \Cat_\infty$. Picking for every $c \in C$ the object $\id_c \in
C_{/c}$ as the basepoint this functor even refines to a functor $C \to
(\Cat_\infty)_\star$.
\end{lemma}
\begin{proof}
We consider the cocartesian fibration $C^{\Delta[1]} \xrightarrow{ev_1}
C$. Since the slices are large this fibration is classified by a functor $C \to \Cat_\infty$. Thus there is a pullback
diagram
\begin{equation*}
\xymatrix{
C^{\Delta[1]} \ar[r]\ar[d]& (\Cat_\infty)_\star \ar[d] \\
C \ar[r] & \Cat_\infty
}
\end{equation*}
Now we take the canonical section ${\id}_{(-)}: C \to C^{\Delta[1]}$%\fuli{Die Notation mu{\ss} vereinheitlicht werden. Ich schlage $\id(\underline{\hspace{0.2cm}})$ vor. Es kommt auch $(\dots)$ vor irgendwo.} 
and obtain  
the functor $C \to(\Cat_\infty)_\star$ as a composition.
\end{proof}

By definition of $\GCS$ {(see Definition \ref{def1aug13})} there is a pullback diagram 
\begin{equation}\label{diag_pullop}
\xymatrix{
\GCS \ar[rr]\ar[d]^{{U}} && \UFib^{op} \ar[d] \\
\sym \ar[rr]^-{\Fun^{lax}(-,\Sp)} && \Cat_\infty^{op}
} ,
\end{equation}
where $\sym$ denotes the very large category of (large) symmetric monoidal categories and lax symmetric monoidal functors.
{We denote the left vertical functor that associates to every graded ring spectrum its underlying grading symmetric monoidal $\infty$-category by $U$.}
Here the right vertical morphism is just the opposite of the universal cocartesian fibration, hence the universal cartesian fibration.

Note that for an arbitrary (large) $\infty$-category $C$, the presheaf category
$\Fun(C^{op},\Sp)$ becomes a symmetric monoidal $\infty$-category such that
evaluation for every object $c \in C$ is a 
monoidal functor $\Fun(C^{op},\Sp) \to \Sp$. Thus for any object $c \in C$ the pair $(\Fun(C^{op},\Sp), ev_c)$ defines a 
graded ring spectrum. We now want to show that this assignment is even functorial in $(C,c)$.

\begin{lemma}\label{lem_zwei}
There is a functor 
\begin{equation*}
\UFib^{op} \to \GCS  \qquad (C,c) \mapsto \big(\Fun(C^{op},\Sp), ev_c) \end{equation*}
\end{lemma}
\begin{proof}
%\ntho{
%We 
%\ntho{deal with the opposite case and} 
%show that there is a functor $f: \UFib^{op} \to \GCS$ with $f(C,c)=(\Fun(C,\Sp), ev_c)$. The stated 
%result then follows by taking the opposite}. 
Let $g: \Cat_\infty^{op} \to \sym$ be any functor.
Since $\GCS$ fits into the pullback diagram \eqref{diag_pullop} a functor
$f: \UFib^{op} \to \GCS$ {covering $g$} is then uniquely determined (up to a contractible space of choices) by a commuting diagram\begin{equation*}
\xymatrix{
\UFib^{op} \ar[rrrr]^{f'} \ar[d] &&&& \UFib^{op} \ar[d] \\
\Cat_\infty^{op}\ar[rr]^-g && \sym \ar[rr]^-{\Fun^{lax}(-,\Sp)} &&
\Cat_\infty^{op}
}
%\xymatrix{
%\UFib^{op} \ar@{-->}[rr]^-f \ar[d] && \GCS \ar[rr]\ar[d] && \UFib^{op} \ar[d] \\
%\Cat_\infty^{op}\ar[rr]^-g && \sym \ar[rr]^-{\Fun^{lax}(-,\Sp)} &&
%\Cat_\infty^{op}
%}
{\ .}
\end{equation*}
 For our concrete construction we define the functor $g$ by $g(C) {:=} \Fun(C,\Sp)$. Since the right vertical
morphism in the diagram is the universal cocartesian fibration, a functor $f'$
that makes the diagram commute can be written as a functor into the fibration classfied by 
the lower horizontal composition $g'$. Moreover the source of $f'$ is itself a
Grothendieck construction, described by the identity functor. Thus the functor
{$f'$ (and therefore also the original functor $f$)} can equivalently be described as a 
transformation of functors ${\id} \Rightarrow g'$ from $\Cat_\infty$ to itself.
Such a transformation is given as follows: for every $\infty$-category $C$ it is
the assignment 
\begin{equation*}
  C \to \Fun^{lax}(\Fun(C,\Sp), \Sp) \qquad c \mapsto ev_c
\end{equation*}
which can be easily seen to exist functorially in $C$   using the definitions of the functor
categories.  
%\footnote{\uli{Bitte nohc umschreiben, da‚Äö√Ñ√∂‚àö‚Ä†‚àö‚àÇ‚Äö√†√∂¬¨‚à´ die Logik klarer herauskommt.}}
\end{proof}

\newcommand{\calm}{\mathcal{M}}

Let now $\calm$ be an arbitrary, essentially small $\infty$-category, for example the category of smooth manifolds. Then for 
every object $M \in \calm$ we define the category of {presheaves of spectra} on $M$ as 
\begin{equation*}
 \PSh_{\Sp}(M) := \Fun\big((\calm_{/M})^{op}, \Sp\big)
\end{equation*}
which becomes a symmetric monoidal $\infty$-category with the pointwise tensor product. There is a global sections functor 
\begin{equation*}
 \Gamma(\bullet, M): \PSh_{\Sp}(M) \to \Sp
\end{equation*}
which is in our case just evaluation at the identity of $M$. The last two lemmas imply the following functoriality statement.
\begin{corollary}\label{presheafun}
The assignment $M \mapsto \big(\PSh_{\Sp}(M), \Gamma(\bullet, M)\big)$ refines %\marginpar{\uli{Genauer!} \tho{Wie?}} 
to a functor
\begin{equation*}
 \widehat{ \PSh_{\Sp}}: \quad \calm^{op} \to \GCS
\end{equation*}
\end{corollary}
\begin{proof}
{
We apply Lemma \ref{lem_eins} with $\calc = \calm$ and obtain a functor $\calm \to (\Cat_\infty)_\star$. Then we take the opposite and compose this functor with the functor  $\UFib^{op} \to \GCS$ from Lemma \ref{lem_zwei}.
}
\end{proof}

The last corollary establishes the basic example of a presheaf of graded ring spectra. We now want to show how to derive new examples from this one.

\begin{proposition}[Change of grading]\label{lem_drei}
Assume that we have an $\infty$-category $\cald$ together with functors 
\begin{equation*}
F: \cald \to \GCS {\ ,}\quad G: \cald \to \sym{\ ,}
\end{equation*}
and a transformation $\eta$ from $G$ to the composition $\cald \xrightarrow{F}
\GCS {\xrightarrow{U}} \sym$. Then there is a functor
\begin{equation*}
  \eta^*F: \quad  \cald \to \GCS \qquad d \mapsto (G(d), G(d) \xrightarrow{\eta} U{(}F(d){)} \xrightarrow{F(d)} \Sp)\ .
\end{equation*}
 %\nuli{the underlying symmetric monoidal category of a graded ring spectrum.}
\end{proposition}
\begin{proof}
We write the transformation $\eta$ as a morphism $\cald \times \Delta[1] \to
\sym$. Then we have a commuting square
\begin{equation*}
\xymatrix{
\cald \ar[r]^F \ar[d]^{(-,1)} & \GCS \ar[d] \\
\cald \times \Delta[1]  \ar[r]^\eta & \sym
}
 \end{equation*}
 %\tfoot{Wuerde den diagonalen Pfeil nicht einzeichnen. Das ist verwirrend im Lesefluss}
We want to show that there is a lift $L: \cald \times \Delta[1] \to \GCS$ in the
diagram with the property that for every $d \in \cald$
the edge $L({\id}_d,0 \to 1)$ is cartesian. Then the desired functor is given by
the restriction of $L$ to $\cald \times {0}$. 

Now we want to show that such a lift $L$ exists. This is equivalent to a cartesian lift in
the diagram
 \begin{equation*}
\xymatrix{
\Delta[0] \ar[r] \ar[d] & \Fun(\cald, \GCS) \ar[d] \\
\Delta[1] \ar[r]^-\eta & \Fun(\cald,\sym)
}
\end{equation*}
But the right vertical morphism  is a cartesian
fibration since $\GCS \to \sym$ is one by definition and exponentials of
cartesian fibrations are again cartesian fibrations as shown in
\cite[Proposition 3.1.2.1]{HTT}. Thus there is an essentially  unique cartesian lift as desired.
\end{proof}

\begin{rem2}
Let $\calm$ be an $\infty$-category. From Corollary \ref{presheafun} we have an object $\PSh_{\Sp} \in \Fun(\calm^{op},\sym)$
together with a lift $\widehat{ \PSh_{\Sp}} \in \Fun(\calm^{op},\GCS)$. Together with the change of grading statement we can obtain new examples of functors $$\eta^*\widehat{ \PSh_{\Sp} }  \in \Fun(\calm^{op},\GCS)$$  {from}  a natural transformation $\eta: G \to \PSh_{\Sp}$ for any functor $G \in \Fun(\calm^{op},\sym)$. All examples in this paper will arise in {this} way. This is not a coincidence, since one can show that all examples of functors in $\Fun(\calm^{op},\GCS)$ {can be constructed from $\PSh_{\Sp}$ using the change of grading statement}. In other words the object
$\widehat{ \PSh_{\Sp} }$ over $\PSh_{\Sp}$ {is the universal presheaf of graded ring spectra over $\cM${.} We refrain from giving the proof here since it is {technically} involved {and we do not need the statement}. }
\end{rem2}

\begin{lemma}\label{cor_ftilde}
We assume {that} $I$  is a small $\infty$-category and let $v \in I^{\triangleleft}$ be the cone point. Then every functor 
$F: I^{\triangleleft} \to \GCS$ induces 
a functor $\tilde F  : I^{\triangleleft} \to  \GCS_{U(F(v))}
%\footnote{\uli{Was bedeutet der Subskript.}}
$. 
 Hence for every $c \in U(F(v))
 %\footnote{\uli{undefinde symbol}}
 $ {(see \eqref{diag_pullop} for $U$)} we obtain a functor $\tilde{F}^c: I^\triangleleft \to \Sp$. 
\end{lemma}
\begin{proof}
Apply the change of grading Proposition \ref{lem_drei} to the functor $F$,
the constant functor $G: I^\triangleleft \to \sym$ with image
 $ U(F(v)) \in \sym$, and the obvious natural transformation $G \to U{\circ} F$. Explicitly this transformation is given as the composition
$$\eta: \quad I^\triangleleft \times \Delta[1] \to I^\triangleleft \to \GCS\to \sym\ .$$
where the first functor is the identity on $I^\triangleleft \times \{1\}$ and carries $I^\triangleleft \times \{0\}$ to the cone point.
\end{proof}
%
%\begin{corollary}\label{cor_mai11}
%Let $F: \cald \to \GCS$ be a functor and $d_0 \in \cald$ be an arbitrary object. We denote the grading category $UF(d)$ by $\calc_0$. Then we obtain an induced functor
%$$
%F_{d_0}: \quad \cald_{d_0/} \to \GCS_{\calc_0} \qquad (d_0 \to d) \mapsto (\calc_0 \to UF(d) \xrightarrow{F(d)} \Sp)    	
%$$
%In particular for every object $c \in \calc_0$ evaluation at $c$ yields  functor 
%$$
%F_{d_0}^{c_0}: \quad \cald_{d_0/}  \to \Sp
%$$ 
%\end{corollary}
%\begin{proof}
%Apply the change of grading Proposition \ref{lem_drei} to the functor,
%$ 
%\cald_{d_0/} \to \calc \xrightarrow{F} \GCS
%$ the constant functor $\cald_{d_0/} \to \sym$ which maps to $\calc_0 \in \sym$ and the obious natural transformation.
%\end{proof}
%
%\begin{example}\label{ex_mai10}
%\begin{enumerate}
%\item
%Assume $\cald$ is an $\infty$-category with initial object $\emptyset$. Then every functor $F: \cald \to \GCS$ induces 
%a functor $\tilde F  : \cald \to  \GCS_{\calc_0}$ with $\calc_0 := UF(\emptyset)$. Hence for every $c \in C_0$ we obtain a functor $\tilde{F}^c: \cald \to \GCS$. 
%\item
%Let $X$ be a presheaf of graded ring spectra on a category $\calm$, i.e. a functor
%$\calm^{op} \to \GCS$. Then we obtain induced functors:
%$$ \widetilde X_{M} : (\calm/M)^{op} \to \GCS_{\calc_M}$$
%and by evaluation
%$$ \widetilde X^c_{M} : (\calm/M)^{op} \to \Sp$$
%for $M \in \calm$ and $c \in \calc_M := UX(M)$.
%\end{enumerate}
%\end{example}
%

Now we want to give a short discussion of limits of graded ring spectra. This will be important when imposing descent conditions for presheaves of graded ring spectra. The following result is essentially due to Lurie \cite[Section 4.3.1]{HTT} and holds more {generally}, but we state it explicitly in the form {needed below}.
\begin{proposition}\phantomsection\label{limits}
\begin{enumerate}
\item
The  $\infty$-category $\GCS$ admits  limits
%\tfoot{habe jetzt colimits mal weggelassen, sonst brauchen wir Pr\"asentierbarkeit und muessen uns auf kleine Kategorien beschr‚Äö√Ñ√∂‚àö‚Ä†‚àö‚àÇ¬¨¬®‚àö√ºnken} %and colimits 
and the projection $\GCS \to \sym$ preserves limits.
\item
A diagram $F: I^{\triangleleft} \to \GCS$ is a limit diagram   if and only if  the following conditions are satisifed: 
\begin{itemize}
\item The induced diagram $I^{\triangleleft} \to \GCS \to \sym$ is a limit diagram
\item For every object $c \in U(F({v}))$ the resulting diagram
$\tilde{F}^c:  I^{\triangleleft} \to \Sp$ (see Lemma \ref{cor_ftilde}) is a limit diagram
\end{itemize}
\end{enumerate}
\end{proposition}
\begin{proof}
1.) 
%We first observe that $\GCS \to \sym$ is not only a cartesian fibration, but also a cocartesian fibration. This follows together with \cite[Proposition 5.5.3.3.]{HTT} from the observation that the functor $\GCS_C \to \GCS_{C'}$ is right adjoint for every morphism $C' \to C$ in $\sym$. To see that these functors are right adjoint note that the category $\GCS_C$ and $\GCS_{C'}$ are presentable for small $C$ and $C'$ according to \cite{HA} and limits and sifted colimits are computed pointwise. Thus the functor clearly preserves limits and filtered colimits, hence be the adjoint functor theorem it is right adjoint.  
We first observe that the functor  $U: \GCS \to \sym$ admits a left adjoint given by the functor $\sym \to \GCS$ that maps every $C$ to the initial object of the fibre $\GCS_C = \Fun^\otimes_{\mathrm{lax}}(C,\Sp)$, see \cite[Proposition 2.4.4.9]{HTT} for the existence of this functor and \cite[Proposition 5.2.4.3]{HTT} for the fact that it is left adjoint to $U$.
%\footnote{\uli{Aus Lurie lese ich nur die Existenz eines solchen Funktors geraus. Warum der linksadjungiert in der gew‚Äö√Ñ√∂‚àö‚Ä†‚àö‚àÇ¬¨¬®‚Äö√†¬¥nschten weise ist???}}
This shows that $U$ preserves all limits. 

We now want to show that $\GCS$ has all limits. 
It follows from \cite[Remark 4.3.1.5]{HTT}
%\footnote{\uli{Die gibts nicht!}}
 that a diagram $F: I^\triangleleft \to \GCS$ is a limit diagram if $I^\triangleleft \to \GCS \to \sym$ is a limit diagram and  $F$ is a $U$-limit diagram (this is a relative limit as discussed in \cite[section 4.3.1]{HTT}). Thus using the fact that $\sym$ has all limits we conclude that we can extend any diagram $I \to \GCS$ to a square
$$
\xymatrix{
I \ar[r] \ar[d]& \GCS \ar[d]\\
I^\triangleleft \ar[r]& \sym
}
$$
where the lower horizontal morphism is a limit diagram. Then we need to show that we can find a lift $I^\triangleleft \to \GCS$ in this diagram that is a $U$-limit. But the existence of such a lift is shown in \cite[Corollary 4.3.1.11]{HTT}
 under the condition that 
 \begin{itemize}
 \item The fibres $\GCS_C = \Fun^\otimes_{\mathrm{lax}}(C,\Sp) $ have all limits {for all $C\in \sym$}.
 \item The functor $\GCS_C \to \GCS_{C'}$ associated to a morphism $C' \to C$ in $\sym$ preserves all limits.
 \end{itemize}
These conditions are satisfied as shown in {\cite[Section 3.2.2]{HA}}.

\bigskip 
2) Assume $F: I^\triangleleft \to \GCS$ is given. The above discussion shows that  the assertion that $F$ is a limit {diagram} is equivalent to
\begin{itemize}
\item The induced diagram $I^{\triangleleft} \to \GCS \to \sym$ is a limit diagram
\item The diagram $F$ is a $U$-limit diagram.
\end{itemize}
We now want to show that the second condition ($F$ is a $U$-limit diagram) is equivalent to the second condition of the proposition. Therefore we claim that it follows from \cite[Proposition 4.3.1.9.]{HTT} that this condition is equivalent to the condition:
\begin{itemize}
\item The diagram 
$\tilde{F}:  I^{\triangleleft} \to \GCS_{U(F(v))}$ (see Lemma \ref{cor_ftilde}) is a limit diagram.
\end{itemize}
To see this note that $\tilde{F}$ comes by construction with a natural tranformation $\tilde{F} \to F$ which satisifies the assumption of \cite[Proposition 4.3.1.9.]{HTT}. Now we only need to observe that limits in 
$$\GCS_{U(F(v))} = \Fun^\otimes_{\mathrm{lax}}(U(F(v)),\Sp)$$
 are computed pointwise, i.e. they can be detected by evaluation at each $c \in U(F(v))$%\fuli{der e.i. Nachsatz kling nach Slang und macht eigentlich keinen Sinn, jedenfalls f\"ur mich.\tho{Warum? Soll ich das noch genauer ausf√ºhren?}}.
\end{proof}

%\begin{corollary}
%Let $X: \calm^{op} \to \GCS$ be a presheaf of graded ring spectra on a Grothendieck site $\calm$. Then $X$ is a sheaf if and only if $UX: \calm^{op} \to \sym$ is a sheaf and for every $M \in \calm$ and $c \in UX(M)$ the induced presheaf
%$$\widetilde{X}_M^c:  (\calm/M)^{op} \to \Sp \hspace{2cm} \text{(see Example \ref{ex_mai10})}$$
%is a sheaf of spectra.
%\end{corollary}
%\begin{proof}
%The presheaf $X$ is a sheaf if for every $N \in \Mf$ and open covering $(U_i)_{i \in I}$ tbc
%\end{proof}

Let $\calm$ be a Grothendieck site, i.e. an $\infty$-category equipped with a Grothendieck topology on its homotopy category $\Ho({\cM})$. 
In Corollary \ref{presheafun} we constructed a presheaf of graded ring spectra $\widehat{\PSh_{\Sp}}$. Now we want to discuss a variant and consider {for $M\in \calm$} the full subcategory 
$$\Sh_{\Sp}(M) \subset \PSh_{Sp}(M)$$
of sheaves on the overcategory $\calm/M$ {which also becomes a site when} equipped with the induced Grothendieck topology.
This defines a presheaf of symmetric monoidal categories $\Sh_{\Sp} \in \PSh_{\sym}(\calm)$. Then the change of grading statement immediately allows to refine this to a presheaf of graded ring spectra 
$$\widehat{\Sh_{\Sp}} \in \PSh_{\GCS}(\calm)$$
We now want to show that this assignment is a sheaf of graded ring spectra. First note that $\PSh_{\Sp}$ and $\Sh_{\Sp}$ both form sheaves of symmetric monoidal categories as opposed to presheaves. 

\begin{proposition}\label{changesheaf} Let $\calm$ be a Grothendieck site. 
\begin{enumerate} 
\item
The assignment 
$M \mapsto \big(\Sh_{\Sp}(M), \Gamma(\bullet, M)\big)$ refines to a sheaf of graded ring spectra  
on $\calm$:
$$\widehat{\Sh_{\Sp}} \in \Sh_{\GCS}(\calm).$$
\item
Let $F: \calm^{op} \to \GCS$ be a sheaf of graded ring spectra, $G: \calm^{op} \to \sym$ be a sheaf of symmetric monoidal categories{,} and $\eta: G \to U{\circ }F$ {be} a natural transformation. Then the functor
$\eta^*F: \calm^{op} \to \GCS$ obtained by change of grading is also a sheaf of graded ring spectra. 
\item
The functor $$\Sh_{\GCS}(\calm) \to \Sh_{\sym}(\calm)$$ induced by $U$ is a cartesian fibration.
\end{enumerate}
\end{proposition}
\begin{proof}
In order to show (1) let $M \in \calm$ be an object and $\{U_i \to M\}_{i \in I}$ be a covering family. Forming the \u{C}ech nerve and applying the functor $\widehat{\Sh_{\Sp}}$ we obtain the augmented cosimplicial object
$$ \widehat{\Sh_{\Sp}}(U_\bullet): \quad  \Delta^\triangleleft \to \GCS.$$
We have to show that this augmented cosimplicial object is a limit diagram. According to Proposition \ref{limits} we have to check that 
\begin{itemize}
\item The induced diagram $\Sh_{\Sp}(U_\bullet):  \Delta^{\triangleleft} \to \sym$ is a limit diagram\uli{.}
\item For every object $c \in \Sh(M)$ the resulting diagram
$c(U_\bullet): I^{\triangleleft} \to \Sp$  is a limit diagram\uli{.}
\end{itemize}
The first condition is precisely the sheaf condition for  $\Sh_{\Sp}: \calm^{op} \to \sym$ and therefore true as remarked before the Proposition. The second condition is precisely the sheaf condition for $c: (\calm/M)^{op} \to \Sp$ and is also satisified since $c$ is a sheaf. \\

Now we want to show (2). The proof works similar to the proof of (1). One has to show that a diagram 
$$ \eta^*F(U_\bullet): \quad  \Delta^\triangleleft \to \GCS.$$
is a limit {diagram}. In this case the two diagrams that have to be identified as limits according to Proposition \ref{limits} are
$$G(U_\bullet):  \Delta^\triangleleft \to \sym \qquad \text{and} \qquad \tilde F^{d}(U_\bullet): \Delta^\triangleleft \to \Sp$$
 for $c \in G(M)$ and 
$d = \eta_M(c) \in U(F(M))$ {(see \ref{cor_ftilde} for the notation $ \tilde F^{d}$)}. The first follows from the fact that $G$ is a sheaf and the second from the fact that $F$ was a sheaf. \\

Finally for (3) first note that the induced functor 
$$U_*: \PSh_{\GCS}(\calm) \to \PSh_{\sym}(\calm)$$
 restricts to a functor of sheaves 
$$\hat{U}_*:\Sh_{\GCS}(\calm) \to \Sh_{\sym}(\calm)$$
as stated since $U$ preserves limits, as shown in Proposition \ref{limits}. The next step is to observe that the inclusion  
$i: \Sh_{\GCS}(\calm) \to  \PSh_{\GCS}(\calm)$ has the following property:
\begin{quote}
Let $e: \Delta[1] \to \Sh_{\GCS}(\calm)$ be an edge such that $i(e)$ is $U_*$-cartesian. Then $e$ is  $\hat U_*$-cartesian.
\end{quote}
This basically follows since the inclusion of sheaves into presheaves is full. To show that $\hat U_*$ is a cartesian fibration we have to show that $\hat U_*$ is an inner fibration and every edge in $\Sh_{\sym}(\calm)$ with a {given} lift of the target has a $\hat U_*$-cartesian lift to $\Sh_{\GCS}(\calm)$. We observe that  $\hat U_*$ is an inner fibration   since the inlcusion of presheaves into sheaves is full and $U_*$ is an inner fibration. 
%\fuli{Ich verstehe die Aussage des folgenden Satzes nicht (sogar sprachlich) Sehe auch nicht, wo (2) hilft.}
For the lifting property we use statement (2)  of the Proposition to lift to an edge $e$ in $\Sh_{\GCS}(\calm)$ whose image $i(e)$ is $U_*$ cartesian. Then we can conclude that $e$ is $\hat U_*$-cartesian as remarked above.

\end{proof}
\section{Locally constant sheaves}\label{bapend}
%%%%%%%%%%%%%%%%%%%%%%%%%%%%%%%%%%%%%%%%%%%%%%%%%%%%%%%

Let $\calc$ be a presentable $\infty$-category and $M$ be a smooth manifold. Recall that {a $\calc$-valued sheaf on a manifold }$M$ (more precisely the site $\Mf/M$) is called constant when it is {equivalent to  the restriction   $(\hat C_{X})_{|M}$ of a constant sheaf $\hat C_{X}\in \Sh_{\cC}$ for {an} object $X \in \calc$. Also recall that we have an equivalence of $\hat C_{X}$ with the sheaf
$$ \underline{X}\in \Sh_{\calc}  \ , \qquad   M \mapsto X^{\Sing (M)} $$
{(see Remark \ref{8dedkjehdke} for the notation $\Sing(M)$)} as shown in \cite[Lemma 6.7]{bg}, see Proposition \ref{propju22}.} A locally constant sheaf was defined in Section \ref{ulimar2201} as follows:

\begin{definition}
A sheaf $\cF\in \Sh_{\cC}(M)$ is called locally constant if every point $m\in M$ has a neighbourhood $U\subseteq M$ such that $\cF_{|U}\in \Sh_{\cC}(U)$ is constant.
We write
$$\Sh^{\loc}_{\cC}(M)\subseteq \Sh_{\cC}(M)$$ for the full subcategory of locally constant sheaves.
\end{definition}

We want to prove the following statement, which is a {$\infty$-categorical} generalization of the classification of local systems using covering theory:

\begin{thm}\label{thmrep}
There is an equivalence of $\infty$-categories $\Sh^\loc_\calc(M) \simeq \Fun(\Sing {(}M{)}, \calc) $ which is natural in $M$.
\end{thm}
\begin{rem2}\label{8dedkjehdke}
\begin{itemize}
\item
The $\infty$-groupoid $\Sing (M)$ is the simplicial version of the fundamental-$\infty$-groupoid of the space $M$. Its homotopy category is the ordinary fundamental groupoid $\Ho(\Sing( M)) \simeq \Pi_1(M)$. Therefore if $\calc$ is a 1-category, Theorem {\ref{thmrep}} reduces to the well-known classification of ordinary local systems.
%\item \fuli{Wollen wir das drin lassen? Mir ist nicht klar, worauf ich bei der Definition der Site achten mu{\ss}?}
%This classifcation is also true for \tho{more general} topological spaces $M$, but then one has to be careful with the definition of the category on which the sheaves are defined. But we will not need this case  here.
%%\tho{\footnote{\tho{Stimmt das \"uberhaupt, man braucht wahrscheinlich noch Parakompakt oder so, damit man gute offene \"Uberdeckungen bekommt.}}}
\item
If $\calc$ is symmetric monoidal{,} then the equivalence is compatible with the induced monoidal structures.
\end{itemize}
\end{rem2}

We now want to prove Theorem \ref{thmrep} in several steps.
\begin{lemma}\label{lemmai17}
\begin{itemize}
\item
The presheaves
$$\Sh^\loc_\calc(-): \Mf^{op} \to \Cat_\infty \quad \text{and}  \quad \Fun(\Sing (-),\calc):  \Mf^{op} \to \Cat_\infty$$
are sheaves of $\infty$-categories on the site of smooth manifolds.
\item
There are functors $F_M: \Fun(\Sing (M),\calc) \to \Sh^\loc_\calc(M)$  which are natural in $M$ such that 
the diagram 
\begin{equation}\label{dreimai17}
\xymatrix{
 & \calc \ar[dl]_{\mathrm{const}}\ar[rd]^{\underline{(-)}} & \\
 \Fun(\Sing (M), \calc) \ar[rr]^-{F_M} & & \Sh^\loc_\calc(M) 
}
\end{equation}
commutes for every $M$.
\end{itemize}
\end{lemma}
\begin{proof}
To see that $\Sh^\loc_\calc(-)$ is a sheaf it suffices to observe that $\Sh_\calc(-)$ is a sheaf and the condition of being locally constant is a local condition. For the second statement we consider the category $\calc \in \Cat_\infty$ as an object in the (very large) presentable category of (large) $\infty$-categories. Then $\calc$ induces the constant sheaf {$\hat C_{\calc}\in \Sh_{\Cat_{\infty}}$.}
 {As we have remarked above (for general target categories)
 it is equivalent to the sheaf $\Fun(\Sing(-),\calc)\in \Sh_{\Cat_{\infty}}$. 
In particular, $\Fun(\Sing(-),\calc)$ is constant.}
  In order to construct a transformation
$$
\Fun(\Sing(-) , \calc) \to \Sh^{\loc}_\calc(-)
$$
{it therefore suffices to construct a  transformation of $\infty$-category valued {sheaves}
$$\hat C_{\cC}\to  \Sh^{\loc}_\calc\ .$$}
%we want to make use of the obervation that $\Fun(\Sing(-) , \calc)$ is constant. Since $ \Sh^{loc}_\calc(-)$ is a sheaf such a functor is uniquely determined by a transformation of $\infty$-category valued presheaves from the constant presheaf with value $\calc$ to $\Sh^{\loc}_\calc(-)$. 
Such a transformation is clearly given by the transformation of presheaves
$$ {
\calc\cong C_{\calc}(M)} \to \Sh^{\loc}_\calc(M) \qquad C \mapsto \underline{C}.
$$
for every $M$.
\end{proof}

\begin{rem2}
The proof of {Lemma \ref{lemmai17}} shows that the sheaf $\Fun(\Sing(-),\calc)$ is constant as a sheaf of categories. Thus the equivalence $\Fun(\Sing(M),\calc) \cong \Sh^\loc_\calc(M)$ (which still has to be proven) shows that also $\Sh^{{\loc}}_\calc(-)$ is constant as a sheaf of categories. This underlines the general philosphy that locally constant objects are sections in a constant sheaf. 
\end{rem2}
As a next step we need a technical lemma about limits in $\Cat_\infty$.

\begin{lemma}\label{lem2mai17}
Assume  we have two cosimplicial $\infty$-categories $C_\bullet$ and $D_\bullet$ (i.e. functors $\Delta \to \Cat_\infty$) together with a cosimplicial functor 
$F_\bullet: C_\bullet \to D_\bullet$ such that $F_i$ is fully faithful for each $i \in \Delta$. 
{Then we have the following assertions:}
\begin{enumerate}
\item The functor $\lim_{\Delta}\  C_\bullet \to \lim_{{\Delta}}\ D_\bullet$ is fully faithful
\item
The diagram
$$
\xymatrix{
\lim_{{\Delta}} C_\bullet \ar[r]\ar[d] & \lim_{{\Delta}} D_\bullet \ar[d] \\
C_0\ar[r] & D_0 
}
$$
is a pullback in the $\infty$-category $\Cat_\infty$. 
\end{enumerate}
\end{lemma}
\begin{proof}
The first part is true for any limit, not just limits indexed by $\Delta$. {This} follows from the observation that 
a functor $F: C \to D$ is fully faithful precisely if the diagram
$$ 
\xymatrix{
C^{\Delta[1]} \ar[r]\ar[d] & D^{\Delta[1]} \ar[d] \\
C \times C \ar[r] & D \times D 
}$$ {is a pullback.}
If $F$ is a limit of functors $F_i$ then this diagram is the limit of the corresponding diagrams for the functors $F_i$ since all operations commute with limits. Therefore if all the $F_i$-diagrams are pullbacks it follows that also the limit diagram is a pullback.

To show the pullback property we want to show that the natural map 
$$\varphi: \lim_{{\Delta}} C_\bullet \to C_0 \times_{D_0} \lim_{{\Delta}} D_\bullet$$
is an equivalence. Therefore we first note that the functor $C_0 \times_{D_0} \lim_{{\Delta}} D_\bullet \to \lim_{{\Delta}} D_\bullet$ is fully faithful as the pullback of a fully faithful functor. Also from the first assertion we know that $\lim_{{\Delta}} C_\bullet \to \lim_{{\Delta}} D_\bullet$ is fully faithful. Thus we really compare two full subcategories of $\lim_{{\Delta}} D_\bullet$ and it suffices to compare their essential images. 
{This means that}
 it suffices  to check that the fibre of 
$$ \lim_{{\Delta}} F_\bullet: \lim_{{\Delta}} C_\bullet \to \lim_{{\Delta}} D_\bullet $$
{at every object  {of $\lim_{\Delta} D_{\bullet}$}}
is equivalent to the fibre of 
$$ F_0:  C_0 \to D_0.$$
The fibre of $\lim_{{\Delta}}  F_\bullet$ is given by the limit of the fibres of the $F_i's$, hence a limit of categories which are either a singleton or empty (since the $F_i$ are fully faithful). This limit is either empty or a singleton depending on whether the corresponding diagram consists entirely of empty categories or singletons. In particular it is empty if and only if the  fibre of $F_0: C_0 \to D_0$ is empty. This implies the second assertion of the lemma.\end{proof}

\begin{corollary}\label{cor3mai17}
{We 
assume} $F$ and $G$ are sheaves of $\infty$-categories on {some} Grothendieck site $\cM$ and $\eta: F \to G$ is a morphism
of sheaves such that $\eta_M: F(M) \to G(M)$ is fully faithful for every element $M \in \cM$. Then for every covering family $\{U_i \to M\}_{i \in I}$ the diagram
$$
\xymatrix{
F(M) \ar[r] \ar[d]& G(M)  \ar[d]\\
\prod_{i \in I} F(U_i) \ar[r] & \prod_{i \in I} G(U_i) 
}$$
is a pullback diagram in $\Cat_\infty$. In particular an object $x \in G(M)$ lies in the essential image of $\eta_M: F(M) \to G(M)$ precisely if there exist a covering family $\{U_i \to M\}_{{i\in I}}$ such that {for every $i\in I$} the {restriction} $x|_{U_i} \in G(U_i)$ lies in the essential image of $\eta_{U_i}: F(U_i) \to G(U_i)$. 
\end{corollary}

\begin{proof}[Proof of Theorem \ref{thmrep}]
We want to show that the functors  $F_M: \Fun(\Sing (M),\calc) \to \Sh^\loc_\calc(M)$ constructed in Lemma \ref{lemmai17} are equivalences for every $M$. We first assume that $M$ is contractible. Then using Diagram \eqref{dreimai17} we see that the functor $F_M$ is equivalent to the functor
$$ \calc \to \Sh^{\loc}_\calc(M) \qquad C \mapsto \underline{C}. $$
This functor is clearly fully faithful since one easily computes 
\begin{eqnarray*}
\Map_{\Sh^{\loc}_\calc(M)}\big(\underline{C}, \underline{D}\big)  \simeq \Map_{\Sh_\calc(M)}\big(\underline{C}, \underline{D}\big) 
&\simeq &\Map_{\PSh_\calc(M)}\big(\mathrm{const}_C, \underline{D}\big) \\
&\simeq &\Map_{\calc}(C, D^{\Sing M}) \\ &\simeq& \Map_{\calc}(C, D) \ .
\end{eqnarray*}
Thus we know that the functor $F_M$ is fully faithful for contractible $M$. 

For arbitrary $M$ we can choose a good open cover $\{U_i\}$, i.e. an open cover that has the property that all finite intersections $U_{i_1,...,i_k} := U_{i_1} \cap ... \cap U_{i_k}$ are contractible. Then using the fact that $\Fun(\Sing (-), \calc)$ and $\Sh^{\loc}_\calc(-)$ are sheaves as established in 
\ref{lemmai17} we can write the functor $F_M$ as the limit of functors 
$$F_{U_{i_1,...,i_k}}:  \quad \Fun(\Sing (U_{i_1,...,i_k}) ,\calc) \to \Sh^\loc_\calc(U_{i_1,...,i_k})$$
Using that all these functors are fully faithful since $U_{i_1,...,i_k}$ is contractible and the first claim of lemma \ref{lem2mai17} we can conclude that $F_M$ is also fully faithful for any $M$. 

Finally we want to show that $F_M: \Fun(\Sing (M),\calc) \to \Sh^\loc_\calc(M)$ is essentially surjective. Therefore let 
$\calf \in  \Sh^\loc_\calc(M)$.  Since $\calf$ is locally constant we find an open cover $\{U_i\}$ of $M$ such that 
$\calf|_{U_i}$ is constant. But constant sheaves clearly lie in the essential image of the functors $F_{U_i}: \Fun(\Sing (U_i),\calc) \to \Sh^\loc_\calc(U_i)$ which can be seen from Diagram \eqref{dreimai17}. But then {Corollary} \ref{cor3mai17} implies that $\calf$ lies in the essential image of $F_M$.

\end{proof}

\begin{corollary}
Every locally constant sheaf $\calf \in \Sh^{\loc}_\calc(M)$ is homotopy invariant, i.e. for a  {morphism $N \xrightarrow{f} N' \to M$ of manifolds over $M$}    with $f$  a homotopy equivalence the induced morphism $\calf(N') \to \calf(N)$ is an equivalence in $\calc$.
\end{corollary}
\begin{proof}
{It suffices} to show that for a locally constant sheaf $\calf$ on $M$ and a homotopy equivalence $f: N \to M$ the global sections $\calf(M)$ and $f^*\calf(N)$ are equivalent in $\calc$. Thus we need to show that the diagram
$$
\xymatrix{
\Sh^\loc_\calc(M) \ar[rr]^{f^*}\ar[dr]_{{\ev_M}} && \Sh^{\loc}_\calc(N) \ar[dl] ^{{\ev_N}}\\
& \calc &
}
$$
commutes. By the classification of locally constant sheaves as in Theorem \ref{thmrep} we know that $f^*$ is an equivalence. 

We now use that the global sections functors are right adjoint to the functors $\calc \to \Sh^\loc_\calc(M)$ and 
$ \calc \to \Sh^\loc_\calc(N)$ which are given by the inclusion of constant sheaves. Thus the commutativity of the diagram above is equivalent to the commutativity of the diagram of right adjoints. The right adjoint of $f^*$ is its inverse. But then we replace this inverse by $f^*$ we arrive at the point where we have to check that the diagram 
$$
\xymatrix{
\Sh^\loc_\calc(M) \ar[rr]^{f^*} && \Sh^{\loc}_\calc(N)  \\
& \calc \ar[lu]^{{(\hat C_{(-)})_{|M}}}\ar[ru]_{{(\hat C_{(-)})_{|N}}} &
}
$$
commutes, which is obvious.
\end{proof}

\bibliographystyle{alpha}
\bibliography{zweite}
\end{document}